\documentclass[10pt,a4paper]{article}
\usepackage[utf8]{inputenc}
\usepackage[english]{babel}
\usepackage[T1]{fontenc}   
\usepackage[utf8]{inputenc}
\usepackage{a4}
\usepackage{setspace}
\usepackage{mathrsfs}
\usepackage[numbers]{natbib}
\usepackage{epsfig}
\usepackage{amsmath}
\usepackage{amssymb}

\usepackage{bm}
\usepackage{bbm}
\usepackage{dsfont}
\usepackage{xcolor}
\usepackage{comment}
\usepackage{authblk}
\usepackage{amsfonts}
\usepackage{stmaryrd}
\usepackage{epsf}
\usepackage{amsfonts,bbm}
\usepackage{amsthm}
\usepackage{amscd}
\usepackage{amsopn}
\usepackage{tabularx}
\usepackage{multirow}
\usepackage{pifont}
\usepackage{setspace}
\usepackage{mathrsfs}
\usepackage{natbib}
\usepackage{epsfig}
\usepackage{bm}
\usepackage{bbm}
\usepackage{amssymb}
\usepackage{pifont}
\usepackage[utf8]{inputenc}
\usepackage[T1]{fontenc}
\usepackage{lmodern}
\usepackage{latexsym}
\usepackage{bbm}
\usepackage{amsthm}
\usepackage{relsize}
\usepackage{amsmath}
\usepackage{amssymb}
\usepackage{bigints}
\usepackage{hyperref}
\usepackage[mathcal]{eucal}
\usepackage{multicol}
\usepackage{ulem}
\usepackage{dsfont}
\usepackage{amsfonts,amsmath,amssymb}
\usepackage{fancybox}
\usepackage{amsthm}
\usepackage{amsmath}
\usepackage{amssymb}
\usepackage{mathrsfs}
\usepackage{latexsym}

\usepackage{graphics}  
\usepackage{colortbl}
\usepackage{fancyhdr}
\usepackage{makeidx}
\usepackage{enumitem}
\usepackage{lettrine}
\usepackage{hyperref}
\usepackage{amssymb,wasysym}
\usepackage{stmaryrd}
\usepackage{epsf}
\usepackage{amsfonts,bbm}
\usepackage{amsthm}
\usepackage{amscd}
\usepackage{amsopn}
\usepackage{leftindex}
\usepackage{tabularx}
\usepackage{multirow}
\usepackage{amsfonts}
\newtheorem{theoreme}{Theorem}[section]
\newtheorem{lemme}{Lemma}[section]
\newtheorem{remarque}{Remark}[section]
\newtheorem{proposition}{Proposition}[section]

\newtheorem{definition}{Definition}[section]

\newcommand{\R}{\mathbb{R}}

\newcommand{\N}{\mathbb{N}}

\newcommand{\norm}[1]{\left\Vert#1\right\Vert}

\renewcommand{\t}{\theta}

\newcommand{\bc}{\begin{center}}
	\newcommand{\ec}{\end{center}}
\newcommand{\benu}{\begin{enumerate}}
	\newcommand{\eenu}{\end{enumerate}}

\newcommand{\real}{\mathbb R}

\title{\bf {On Conditional least squares estimation for the $\mathit{AD}(1,n)$ model}}
\author[1]{Mohamed Ben Alaya}
\author[1,2]{Houssem Dahbi\thanks{  \sffamily Corresponding author}}
\author[2]{Hamdi Fathallah}
\affil[1]{Universit\'e de Rouen Normandie, Laboratoire de Math\'ematiques Rapha\"el Salem, Avenue de l'universit\'e, BP.12 F76801 Saint-\' Etienne-du-Rouvray, (France)}
\affil[2]{Universit\'e de Sousse, Laboratoire LAMMDA,\'Ecole Sup\'erieure des Sciences et de Technologie de Hammam Sousse, Rue Lamine El Abbessi 4011 Hammam Sousse, (Tunisie)}	
\DeclareUnicodeCharacter{2212}{-}
\begin{document}
	\maketitle

	\begin{abstract}
		This paper deals with the problem of global parameter estimation of $\mathit{AD}(1,n)$ where $n$ is a positive integer which is a subclass of affine diffusions introduced by Duffie, Filipovic, and
		Schachermayer in \cite{Duffie}. In general affine models are applied to the pricing of bond and stock options, which is illustrated for the Vasicek, Cox-Ingersoll-Ross and Heston models. Our main results are about the conditional least squares estimation of $\mathit{AD}(1,n)$ drift parameters based on two types of observations : continuous time observations and discrete time observations with high frequency and infinite horizon. Then, for each case, we study the asymptotic properties according to ergodic and non-ergodic cases. This paper introduces as well some moment results relative to the $\mathit{AD}(1,n)$ model.  
		 
		\end{abstract}
	\makeatletter{\renewcommand*{\@makefnmark}{}
		\footnotetext{E-mail adresses:\\
			\url{mohamed.ben-alaya@univ-rouen.fr}\\
			\url{houssem.dahbi@univ-rouen.fr}, \url{houssemdahbi@essths.u-sousse.tn}\\
			\url{hamdi.fathallah@essths.u-sousse.tn}}\makeatother}				

\section{Introduction}
The set of affine processes contains a large class of
important Markov processes such as continuous state branching processes. One of the most important Markov affine models is that of an affine diffusion model which is well characterized on the state space $\R_+^m\times\R^n$ with $m$ and $n$ are non-negative integers (see, e.g., \cite{Duffie}). In this paper, we study the subclass $\mathit{AD}(1,n)$ models introduced in \cite{Dahbi} where we consider the process $Z=(Y,X)^{\mathsf{T}}$ in the $d$-dimensional canonical state space $\mathcal{D}:=[0,\infty)\times (-\infty,\infty)^n$, for
$d=n+1$, strong solution of the following stochastic differential equation (SDE)

\begin{equation}
	\begin{cases}
		\mathrm{d}Y_t=(a-b Y_t) \mathrm{~d}t+\rho_{11}\sqrt{Y_t}\mathrm{~d}B_t^1\\
		\mathrm{d}X_t=(m-\kappa Y_t-\theta X_t) \mathrm{~d}t +\sqrt{Y_t}\tilde{\rho} \mathrm{~d}B_t,
	\end{cases}\quad  t \in \left[ 0, \infty \right), 
	\label{model01}
\end{equation}
where $B_t=(B_t^1,\ldots,B_t^d)^{\mathsf{T}}$, $t\in[0,\infty)$, is a $d$-dimensional independent Brownian vector, $(Y_0,X_0)^{\mathsf{T}}$ is an arbitrary initial value independent of $B$ such that $\mathbb{P}(Y_0\in(0,\infty))=1$, $a\in(0,\infty)$, $m\in(-\infty,\infty)^n$, $b\in(-\infty,\infty)$, $\kappa\in(-\infty,\infty)^n$, $\theta$ is an $n\times n$ real matrix and  $\tilde{\rho}=[\rho_{J1}\;\rho_{JJ}]$, where $J=\lbrace 2,\ldots,d\rbrace$, $\rho_{J1}=(\rho_{i1})_{i\in J}$ and $\rho_{JJ}=(\rho_{ij})_{i,j\in J}$ which are introduced via the block representation of the $d\times d$ positive definite triangular low matrix $\rho=\begin{pmatrix}
	\rho_{11}&\rho_{1J}\\
	\rho_{J1}&\rho_{JJ}
	
\end{pmatrix}$ satisfying $\sum_{j=1}^d {\rho}^2_{ij}=\sigma_i^2$, for all $i\in\lbrace 1,\ldots,d\rbrace$, where $\mathbf{\sigma}=(\sigma_1,\ldots,\sigma_d)^{\mathsf{T}}\in (0,\infty)^d$. Note that the first component of the $\mathit{AD}(1,n)$ model is a CIR process which is also a continuous state branching process with branching mechanism $by+\dfrac{\rho_{11}^2}{2}y^2$ and immigration mechanism $ay$ and conditionally on $\sigma (Y_t, t\in[0,\infty))$ the second component is an $n$-dimensional Vasicek process.
	
The $\mathit{AD}(1,n)$, as an affine process, has been employed in finance since the last decades and they have found growing interest due to their computational tractability as well as their capability to capture empirical evidence from financial time series, see, e.g., Duffie et al \cite{Duffie} and  Filipovic and Mayerhofer in \cite{Filipovic}, where they gave more financial applications of affine models and proved existence and uniqueness through stochastic invariance of the canonical state space. The $\mathit{AD}(1,n)$ model $\eqref{model01}$ covers different known diffusion models as Vasicek, Cox-Ingersoll-Ross and Heston models. Further there is a vast literature on affine models, see, e.g., \cite{Barczy2,Basak,Alaya,Boylog,Duffie,Fathallah,Fathallah-Kebaier,Filipovic,Friesen,Zhang}. Moreover, several papers have studied the drift parameter estimation for affine submodels with different statistical settings based on continuous or discrete observations and using different types of estimators such as the maximum likelihood estimators (MLE) and conditional least squares estimators (CLSE), see, e.g., \cite{Barczy,Barczy2,Basak,Dahbi,Alaya,Alaya1,Boylog,Fathallah} and the references therein. It should be noted that statistical inferences of $\mathit{AD}(1,n)$ model and its submodels are rarely treated in the multidimensional case, see, e.g., \cite{Basak} and \cite{Dahbi}.

In the previous work \cite{Dahbi}, on the one hand, we established the stationarity for the $\mathit{AD}(1,n)$ model when $a\in\left[0,\infty\right),\ b\in(0,\infty)$ and $\theta$ a diagonalizable positive definite matrix. We also proved the exponential ergodicity of the $\mathit{AD}(1,n)$ for $a\in(0,\infty)$, $b\in(0,\infty)$ and $\theta$ a diagonalizable positive definite matrix. On the other hand, we introduced a classification of the $\mathit{AD}(1,n)$ processes based on the asymptotic behavior of its expectation to distinguish three cases : subcritical case when $b\wedge \lambda_{\min}(\theta)\in(0,\infty)$, critical case either when $b\in[0,\infty)$ and $\lambda_{\min}(\theta)=\lambda_{\max}(\theta)=0$ or when $b=0$ and $\lambda_{\min}(\theta)\in(0,\infty)$ and supercritical case when $b\wedge\lambda_{\max}(\theta)\in(-\infty,0)$, for more details see Proposition \ref{classification}. In the statistical part, we  studied the asymptotic properties of the maximum likelihood estimators in the subcritical case and a special supercritical case. In the subcritical case, we proved the strong consistency of the MLE and its asymptotic normality when $a\in(\frac{1}{2},\infty)$, $b\in (0,\infty)$ and $\theta$ is a diagonalizable positive definite matrix. For a special supercritical case, we showed the asymptotic normality of the maximum likelihood estimators with an exponential rate of convergence. Note that we were not able to treat all different cases using the MLE due to the involved terms $\int_0^t \frac{X_s^{(i)}X_s^{(j)}}{Y_s}$$\mathrm{~d}s,$ where asymptotic behavior as $t$ tends to infinity is still open in non-ergodic cases.
 Bolyog and Pap \cite{Boylog} considered the $\mathit{AD}(1,1)$ model called also two-factor affine model with a third Wiener process added to its second component. For a fixed $T\in(0,\infty)$, they constructed the CLSE of the drift parameters based on continuous time observations $(Y_t,X_t)_{t\in[0,T]}$ using the idea that such estimators are nothing but the high frequency limit in probability as $N\to\infty$ of the CLSE based on discrete time observations $(Y_{\frac{k}{N}},X_{\frac{k}{N}})_{k\in\{0,\ldots,\lfloor NT\rfloor\}}$. In the subcritical case, they proved strong consistency and asymptotic normality of the CLSE. In a special critical case and under the additional assumption $\kappa=0$, they showed weak consistency for the CLSE of $(b,\kappa,\theta)$ and they determined the asymptotic behavior for the global CLSE. In a special supercritical case under the additional assumption, $m\kappa\in(-\infty,0)$, they proved the strong consistency for the CLSE of $b$, weak consistency for the CLSE of $(\kappa,\theta)$ and asymptotic mixed normality for the global CLSE.
 
Ben Alaya and Kebaier \cite{Alaya} studied the CIR parameter estimation using the maximum likelihood estimation method based on continuous time observations $(Y_{t})_{t\in[0,T]}$. They established different asymptotic theorems on the error estimation with different rates of convergence and types of limit distributions in the ergodic and non-ergodic cases. They studied as well the estimation based on discrete time observations $(Y_{k\Delta_N})_{k\in\{0,\ldots,N\}}$ of high frequency and infinite horizon, i.e., $\Delta_N\to 0$ and $N\Delta_N\to\infty$ as $N\to\infty$. In this case, the estimator is obtained by discretizing the MLE of the continuous setting. they gave sufficient conditions on stepsize $\Delta_N$ under which the error correctly normalized converges and the limit theorems found for the continuous version can be easily carried out to the discrete one.\\

In this paper, inspired of \cite{Boylog}, we study at first statistical inferences of the continuous CLSE for an $\mathit{AD}(1,n)$ process based on continuous time observations $(Z_t)_{0\leq t\leq T}$, $T\in(0,\infty)$. In this case, the CLSE is well defined through \eqref{fT1 GT1} and \eqref{fT2 GT2} and we present its asymptotic properties in the subcritical, a special critical and a special supercritical cases, see respectively theorems \ref{Asymptotic Normality CLSE b>0 Continuous}, \ref{Théorème SCC continuous observations} and \ref{normality theorem supercritical case}. In a second part, inspired of \cite{Alaya}, we study the estimation problem based on discrete time observations, we suppose that we observe the process at equidistant instants $t_k=k\Delta_N,\ 0\leq k\leq N$ and we consider its correspondent CLSE. Under the conditions of high frequency $\Delta_N\to0$ and infinite horizon $N\Delta_N\to\infty$, as $N\to\infty$, we study asymptotic properties of the discrete CLSE in the subcritical, special critical and special supercritical cases, see respectively Theorem \ref{Asymptotic Normality CLSE b>0 DISCRETE}, Theorem \ref{Asymptotic Normality SCC DISCRETE} and Theorem \ref{normality theorem supercritical DISCRETE case}. Note that using this approach the limit theorems established in this setting are obtained from the ones of the continuous version.\\

The remainder of the paper is organized as follows: In the second section, we present all the principal notations and tools. In section $3$, we present the CLSE of the drift parameters based on continuous time observations and the results related to. the fourth section is devoted to the study of the discrete CLSE. In the appendix, we recall some important theorems used in this work and we give some moment results related to the $\mathit{AD}(1,n)$ processes.

\section{Preliminary}
Let $\N,\; \R, \;\mathbb{R}_{+},\; \mathbb{R}_{++}, \; \mathbb{R}_{-}, \; \mathbb{R}_{--}$ and $\mathbb{C}$
denote the sets of non-negative integers, real numbers, non-negative real numbers, positive real numbers, non-positive real numbers, negative real numbers and complex numbers, respectively and let $\mathcal{D}=\R_{+}\times\R^n$, for $n \in \N \setminus \lbrace 0 \rbrace.$ For $x,y\in\R$, we will use the notation $x\wedge y:=\min(x,y)$ and $x\vee y:=\max(x,y)$. 
For all $z\in\mathbb{C}$, \texttt{Re}($z$) denotes the real part of $z$ and \texttt{Im}($z$) denotes the imaginary part of $z$. Let us denote, for $p,q\in\N \setminus \{0\}$,  by $\mathcal{M}_{p,q}$ the set of $p\times q$ real matrices, $\mathcal{M}_p$ the set of $p\times p$ real matrices, $\mathbf{I}_p$ the identity matrix in $\mathcal{M}_{p}$, $\textbf{0}_{p,q}$ the null matrix in $\mathcal{M}_{p,q}$, $\textbf{0}_{p}:=\textbf{0}_{p,1}$ and $\mathbf{1}_p\in\mathcal{M}_{p,1}$ is the 1-vetor of size $p$.
We denote, for all diagonalizable matrix $A\in\mathcal{M}_p$, by $eig(A)$ the column vector containing all the eigenvalues of $A$ and by $\lambda_i(A)$, $\lambda_{\min}(A)$ and $\lambda_{\max}(A)$, the $i^{th}$, the smallest and the largest eigenvalues of $A$, respectively. For $m,n\in\N \setminus \{0\} $, we write $I=\{1,\ldots,m\}$ and $J=\{m+1,m+n\}$ and for all $x\in \R^{m+n}$, we write $\mathbf{x}_I=(x_i)_{i\in I}$ and $\mathbf{x}_J=(x_j)_{j\in J}$. Throughout this paper, we use the notation 
$$A=\begin{bmatrix}
	A_{II}&A_{IJ}\\
	A_{JI}&A_{JJ}
\end{bmatrix}$$
for $A\in\mathcal{M}_{m+n}$ where $A_{II}=(a_{i,j})_{i,j\in{I}}$, $A_{IJ}=(a_{i,j})_{i\in{I},j\in{J}}$, $A_{JI}=(a_{i,j})_{i\in{J},j\in{I}}$ and $A_{JJ}=(a_{i,j})_{i,j\in{J}}$. 


For all matrices $A_1,\ldots,A_k$, $\text{diag}(A_1,\ldots,A_k)^{\mathsf{T}}$ denotes the block matrix containing the matrices $A_1,\ldots,A_k$ in its diagonal and for all matrix $\textbf{M}$,  $\norm{\textbf{M}}_1:=\max\limits_{j}\displaystyle\sum_i \vert \textbf{M}_{ij}\vert$, $\norm{\textbf{M}}_2:=\sqrt{\max\text{eig}(\textbf{M}^{\mathsf{T}}\textbf{M})}$, with $\textbf{M}^{\mathsf{T}}$ the transpose matrix of $\textbf{M}$ and $\norm{\textbf{M}}_{\mathrm{F}}:=\sqrt{\sum_{ij}\vert \textbf{M}_{ij}\vert^2}$ called the Frobenius norm. In particular, for all vector $\textbf{x}=(x_1,\ldots,x_p)^{\mathsf{T}}\in\R^p$, we have $\norm{\textbf{x}}_1:=\vert {x}_1\vert+\ldots+\vert x_n\vert $ and $\norm{\mathbf{x}}_2=\norm{\mathbf{x}}_{\mathrm{F}}:=\sqrt{ {x}_1^2+\ldots+x_n^2}$.
We denote by $\otimes$ the Kronecker product defined, for all matrices $\mathbf{A},\mathbf{B}$ with suitable dimensions,
$$
\mathbf{A}\otimes\mathbf{B}=\begin{bmatrix}
	a_{11}\mathbf{B}&  \cdots & a_{1q}\mathbf{B}\\
	\vdots&  \ddots &\vdots\\
	a_{p1}\mathbf{B}&  \cdots & a_{pq}\mathbf{B}
\end{bmatrix}
$$
and by $\oplus$ the Kronecker sum defined by $\mathbf{A}\oplus\mathbf{B}=A\otimes\mathbf{I}+\mathbf{I}\otimes B$. We define the vec-operator applied on a matrix $\mathbf{A}$ denoted by $\text{vec}(\mathbf{A})$, which stacks its columns into a column vector, for more details, see, e.g., \cite{Horn matrix} and \cite{matrix cookbook}. Some of associated matrix properties used in this paper are given, for all matrices with suitable dimensions $\mathbf{A},\mathbf{B}, \mathbf{C}$ and $\mathbf{D}$, by
\begin{equation}
	(\mathbf{A}\otimes\mathbf{C})(\mathbf{B}\otimes\mathbf{D})=\mathbf{A}\mathbf{B}\otimes\mathbf{C}\mathbf{D},\quad 
	\text{vec}(\mathbf{A}\mathbf{B}\mathbf{C})=(\mathbf{C}^{\mathsf{T}}\otimes\mathbf{A})\text{vec}(\mathbf{B}),\quad
	e^{\mathbf{A}\oplus\mathbf{B}}=e^{\mathbf{A}}\otimes e^{\mathbf{B}}.
	\label{matrix property}
\end{equation}
We use the notations $H_{\mathbf{x}}(f)$ and $\nabla_{\mathbf{x}}f$ for the Hessian matrix and the gradient column vector of the function $f$ with respect to ${\mathbf{x}}$. we denote by $\mathcal{C}^2(\mathcal{D},\R)$ the set of twice continuously differentiable real-valued functions on $\mathcal{D}$ and by $\mathcal{C}_c^2(\mathcal{D},\R)$ its subset of functions with compact support. We will denote the convergence in probability, in distribution and almost surely  by $\stackrel{\mathbb{P}}\longrightarrow$, $\stackrel{\mathcal{D}}{\longrightarrow}$, $\stackrel{a.s.}{\longrightarrow}$ respectively.
Let $(\Omega,\mathcal{F}, \mathbb{P})$ be a probability space endowed with the augmented filtration $(\mathcal{F}_t)_{t \in \R_{+}}$ corresponding to $(B_t)_{t\in\R_+}$ and a given initial value $Z_0=(Y_0,X_0)^{\mathsf{T}}$ being independent of $(B_t)_{t\in\R_+}$ such that $\mathbb{P}(Y_0\in\R_{++})=1$. We suppose that $(\mathcal{F}_t)_{t\in\R_+}$ satisfies the usual conditions, i.e., the filtration 
$(\mathcal{F}_{t})_{t \in \R_{+}}$ is right-continuous and $\mathcal{F}_{0}$ contains all the $\mathbb{P}$-null sets in $\mathcal{F}$. For all $t\in\R_+$, we use the notation  $\mathcal{F}_t^Y:=\sigma(Y_s;\,0\leq s\leq t)$ for the $\sigma$-algebra generated by $(Y_s)_{s\in[0,t]}$.\\

Note that we can rewrite the model $\eqref{model01}$ using a matrix representation as follows
\begin{equation}
	\mathrm{d}Z_t=\Lambda(Z_t)\tau \mathrm{d}t+\sqrt{Z_t^1}\rho\,\mathrm{d}B_t,\quad  t \in \mathbb{R}_{+},
	\label{model0 Z}
\end{equation} 
where $Z_t=(Z_t^1,\ldots,Z_t^d)^{\mathsf{T}}=(Y_t,X_t^{1},\ldots,X_t^{n})^{\mathsf{T}}$, for all $t\in\R_{+}$, $
\Lambda(Z_t)=\begin{bmatrix}
	\Lambda_1(Z_t)&\mathbf{0}_{1,n(d+1)}\\
	\mathbf{0}_{n,2}&\mathbf{I}_n \otimes K(Z_t)
\end{bmatrix},$
with $\Lambda_1(Z_t)=(1,-Z_t^1)$ and $K(Z_t)=(1,-Z_t^1,\ldots,-Z_t^d)$ and the $d^2+1$-dimensional vector $\tau$ is stacking in turn the unknown drift parameters of $Y,X^{1},\ldots,X^{n}$ into a column vector, namely, $$\tau=\left(a,b,m_1,\kappa_1,\theta_{11},\ldots,\theta_{1n},\ldots,m_n,\kappa_n,\theta_{n1},\ldots,\theta_{nn}\right)^{\mathsf{T}}.$$
\begin{remarque}
Let us recall that for statistical estimations with continuous observations, we always suppose that the diffusion parameter ${\rho}$ is known, see for example the arguments given in \cite[Section 5, page 14]{Dahbi}.
\end{remarque}

\section{CLSE based on continuous time observations}\label{Continuous section}
We consider $(Y_t,X_t)_{t\in\R_+}$ strong solution of the SDE $\eqref{model01}$. The following discussion is about the construction of a CLSE for the drift parameter $\tau$ based on continuous time observation of $(Y_t,X_t)_{t\in[0,T]}$, for some $T\in\R_{++}$. 
At first, we consider the CLSE $\hat{\tau}_{T,N}$ based on the following discretizated process $(Y_{\frac{i}{N}},X_{\frac{i}{N}})_{i\in\lbrace 0,1,\ldots,\lfloor NT \rfloor\rbrace}$, for $N\in\N\setminus\lbrace 0\rbrace$ and it is obtained by solving the extremum problem
{\begin{equation}
	\begin{aligned}
	\hat{\tau}_{T,N}:=\underset{\tau\in \R^{d^2+1}}{\arg\min}\displaystyle\sum_{i=1}^{ \lfloor NT \rfloor}\left[\left(Y_{\frac{i}{N}}-\mathbb{E}\left( Y_{\frac{i}{N}}\vert \mathcal{F}_{\frac{i-1}{N}}\right)\right)^2+\left(X_{\frac{i}{N}}-\mathbb{E}\left( X_{\frac{i}{N}}\vert \mathcal{F}_{\frac{i-1}{N}}\right)\right)^{\mathsf{T}}\left(X_{\frac{i}{N}}-\mathbb{E}\left(X_{\frac{i}{N}}\vert \mathcal{F}_{\frac{i-1}{N}}\right)\right)\right].
	\end{aligned}
\label{extremum problem}
\end{equation}
}
Thanks to Itô's formula applied on the processes $(e^{bt}Y_t)_{t\in[0,T]}$ and $(e^{t\theta}X_t)_{t\in[0,T]}$, we obtain
\begin{equation}
Y_t= e^{-b(t-s)}Y_s+a\displaystyle\int_s^t e^{-b(t-u)}\mathrm{d}u+\rho_{11} \displaystyle\int_s^t e^{-b(t-u)}\sqrt{Y_u}\mathrm{d}B_u^{1}
\end{equation}
and
\begin{equation}
	X_t=e^{-\theta (t-s)}X_s+\displaystyle\int_s^t e^{-\theta(t-u)}(m- Y_u\kappa)\mathrm{d}u+\int_s^t e^{-\theta(t-u)}\sqrt{Y_u}\tilde{\rho}\mathrm{~d}B_u,
	\label{Xt Xs expression}
\end{equation}
for all $s\in[0,t]$. Consequently, we get
$
		\mathbb{E}\left( Y_t\vert \mathcal{F}_s\right)=e^{-b(t-s)}Y_s+a\displaystyle\int_0^{t-s} e^{-bu}\mathrm{d}u
$
and $\mathbb{E}\left( X_t\vert \mathcal{F}_s\right)$ is equal to
\begin{multline*}
		e^{-\theta(t-s)}X_s+ \displaystyle\int_0^{t-s} e^{-\theta u} m \mathrm{d}u-Y_s\displaystyle\int_0^{t-s} e^{-b u}e^{\theta (u-t+s)}\kappa\, \mathrm{d}u-a\displaystyle\int_0^{t-s} \left(\displaystyle\int_0^{u} e^{-b(u-v)}\mathrm{d}v\right)e^{\theta (u-t+s)}\kappa\,\mathrm{d}u.
\end{multline*}
Thus, for all $i\in \N\setminus\{0\}$, we get
$
		\mathbb{E}\left( Y_{\frac{i}{N}}\vert \mathcal{F}_{\frac{i-1}{N}}\right)=e^{-{\frac{b}{N}}}Y_{\frac{i-1}{N}}+a\displaystyle\int_0^{\frac{1}{N}} e^{-bu}\mathrm{d}u
$ and  $\mathbb{E}\left( X_{\frac{i}{N}}\vert \mathcal{F}_{\frac{i-1}{N}}\right)$ is equal to
\begin{equation*}
e^{-{\frac{1}{N}}\theta}X_{\frac{i-1}{N}}+ \displaystyle\int_0^{\frac{1}{N}} e^{-\theta u} m \mathrm{d}u-Y_{\frac{i-1}{N}}\displaystyle\int_0^{\frac{1}{N}}e^{-b u} e^{\theta (u-{\frac{1}{N}})}\kappa \,\mathrm{d}u-a\displaystyle\int_0^{\frac{1}{N}} \left(\displaystyle\int_0^{u} e^{-b(u-v)}\mathrm{d}v\right)e^{\theta (u-{\frac{1}{N}})}\kappa\mathrm{d}u.
\end{equation*}
Consequently, by combining the relation \eqref{extremum problem} with the above explicit expressions of conditional expectations, we get 
\begin{equation*}
	\begin{aligned}
		\hat{\tau}_{T,N}&=\underset{\tau\in \R^{d^2+1}}{\arg\min}\displaystyle\sum_{i=1}^{ \lfloor NT \rfloor}\left[\left(Y_{\frac{i}{N}}-Y_{\frac{i-1}{N}}-\left( \tilde{a}- \tilde{b}Y_{\frac{i-1}{N}} \right)\right)^2+\left(X_{\frac{i}{N}}-X_{\frac{i-1}{N}}-\left( \tilde{m} -Y_{\frac{i-1}{N}} \tilde{\kappa}-\tilde{\theta} X_{\frac{i-1}{N}}  \right)\right)^{\mathsf{T}}\right.\\
		&\qquad \qquad \qquad \quad \left.\times\left(X_{\frac{i}{N}}-X_{\frac{i-1}{N}}-\left( \tilde{m} - Y_{\frac{i-1}{N}}\tilde{\kappa} -\tilde{\theta} X_{\frac{i-1}{N}}  \right)\right)^{ \,}\right],
	\end{aligned}
\end{equation*}
with
$\tilde{a}:=a\displaystyle\int_0^{\frac{1}{N}} e^{-bu}\mathrm{d}u,\quad \tilde{b}=1-e^{-\frac{b}{N}}$, $\tilde{m}:=\displaystyle\int_0^{\frac{1}{N}} e^{-\theta u} m\, \mathrm{d}u-a\displaystyle\int_0^{\frac{1}{N}}\left(\displaystyle\int_0^{u} e^{-b(u-v)}\mathrm{d}v\right) e^{\theta (u-{\frac{1}{N}})}\kappa\,\mathrm{d}u$, $
\tilde{\kappa}:=\displaystyle\int_0^{\frac{1}{N}} e^{-b u}e^{\theta (u-{\frac{1}{N}})}\kappa \,\mathrm{d}u$ and $\tilde{\theta}=I_n-e^{-\frac{1}{N}\theta}$.\\
Let the function $g_N$ defined, for all $N\in\N\setminus\lbrace0\rbrace$, by $g_N(a,b,m,\kappa,\theta)=(\tilde{a},\tilde{b},\tilde{m},\tilde{\kappa},\tilde{\theta})$
 and let us introduce the vector $(\tilde{a}_{T,N},\tilde{b}_{T,N},\tilde{m}_{T,N},\tilde{\kappa}_{T,N},\tilde{\theta}_{T,N})^{\mathsf{T}}$ as the CLSE associated to the parameter vector $(\tilde{a},\tilde{b},\tilde{m},\tilde{\kappa},\tilde{\theta})$.
 Then, it can be observed that, on one hand, we have
\begin{equation}
	(\tilde{a}_{T,N},\tilde{b}_{T,N})^{\mathsf{T}}=\underset{(\tilde{a},\tilde{b})^{\mathsf{T}}\in \R^2}{\arg\min}\displaystyle\sum_{i=1}^{ \lfloor NT \rfloor}\left(Y_{\frac{i}{N}}-Y_{\frac{i-1}{N}}-\left( \tilde{a}- \tilde{b}Y_{\frac{i-1}{N}} \right)\right)^2,
	\label{c d}
\end{equation}
and on the other hand, we have
\begin{equation}
	\begin{aligned}
		\text{vec}([\tilde{m}_{T,N},\tilde{\kappa}_{T,N},\tilde{\theta}_{T,N}]^{\mathsf{T}})=\underset{\text{vec}([\tilde{m},\tilde{\kappa},\tilde{\theta}]^{\mathsf{T}})\in \R^{d^2-1}}{\arg\min}\displaystyle\sum_{i=1}^{ \lfloor NT \rfloor}&\left(X_{\frac{i}{N}}-X_{\frac{i-1}{N}}-\left( \tilde{m} - Y_{\frac{i-1}{N}}\tilde{\kappa} -\tilde{\theta} X_{\frac{i-1}{N}}  \right)\right)^{\mathsf{T}}\\
		&\times\left(X_{\frac{i}{N}}-X_{\frac{i-1}{N}}-\left( \tilde{m} - Y_{\frac{i-1}{N}}\tilde{\kappa} -\tilde{\theta} X_{\frac{i-1}{N}}  \right)\right).
	\end{aligned}
	\label{delta varepsilon zeta}
\end{equation}

Hence, in order solve first the extremum problem $\eqref{c d}$, it is enough to solve the system
\begin{equation}\label{derivative 1 equals zero}
	\begin{cases}
		\displaystyle\sum_{i=1}^{ \lfloor NT \rfloor} Y_{\frac{i}{N}}-Y_{\frac{i-1}{N}}-\left( \tilde{a}- \tilde{b}Y_{\frac{i-1}{N}} \right)=0,\\
		\displaystyle\sum_{i=1}^{ \lfloor NT \rfloor} \left(Y_{\frac{i}{N}}-Y_{\frac{i-1}{N}}-\left( \tilde{a}- \tilde{b}Y_{\frac{i-1}{N}} \right)\right)Y_{\frac{i-1}{N}}=0,
	\end{cases}
\end{equation}
which can be written as follows
$\Gamma^{(1)}_{T,N}\begin{bmatrix}
		\tilde{a}\\\tilde{b}
	\end{bmatrix}=\phi^{(1)}_{T,N},
$
with
\begin{equation*}
	\Gamma^{(1)}_{T,N}=\begin{bmatrix}
		\lfloor NT \rfloor& -\displaystyle\sum_{i=1}^{ \lfloor NT \rfloor}Y_{\frac{i-1}{N}}\\
		-\displaystyle\sum_{i=1}^{ \lfloor NT \rfloor}Y_{\frac{i-1}{N}}& \displaystyle\sum_{i=1}^{ \lfloor NT \rfloor}Y_{\frac{i-1}{N}}^2
	\end{bmatrix}\quad\text{and}\quad \phi^{(1)}_{T,N}=\begin{bmatrix}
		Y_{\frac{\lfloor NT \rfloor}{N}}-Y_{0}\\
		-\displaystyle\sum_{i=1}^{ \lfloor NT \rfloor} \left(Y_{\frac{i}{N}}-Y_{\frac{i-1}{N}}\right)Y_{\frac{i-1}{N}}
	\end{bmatrix}.
\end{equation*}
Then, povided that $\lfloor NT \rfloor \displaystyle\sum_{i=1}^{ \lfloor NT \rfloor}Y_{\frac{i-1}{N}}^2 -\left(\displaystyle\sum_{i=1}^{ \lfloor NT \rfloor}Y_{\frac{i-1}{N}}\right)^2>0$, we deduce that
\begin{equation*}
\begin{bmatrix}
	\tilde{a}_{T,N}\\\tilde{b}_{T,N}
\end{bmatrix}=\left(\Gamma^{(1)}_{T,N}\right)^{-1}\phi^{(1)}_{T,N}.	
\end{equation*}
Next, in order to solve the second extremum problem $\eqref{delta varepsilon zeta}$, it is enough to solve the system
\begin{equation}\label{derivative 2 equals 0}
	\begin{cases}
		\displaystyle\sum_{i=1}^{ \lfloor NT \rfloor} X_{\frac{i}{N}}-X_{\frac{i-1}{N}}-\left( \tilde{m} -Y_{\frac{i-1}{N}} \tilde{\kappa}-\tilde{\theta} X_{\frac{i-1}{N}}  \right) =\mathbf{0}_n\\
		\displaystyle\sum_{i=1}^{ \lfloor NT \rfloor}Y_{\frac{i-1}{N}}\left(X_{\frac{i}{N}}-X_{\frac{i-1}{N}}-\left( \tilde{m} -Y_{\frac{i-1}{N}} \tilde{\kappa}-\tilde{\theta} X_{\frac{i-1}{N}}  \right)\right)=\mathbf{0}_n\\
		\displaystyle\sum_{i=1}^{ \lfloor NT \rfloor}\left( X_{\frac{i}{N}}-X_{\frac{i-1}{N}}-\left( \tilde{m} -Y_{\frac{i-1}{N}} \tilde{\kappa}-\tilde{\theta} X_{\frac{i-1}{N}}  \right)\right)X_{\frac{i-1}{N}}^{\mathsf{T}}=\mathbf{0}_{n\times n}.
	\end{cases}
\end{equation}
The above system can be written in the following matrix form
$
	\Gamma^{(2)}_{T,N}\begin{bmatrix}
		\tilde{m}^{\mathsf{T}}\\
		\tilde{\kappa}^{\mathsf{T}}\\
		\tilde{\theta}^{\mathsf{T}}
	\end{bmatrix} = \phi^{(2)}_{T,N},
$
with
\begin{equation*}
	\Gamma^{(2)}_{T,N}=\begin{bmatrix}
		\lfloor NT \rfloor & -\displaystyle\sum_{i=1}^{ \lfloor NT \rfloor}Y_{\frac{i-1}{N}}  &-\displaystyle\sum_{i=1}^{ \lfloor NT \rfloor}X_{\frac{i-1}{N}}^{\mathsf{T}}\\
		-\displaystyle\sum_{i=1}^{ \lfloor NT \rfloor}Y_{\frac{i-1}{N}} & \displaystyle\sum_{i=1}^{ \lfloor NT \rfloor}Y_{\frac{i-1}{N}}^2 & \displaystyle\sum_{i=1}^{ \lfloor NT \rfloor} Y_{\frac{i-1}{N}}X_{\frac{i-1}{N}}^{\mathsf{T}}\\
		-\displaystyle\sum_{i=1}^{ \lfloor NT \rfloor}X_{\frac{i-1}{N}} &  \displaystyle\sum_{i=1}^{ \lfloor NT \rfloor} Y_{\frac{i-1}{N}}X_{\frac{i-1}{N}} & \displaystyle\sum_{i=1}^{ \lfloor NT \rfloor} X_{\frac{i-1}{N}}X_{\frac{i-1}{N}}^{\mathsf{T}}
	\end{bmatrix}
	\quad\text{and}\quad \phi^{(2)}_{T,N}=\begin{bmatrix}
		X_{\frac{\lfloor NT \rfloor}{N}}^{\mathsf{T}}-X_{0}^{\mathsf{T}}\\
		-\displaystyle\sum_{i=1}^{ \lfloor NT \rfloor} Y_{\frac{i-1}{N}}\left(X_{\frac{i}{N}}^{\mathsf{T}}-X_{\frac{i-1}{N}}^{\mathsf{T}}\right)\\
		-\displaystyle\sum_{i=1}^{ \lfloor NT \rfloor} X_{\frac{i-1}{N}}\left(X_{\frac{i}{N}}^{\mathsf{T}}-X_{\frac{i-1}{N}}^{\mathsf{T}}\right)
	\end{bmatrix}.
\end{equation*}
Then, provided the invertibility of $\Gamma^{(2)}_{T,N}$, we deduce that
\begin{equation*}
	\begin{bmatrix}
		\tilde{m}_{T,N}^{\mathsf{T}}\\
		\tilde{\kappa}_{T,N}^{\mathsf{T}}\\
		\tilde{\theta}_{T,N}^{\mathsf{T}}
	\end{bmatrix}=\left(\Gamma^{(2)}_{T,N}\right)^{-1}\phi^{(2)}_{T,N}.
\end{equation*}
Consequently, it remains to prove the invertibility of the matrices $\Gamma^{(1)}_{T,N}$ and $\Gamma^{(2)}_{T,N}$. In our case, it is enough to show that they are almost surely positive definite. Let $x\in\R^2\setminus\lbrace \mathbf{0}_2 \rbrace$ and $y\in\R^{n+2}\setminus\lbrace \mathbf{0}_{n+2} \rbrace$. From one side, we have
\begin{equation}
	\begin{bmatrix}
		x_1\\
		x_2
	\end{bmatrix}^{\mathsf{T}}\Gamma^{(1)}_{T,N} \begin{bmatrix}
		x_1\\
		x_2
	\end{bmatrix} =\displaystyle\sum_{i=1}^{ \lfloor NT \rfloor}\begin{bmatrix}
		x_1\\
		x_2
	\end{bmatrix}^{\mathsf{T}}
	\begin{bmatrix}
		1\\
		- Y_{\frac{i-1}{N}}
	\end{bmatrix}
	\begin{bmatrix}
		1\\
		-Y_{\frac{i-1}{N}}
	\end{bmatrix}^{\mathsf{T}}
	\begin{bmatrix}
		x_1\\
		x_2
	\end{bmatrix}= \displaystyle\sum_{i=1}^{ \lfloor NT \rfloor} (x_1-x_2 Y_{\frac{i-1}{N}})^2\geq 0.
\label{invertibility 1}
\end{equation}
From the other side,
$	\begin{bmatrix}
		x_1\\
		x_2
	\end{bmatrix}^{\mathsf{T}}\Gamma^{(1)}_{T,N} \begin{bmatrix}
		x_1\\
		x_2
	\end{bmatrix} =0\ \Leftrightarrow \ x_1-x_2 Y_{\frac{i-1}{N}}=0,$ for all $ i\in\lbrace 1,2,\ldots,\lfloor NT\rfloor\rbrace, 
$ which is impossible since $x\neq 0_2$ and, for each $i\in \lbrace 1,2,\ldots,\lfloor NT\rfloor\rbrace $, the distribution of $Y_{\frac{i-1}{N}}$ is absolutely continuous because its conditional distribution given $Y_0$ is absolutely continuous.
Likewise,
\begin{equation}
y^{\mathsf{T}}\Gamma^{(2)}_{T,N} y =\displaystyle\sum_{i=1}^{ \lfloor NT \rfloor}y^{\mathsf{T}}
	\begin{bmatrix}
		1\\
		- Y_{\frac{i-1}{N}}\\
		- X_{\frac{i-1}{N}}
	\end{bmatrix}
	\begin{bmatrix}
		1\\
		-Y_{\frac{i-1}{N}}\\
		- X_{\frac{i-1}{N}}^{\mathsf{T}}
	\end{bmatrix}^{\mathsf{T}}
y= \displaystyle\sum_{i=1}^{ \lfloor NT \rfloor} (y_1-y_2 Y_{\frac{i-1}{N}}-\tilde{y}^{\mathsf{T}}X_{\frac{i-1}{N}})^2\geq 0,
\label{invertibility 2}
\end{equation}
where $\tilde{y}=(y_3,\ldots,y_{n+2})^{\mathsf{T}}$.
In addition,
$
y^{\mathsf{T}}\Gamma^{(2)}_{T,N} y =\mathbf{0}_{n+2}\ \Leftrightarrow\ y_1-y_2 Y_{\frac{i-1}{N}}-\tilde{y}^{\mathsf{T}}X_{\frac{i-1}{N}}=0_{n+2}$, for all  $i\in\lbrace 0,1,\ldots,\lfloor NT\rfloor\rbrace 
$
which is also impossible since $y\neq \mathbf{0}_{n+2}$ and, for each $i\in \lbrace 1,2,\ldots,\lfloor NT\rfloor\rbrace $, the distribution of $(Y_{\frac{i-1}{N}},X_{\frac{i-1}{N}})$ is absolutely continuous because its conditional distribution given $(Y_0,X_0)$ is absolutely continuous (See Theorem \ref{Stationarity theorem} and relation (31) in \cite{Dahbi}).

Now, we define the approximate CLSE of $\tau$ by 
\begin{equation*}
	\hat{\tau}_{T,N}^{\mathrm{approx}}:=N(\tilde{a}_{T,N},\tilde{b}_{T,N},\text{vec}([\tilde{m}_{T,N},\tilde{\kappa}_{T,N},\tilde{\theta}_{T,N}]^{\mathsf{T}}))^{\mathsf{T}}.
\end{equation*}
Thanks to the almost sure continuity of  $(Y_t,X_t)_{t\in \R_{+}}$ and by Propsition $4.44$ in Jacod and Shirayev \cite{Jacod} with the Riemann sequence of adapted subdivisions $(\frac{i}{N}\wedge T)_{i\in\N}$, we obtain
\begin{equation}
	\dfrac{1}{N}\Gamma^{(1)}_{T,N}\stackrel{a.s.}{\longrightarrow}\begin{bmatrix}
		T & -\displaystyle\int_0^T Y_s \mathrm{d}s\\[8pt]
		-\displaystyle\int_0^T Y_s \mathrm{d}s & \displaystyle\int_0^T Y_s^2 \mathrm{d}s
	\end{bmatrix}=:G_T^{(1)},\quad	\phi^{(1)}_{T,N}\stackrel{\mathbb{P}}{\longrightarrow}\begin{bmatrix}
	Y_T-Y_{0}\\
	-\displaystyle\int_0^T Y_s \mathrm{d}Y_s\end{bmatrix}:=f_T^{(1)}
\label{fT1 GT1}
\end{equation}
and
\begin{equation}
	\dfrac{1}{N}\Gamma^{(2)}_{T,N}\stackrel{a.s.}{\longrightarrow}\begin{bmatrix}
		T & -\displaystyle\int_0^T Y_s \mathrm{d}s  &-\displaystyle\int_0^T X_s^{\mathsf{T}} \mathrm{d}s\\[8pt]
		-\displaystyle\int_0^T Y_s \mathrm{d}s & \displaystyle\int_0^T Y_s^2 \mathrm{d}s & \displaystyle\int_0^T Y_s X_s^{\mathsf{T}}  \mathrm{d}s\\[8pt]
		-\displaystyle\int_0^T X_s \mathrm{d}s  &   \displaystyle\int_0^T Y_s X_s \mathrm{d}s  &  \displaystyle\int_0^T X_s X_s^{\mathsf{T}} \mathrm{d}s 
	\end{bmatrix}=:G_T^{(2)},\quad \phi^{(2)}_{T,N}\stackrel{\mathbb{P}}{\longrightarrow}\begin{bmatrix}
	X_{T}^{\mathsf{T}}-X_{0}^{\mathsf{T}}\\[2pt]
	-\displaystyle\int_0^T Y_{s} \mathrm{d}X_{s}^{\mathsf{T}}\\[8pt]
	-\displaystyle\int_0^T X_{s} \mathrm{d}X_{s}^{\mathsf{T}}
\end{bmatrix}:=f_T^{(2)},
\label{fT2 GT2}
\end{equation}
as $N\to\infty$. Hence, provided the invertibility of $G_T^{(1)}$ and $G_T^{(2)}$ and using relations \eqref{fT1 GT1} and \eqref{fT2 GT2}, the Slutsky's theorem yields
\begin{equation}
	\hat{\tau}_{T,N}^{\mathrm{approx}}=N\begin{bmatrix}
		\tilde{a}_{T,N}\\\tilde{b}_{T,N}\\\text{vec}
		\begin{bmatrix}
			\tilde{m}_{T,N}^{\mathsf{T}}\\\tilde{\kappa}_{T,N}^{\mathsf{T}}\\\tilde{\theta}_{T,N}^{\mathsf{T}}
	\end{bmatrix} \end{bmatrix}\stackrel{\mathbb{P}}{\longrightarrow}\begin{bmatrix}
		(G_T^{(1)})^{-1}f_T^{(1)}\\
		\text{vec}\left((G_T^{(2)})^{-1}f_T^{(2)}\right)
	\end{bmatrix}=:\begin{bmatrix}
		\hat{a}_{T}\\\hat{b}_{T}\\
		\text{vec} \begin{bmatrix}
			\hat{m}_{T}^{\mathsf{T}}\\\hat{\kappa}_{T}^{\mathsf{T}}\\\hat{\theta}_{T}^{\mathsf{T}}
	\end{bmatrix} \end{bmatrix}=:\hat{\tau}_T,\quad \text{as } N\to\infty.
	\label{tau approx}
\end{equation}
Hence using the same argument to prove the invertibility of matrices $\Gamma^{(1)}_{T,N}$ and $\Gamma^{(2)}_{T,N}$, by replacing the sum from $0$ to $\lfloor NT \rfloor$ by the integral on $[0,T]$, it is easy to check that $G_T^{(1)}$ and $G_T^{(2)}$ are invertible matrices, see relations \eqref{invertibility 1} and \eqref{invertibility 2} and calculation there.\\ 
Note that this approximate CLSE is constructed through the regular part of a first order Taylor approximation of the function $g_N$ at zero point. In the following lemma we prove that $\hat{\tau}_T$ is also the limit of the CLSE $\hat{\tau}_{T,N}$ introduced in \eqref{extremum problem}. Hence, $\hat{\tau}_T$ is called the CLSE based on continuous time observations.  
\begin{lemme}
	Let us consider the affine model $\eqref{model01}$ with $a\in\R_+$, $b\in \R$, $m\in\R^n$, $\kappa \in \R^n$, $\theta \in \mathcal{M}_n$. Then for each $T\in\R_{++}$, we have $\hat{\tau}_{T,N}\stackrel{\mathbb{P}}{\longrightarrow}\hat{\tau}_T$, as $N\to\infty$. 
	\label{lemma hat tau T N to hat tau}
\end{lemme}
\begin{proof}
Thanks to relation \eqref{tau approx}, it is easy to check that  
\begin{equation}
	(\tilde{a}_{T,N},\tilde{b}_{T,N},\tilde{m}_{T,N},\tilde{\kappa}_{T,N},\tilde{\theta}_{T,N})\stackrel{\mathbb{P}}{\longrightarrow} \left(0,0,\mathbf{0}_n,\mathbf{0}_n,\mathbf{0}_{n, n}\right),\quad \text{as } N\to\infty.
	\label{tau hat T N goes to zero}
\end{equation}
Hence, since the function $g_N$ admits an inverse function on the neighbourhood of the zero point, we deduce that $$g_N^{-1}(\tilde{a}_{T,N},\tilde{b}_{T,N},\tilde{m}_{T,N},\tilde{\kappa}_{T,N},\tilde{\theta}_{T,N})=\hat{\tau}_{T,N},$$ with probability tending to one as $N\to \infty$. Namely, we get $$\hat{b}_{T,N}=-N\log(1-\tilde{b}_{T,N}),\quad \hat{a}_{T,N}=\dfrac{\tilde{a}_{T,N}}{\int_0^{\frac{1}{N}}e^{-\hat{b}_{T,N}u}\mathrm{d}u},\quad\hat{\theta}_{T,N}=-N\log(I_n-\tilde{\theta}_{T,N}),$$
	$$\hat{\kappa}_{T,N}=\left(\displaystyle\int_0^{\frac{1}{N}} e^{\hat{\theta}_{T,N} (u-{\frac{1}{N}})} e^{-\hat{b}_{T,N} u}\mathrm{d}u \right)^{-1}\tilde{\kappa}_{T,N}$$ and $$\hat{m}_{T,N}=\left(\displaystyle\int_0^{\frac{1}{N}} e^{-\hat{\theta}_{T,N} u} \mathrm{d}u \right)^{-1}\left(\tilde{m}_{T,N}+\hat{a}_{T,N}\displaystyle\int_0^{\frac{1}{N}} e^{(u-{\frac{1}{N}})\hat{\theta}_{T,N}}\hat{\kappa}_{T,N}\left(\displaystyle\int_0^{u} e^{-\hat{b}_{T,N}(u-v)}\mathrm{d}v\right)\mathrm{d}u\right).$$
The following discussion is similar in spirit to the one used in the proof of 3.3 Lemma in \cite{Boylog}. Using the following first order Taylor approximation at zero point $\ln(1-x)=-x+o(x)$, relations \eqref{tau approx} and \eqref{tau hat T N goes to zero} implies 
	$$\hat{b}_{T,N}=-N\log(1-\tilde{b}_{T,N})=N\tilde{b}_{T,N}+o(\tilde{b}_{T,N})\stackrel{\mathbb{P}}{\longrightarrow}\hat{b}_T,\quad \text{as}\ N\to \infty.$$
	Moreover, by relations \eqref{tau approx} and \eqref{tau hat T N goes to zero}, thank to the following first order Taylor approximation at zero point $\dfrac{1-e^{-x}}{x}=1+o(x)$, we get 
	$$\hat{a}_{T,N}=\dfrac{\tilde{a}_{T,N}}{\int_0^{\frac{1}{N}}e^{-\hat{b}_{T,N}u}\mathrm{d}u}=N\tilde{a}_{T,N}\dfrac{\frac{1}{N}\hat{b}_{T,N}}{1-e^{-\frac{1}{N}\hat{b}_{T,N}}}=\dfrac{N\tilde{a}_{T,N}}{1+o(\frac{1}{N}\hat{b}_{T,N})}\stackrel{\mathbb{P}}{\longrightarrow} \hat{a}_T$$
likewise, thanks to the following matrix version of the first order Taylor approximation at zero point $\ln(\mathbf{I}_n-A)=-A+o(\norm{A}_2)$, we obtain
$$\hat{\theta}_{T,N}=-N\ln(\mathbf{I}_n-\tilde{\theta}_{T,N})=N\tilde{\theta}_{T,N}+o\left(\norm{\tilde{\theta}}_2\right)\stackrel{\mathbb{P}}{\longrightarrow}\hat{\theta}_T,\quad \text{as}\ N\to \infty.$$
In addition, we have
	\begin{equation*}
		\begin{aligned}
			\hat{\kappa}_{T,N}	=\left(N\displaystyle\int_0^{\frac{1}{N}} e^{(\hat{\theta}_{T,N}(u-\frac{1}{N})}e^{-\hat{b}_{T,N} u}\mathrm{d}u \right)^{-1} N\tilde{\kappa}_{T,N}.
		\end{aligned}
	\end{equation*}
On the one hand, by the convergence $\eqref{tau approx}$, we have $N\tilde{\kappa}_{T,N}\stackrel{\mathbb{P}}{\longrightarrow}\hat{\kappa}_T$ as $N\to \infty$. On the other hand, by the mean value theorem, for all $1\leq i,j\leq n$, there exists $0 \leq u^{\star}\leq \dfrac{1}{N}$, such that
\begin{equation*}
	\left(N\displaystyle\int_0^{\frac{1}{N}} e^{\hat{\theta}_{T,N}(u-\frac{1}{N})}e^{-\hat{b}_{T,N} u}\mathrm{d}u \right)_{i,j}=\left(e^{\hat{\theta}_{T,N}(u^{\star}-\frac{1}{N})}\right)_{i,j}e^{-\hat{b}_{T,N} u^{\star}}\stackrel{\mathbb{P}}{\longrightarrow}\delta_{ij},\quad \text{as}\ N\to \infty.
\end{equation*}
Consequently, we obtain $\hat{\kappa}_{T,N}\stackrel{\mathbb{P}}{\longrightarrow}\hat{\kappa}_T$ as $N\to \infty$.
Finally, since we have
	\begin{equation*}
			\hat{m}_{T,N}=\left(N\displaystyle\int_0^{\frac{1}{N}} e^{-\hat{\theta}_{T,N} u} \mathrm{d}u \right)^{-1}\left(N\tilde{m}_{T,N}+N\hat{a}_{T,N}\displaystyle\int_0^{\frac{1}{N}} e^{\hat{\theta}_{T,N} (u-{\frac{1}{N}})}\hat{\kappa}_{T,N}\left(\displaystyle\int_0^{u} e^{-\hat{b}_{T,N}(u-v)}\mathrm{d}v\right)\mathrm{d}u\right),
	\end{equation*}
then it is easy to check using the same arguments that $\hat{m}_{T,N}\stackrel{\mathbb{P}}{\longrightarrow}\hat{m}_T$ which completes the proof.
\end{proof}
In order to study asymptotic properties of the continuous CLSE, we need first to write the relative error term. Using the SDE in \eqref{model01}, we obtain the following equality $f_T^{(1)}=G_T^{(1)}\begin{bmatrix}
		a\\
		b
	\end{bmatrix}+h_T^{(1)},$
with $h_T^{(1)}=\begin{bmatrix}
	\sigma_1\displaystyle\int_0^T\sqrt{Y_s}\mathrm{d}B_s^1\\
	-\sigma_1\displaystyle\int_0^TY_s \sqrt{Y_s}\mathrm{d}B_s^1
\end{bmatrix}$. In a similar way, we obtain
$f_T^{(2)}=G_T^{(2)}\begin{bmatrix}
		m^{\mathsf{T}}\\
		\kappa^{\mathsf{T}}\\
		\theta^{\mathsf{T}}
	\end{bmatrix}+h_T^{(2)}$, where $h_T^{(2)}=\begin{bmatrix}
	\displaystyle\int_0^T\sqrt{Y_s}\mathrm{d}B_s^{\mathsf{T}}\tilde{\rho}^{\mathsf{T}}\\[8pt]
	-\displaystyle\int_0^T Y_s \sqrt{Y_s}\mathrm{d}B_s^{\mathsf{T}}\tilde{\rho}^{\mathsf{T}}\\[8pt]
	-\displaystyle\int_0^T \sqrt{Y_s} X_s \mathrm{d}B_s^{\mathsf{T}}\tilde{\rho}^{\mathsf{T}}
\end{bmatrix}$.
 Therefore, by combining these results with the second property in \eqref{matrix property}, we conclude that 
\begin{equation}
	\hat{\tau}_T-\tau=\begin{bmatrix}
		(G_T^{(1)})^{-1}h_T^{(1)}\\
		\text{vec}\left((G_T^{(2)})^{-1}h_T^{(2)}\right)
	\end{bmatrix}=\begin{bmatrix}
		(G_T^{(1)})^{-1}h_T^{(1)}\\
		(I_{n}\otimes G_T^{(2)})^{-1}\text{vec}(h_T^{(2)})
	\end{bmatrix}=G_T^{-1}h_T,
	\label{err}
\end{equation}
where $G_T:=\begin{bmatrix}
	G_T^{(1)}& \textbf{0}\\
	\textbf{0}& I_{n}\otimes G_T^{(2)}
\end{bmatrix}$ and $h_T:=\begin{bmatrix}
h_T^{(1)}\\
\text{vec}(h_T^{(2)})
\end{bmatrix}$.
\subsection{Subcritical case}
We recall that, based on the classification given in Proposition \ref{classification}, the $\mathit{AD}(1,n)$ model is set to be in the subcritical case if $b\in\R_{++}$ and $\theta$ is a positive definite matrix. In this case, the $\mathit{AD}(1,n)$ process is stationary and exponentially ergodic, see Theorem \ref{Stationarity theorem} and Theorem \ref{ergodicity theorem} in the appendix for more details.   
\subsubsection{Consistency of the CLSE}
\begin{theoreme}\label{consistency CLSE b>0}
	Let us consider the model $\eqref{model01}$ with $a,b\in\R_{++}$, $m,\kappa\in \R^n$ and $\theta \in \mathcal{M}_n$ a positive definite diagonalizable matrix with random initial value $(Y_0,X_0)$ satisfying $\mathbb{P}(Y_0\in\R_{++})=1$. Then the $CLSE$ of $\tau$ is strongly consistent, i.e., $\hat{\tau}_T \stackrel{a.s.}{\longrightarrow} \tau$, as $T\to \infty$. 
\end{theoreme}
\begin{proof}
	By relation$\eqref{err}$, we have
$
		\hat{\tau}_T-\tau=(T^{-1} G_T)^{-1}(T^{-1}h_T)
$. Furthermore,	using Theorem \ref{ergodicity theorem}, we get
	\begin{equation*}
		T^{-1}G_T \stackrel{a.s.}{\longrightarrow} \mathbb{E}(G_{\infty}),\quad \text{as } T\to \infty,
	\end{equation*}
	where 
	\begin{equation}
		G_{\infty}:=\begin{bmatrix}
			G_{\infty}^{(1)}& \textbf{0}\\
			\textbf{0}& I_n\otimes G_{\infty}^{(2)}
		\end{bmatrix}
		\label{Ginfty}
	\end{equation}
	with
	\begin{equation*}
		G_{\infty}^{(1)}:=\begin{bmatrix}
			1 & - Y_{\infty} \\
			- Y_{\infty} & Y_{\infty}^2
		\end{bmatrix}\quad\text{and}\quad G_{\infty}^{(2)}:= \begin{bmatrix}
			1 & - Y_{\infty}  &- X_{\infty}^{\mathsf{T}} \\
			- Y_{\infty} & -Y_{\infty}^2 & - Y_{\infty} X_{\infty}^{\mathsf{T}}\\
			- X_{\infty} & Y_{\infty} X_{\infty}  &   X_{\infty} X_{\infty}^{\mathsf{T}}  
		\end{bmatrix},
	\end{equation*}
	where $Y_\infty$ and $X_\infty$ are defined by the Fourier-Laplace transform given in Theorem \ref{Stationarity theorem}. It is worth to note that $\mathbb{E}(G_{\infty})$ is finite by Remark \ref{Remark moment à l'infini}. In addition, $\mathbb{E}(G_{\infty})$ is an invertible matrix since it is positive definite matrix. In fact, for all $y\in\R^{n+2}\setminus\lbrace0\rbrace$,
	$	\begin{bmatrix}
			y_1\\ y_2
		\end{bmatrix}^{\mathsf{T}} \mathbb{E}(G_\infty^{(1)})\begin{bmatrix}
			y_1\\ y_2
		\end{bmatrix}=\mathbb{E}((y_1-y_2Y_\infty)^2)>0$
	and
$		\begin{bmatrix}
			y_1\\ y_2\\ \tilde{y}
		\end{bmatrix}^{\mathsf{T}} \mathbb{E}(G_\infty^{(2)})\begin{bmatrix}
			y_1\\ y_2\\ \tilde{y}
		\end{bmatrix}=\mathbb{E}((y_1-y_2Y_\infty-\tilde{y}^{\mathsf{T}} X_\infty)^2)>0,$
since $(Y_\infty,X_\infty)$ is absolutely continuous, hence we obtain $y_1-y_2Y_\infty\neq 0$ and $y_1-y_2Y_\infty-\tilde{y}^{\mathsf{T}} X_\infty\neq 0$ with probability $1$, where $\tilde{y}=(y_3,\ldots,y_{n+2})^{\mathsf{T}}$. Hence, we deduce that
	\begin{equation*}
		(T^{-1}G_T)^{-1}\stackrel{a.s.}{\longrightarrow} \begin{bmatrix}
			(\mathbb{E}(G_{\infty}^{(1)}))^{-1}& \textbf{0}\\
			\textbf{0}& I_n\otimes (\mathbb{E}(G_{\infty}^{(2)}))^{-1}
		\end{bmatrix}=(\mathbb{E}(G_\infty))^{-1},\quad \text{as } T\to \infty.
	\end{equation*}
	In order to finish the proof, it is enough to check that
$T^{-1}h_T \stackrel{a.s.}{\longrightarrow} {0},$ as $T\to \infty$.
	Firstly, for all $T\in\R_{+}$, let $(h_{T,1}^{(2)},h_{T,2}^{(2)},\ldots,h_{T,n}^{(2)}):=h_T^{(2)}$, which means that $h_{T,i}^{(2)}$ would be the $i^{\text{th}}$ column of $h_T^{(2)}$ of size $n+2$, for all $1\leq i\leq n$. Consequently, for all $1\leq i,j\leq n$,  $\langle h^{(1)}_T,\left.h^{(1)}_T\right.^{\mathsf{T}}\rangle$, $\langle h_T^{(1)},\left.h^{(2)}_{T,i}\right.^{\mathsf{T}}\rangle$ and $\langle h^{(2)}_{T,i},\left.h^{(2)}_{T,j}\right.^{\mathsf{T}}\rangle$ are given, respectively, by
	$\mathrm{H}_T^{(1)}=\sigma_1^2\displaystyle\int_{0}^T\begin{bmatrix}
		Y_s&-Y_s^2\\-Y_s^2&Y_s^3
	\end{bmatrix}\mathrm{d}s$, $\mathrm{H}_T^{(2)}=\sigma_1 \tilde{\rho}_{i1}\displaystyle\int_{0}^T\begin{bmatrix}
	Y_s&-Y_s^2&-Y_s X_s^{\mathsf{T}}\\-Y_s^2&Y_s^3&Y_s^2X_s^\mathsf{T}
\end{bmatrix}\mathrm{d}s$ and
$\mathrm{H}_T^{(3)}=\left(\tilde{\rho}\tilde{\rho}^{\mathsf{T}}\right)_{ij}\displaystyle\int_{0}^T\begin{bmatrix}
Y_s&-Y_s^2&-Y_s X_s^{\mathsf{T}}\\
-Y_s^2&Y_s^3&Y_s^2X_s^{\mathsf{T}}\\
-Y_s X_s&Y_s^2 X_s& Y_s X_s X_s^{\mathsf{T}}
\end{bmatrix}\mathrm{d}s.$ Secondly, by Remark \ref{Remark moment à l'infini}  and Theorem \ref{ergodicity theorem}, the quadratic variation $\mathrm{H}_T$ of $h_T$ satisfies
\begin{equation*}
\dfrac{1}{T}\mathrm{H}_T\stackrel{a.s.}{\longrightarrow} \mathbb{E}(\mathrm{H}_{\infty}),\quad\text{as }T\to\infty,
\end{equation*}  where
	\begin{equation}
		\mathrm{H}_\infty=\begin{bmatrix}
			\sigma_1^2\mathrm{H}_\infty^{(1)}&\sigma_1(\rho_{J1}\otimes\mathrm{H}_\infty^{(2)})\\\sigma_1(\rho_{J1}\otimes\mathrm{H}_\infty^{(2)})^{\mathsf{T}}&\tilde{\rho}\tilde{\rho}^{\mathsf{T}}\otimes\mathrm{H}_\infty^{(3)}
		\end{bmatrix},
		\label{quadratic variation of h at infinty}
	\end{equation}
	with $\mathrm{H}_\infty^{(1)}=\begin{bmatrix}
		Y_\infty&-Y_\infty^2\\-Y_\infty^2&Y_\infty^3
	\end{bmatrix}$, $\mathrm{H}_\infty^{(2)}=\begin{bmatrix}
		Y_\infty&-Y_\infty^2&-Y_\infty X_\infty^{\mathsf{T}}\\-Y_\infty^2&Y_\infty^3&Y_\infty^2X_\infty^\mathsf{T}
	\end{bmatrix}$ and
	$\mathrm{H}_\infty^{(3)}=\begin{bmatrix}
		Y_\infty&-Y_\infty^2&-Y_\infty X_\infty^{\mathsf{T}}\\
		-Y_\infty^2&Y_\infty^3&Y_\infty^2X_\infty^{\mathsf{T}}\\
		-Y_\infty X_\infty&Y_\infty^2 X_\infty& Y_\infty X_\infty X_\infty^{\mathsf{T}}
	\end{bmatrix}.$
Finally, the strong law of large numbers for continuous local martingales given by Theorem \ref{LFGN} applied on $h_T$ completes the proof.
\end{proof}
\subsubsection{Asymptotic behavior of the CLSE}
\begin{theoreme}\label{Asymptotic Normality CLSE b>0 Continuous}
	Le us consider the affine diffusion model $\eqref{model01}$ with $a,b\in\R_{++}$ $m,\kappa\in\R^n$, $\theta\in\mathcal{M}_n$ a positive definite diagonalizable matrix and the random initial value $(Y_0,X_0)$ independent of $(B_t)_{t\in\R_+}$ satisfying $\mathbb{P}(Y_0\in\R_{++})=1$. Then  the CLSE of $\tau$ is asymptotically normal, namely, 
	\begin{equation*}
		\sqrt{T}(\hat{\tau}_T-\tau)\stackrel{\mathcal{D}}{\longrightarrow}\mathcal{N}_{d^2+1}({0},[\mathbb{E}(G_\infty)]^{-1}\mathbb{E}(\mathrm{H}_\infty)[\mathbb{E}(G_\infty)]^{-1})
	\end{equation*}
	where $G_{\infty}$ and $\mathrm{H}_{\infty}$ are given by relations $\eqref{Ginfty}$ and \eqref{quadratic variation of h at infinty}, respectively.
\end{theoreme}
\begin{proof}
 We have, for all $T\in\R_+$,
	\begin{equation}
		\sqrt{T}(\hat{\tau}_T-\tau)=\left(\dfrac{1}{T}G_T\right)^{-1}\left(\dfrac{1}{\sqrt{T}}h_T\right).
		\label{proof normality}
	\end{equation}
	On one side, by Remark \ref{Remark moment à l'infini}, the ergodicity theorem implies that $\left(\dfrac{1}{T}G_T\right)^{-1}\stackrel{a.s.}{\longrightarrow}[\mathbb{E}(G_\infty)]^{-1}$, as $T\to\infty$. On the other side, by the central limit theorem given by Theorem \ref{CLT Van Zanten}, we obtain
	\begin{equation*}
		\dfrac{1}{\sqrt{T}}h_T \stackrel{\mathcal{D}}{\longrightarrow}\mathcal{N}_{d^2+1}(0,\mathbb{E}(\mathrm{H}_{\infty})),\quad\text{as }T\to\infty.
	\end{equation*}	
	Finally, thanks to Slutsky's lemma and to the convergence $\eqref{proof normality}$, the proof is completed \end{proof}
\subsection{A special critical case} \label{A special critical case continuous observations}
We recall that, based on the classification given in Proposition \ref{classification}, the $\mathit{AD}(1,n)$ model is set to be in the critical case if $b=0$ and $\lambda_{\min}(\theta)\in\R_{++}$ or if $\theta=\mathbf{0}_{n,n}$ and $b\in\R_{+}$. In the following, we consider the special critical case $b=0$ and $\theta=\mathbf{0}_{n,n}$ where we add the assuption $\kappa=\mathbf{0}_n$ which will be necessary in order to use scaling properties introduced by Barczy et al. in \cite{Barczy critical case}. 
\begin{theoreme}
	Let us consider the model $\eqref{model01}$ with $a\in\R_{+}$, $b=0$, $m\in \R^n$, $\kappa=\mathbf{0}_n$ and $\theta=\mathbf{0}_{n,n}$ with random initial value $(Y_0,X_0)^{\mathsf{T}}$ independent of $(B_t)_{t\in\R_+}$ satisfying $\mathbb{P}(Y_0\in\R_{++})=1$. Then 
	{\begin{equation*}
		\begin{bmatrix}
			\hat{a}_T-a\\
			T \hat{b}_T\\
   \text{vec}\begin{bmatrix}
				\hat{m}_T^{\mathsf{T}}-m^{\mathsf{T}}\\
				T \hat{\kappa}_T^{\mathsf{T}}\\
				T \hat{\theta}_T^{\mathsf{T}}
			\end{bmatrix}
		\end{bmatrix}\stackrel{\mathcal{D}}{\longrightarrow}
		\text{diag}\left(\mathit{U}_1^{-1},I_n\otimes\mathit{U}_2^{-1}\right)(\mathit{R}_1,\text{vec}(R_2))^{\mathsf{T}},\quad\text{as }T\to\infty,
\end{equation*}}
with $\mathit{U}_1=\begin{bmatrix}
			1& -\displaystyle\int_0^1 \mathcal{Y}_s \mathrm{d}s\\[8pt]
			-\displaystyle\int_0^1 \mathcal{Y}_s \mathrm{d}s & \displaystyle\int_0^1 \mathcal{Y}_s^2 \mathrm{d}s
		\end{bmatrix}$, $\mathit{U}_2=\begin{bmatrix}
				1 & -\displaystyle\int_0^1 \mathcal{Y}_s \mathrm{d}s  &-\displaystyle\int_0^1 \mathcal{X}_s^{\mathsf{T}} \mathrm{d}s\\[8pt]
				-\displaystyle\int_0^1 \mathcal{Y}_s \mathrm{d}s & \displaystyle\int_0^1 \mathcal{Y}_s^2 \mathrm{d}s & \displaystyle\int_0^1 \mathcal{Y}_s \mathcal{X}_s^{\mathsf{T}}  \mathrm{d}s\\[8pt]
				-\displaystyle\int_0^1 \mathcal{X}_s \mathrm{d}s  &   \displaystyle\int_0^1 \mathcal{Y}_s \mathcal{X}_s \mathrm{d}s  &  \displaystyle\int_0^1 \mathcal{X}_s \mathcal{X}_s^{\mathsf{T}} \mathrm{d}s 
			\end{bmatrix}$,\\
   $\mathit{R}_1=\begin{bmatrix}
			\mathcal{Y}_1 -a\\
a\displaystyle\int_0^1 \mathcal{Y}_s\mathrm{d}s-\displaystyle\int_0^1 \mathcal{Y}_s\mathrm{~d}\mathcal{Y}_s
		\end{bmatrix}$
and
	$\mathit{R}_2=\begin{bmatrix}
				\mathcal{X}_1^{\mathsf{T}}-m^{\mathsf{T}}\\
				\displaystyle\int_0^1 \mathcal{Y}_s\mathrm{~d}s\,m^{\mathsf{T}}-\displaystyle\int_0^1 \mathcal{Y}_s\mathrm{~d}\mathcal{X}_s^{\mathsf{T}}\\
				\displaystyle\int_0^1 \mathcal{X}_s\mathrm{~d}s\,m^{\mathsf{T}}-\displaystyle\int_0^1 \mathcal{X}_s\mathrm{~d}\mathcal{X}_s^{\mathsf{T}}
			\end{bmatrix}$,
			where $(\mathcal{Y}_t,\mathcal{X}_t)_{t\in\R_+}$ is the unique strong solution of 
	\begin{equation}
		\begin{cases}
			\mathrm{d}\mathcal{Y}_t=a \mathrm{~d}t+\rho_{11}\sqrt{\mathcal{Y}_t}\mathrm{~d}B_t^1\\
			\mathrm{d}\mathcal{X}_t=m \mathrm{~d}t +\sqrt{\mathcal{Y}_t}\tilde{\rho} \mathrm{~d}B_t,		
		\end{cases}\quad t\in\R_{+},\quad(\mathcal{Y}_0,\mathcal{X}_0)=(0,\mathbf{0}_n).
	\label{mathcal Y and X}
	\end{equation} 
\label{Théorème SCC continuous observations}
\end{theoreme}
\begin{proof}
Using the two following equations
\begin{equation*}
G_T^{(1)}=	\begin{bmatrix}
		T^{\frac{1}{2}}&0\\
		0&T^{\frac{3}{2}}
	\end{bmatrix}	\begin{bmatrix}
	1 & -\dfrac{1}{T^2}\displaystyle\int_0^T Y_s \mathrm{d}s\\
	-\dfrac{1}{T^2}\displaystyle\int_0^T Y_s \mathrm{d}s & \dfrac{1}{T^3}\displaystyle\int_0^T Y_s^2 \mathrm{d}s
\end{bmatrix}\begin{bmatrix}
	T^{\frac{1}{2}}&0\\
	0&T^{\frac{3}{2}}
\end{bmatrix}
\end{equation*}
and
\begin{equation*}
h_T^{(1)}=	\begin{bmatrix}
	T&0\\
	0&T^2
\end{bmatrix}	\begin{bmatrix}
\dfrac{\sigma_1}{T}\displaystyle\int_0^T \sqrt{Y_s}\mathrm{~d}B_s^1\\[8pt]
-\dfrac{\sigma_1}{T^2} \displaystyle\int_0^T Y_s\sqrt{Y_s}\mathrm{~d}B_s^1
\end{bmatrix},
\end{equation*}
we deduce that 
\begin{equation}\label{hat a - a and T hat b SCC}
	\begin{bmatrix}
		\hat{a}_T-a\\
		T\hat{b}_T
	\end{bmatrix}=\begin{bmatrix}
	1&0\\
	0&T
\end{bmatrix}	\begin{bmatrix}
\hat{a}_T-a\\
\hat{b}_T
\end{bmatrix}=\begin{bmatrix}
1 & -\dfrac{1}{T^2}\displaystyle\int_0^T Y_s \mathrm{d}s\\[8pt]
-\dfrac{1}{T^2}\displaystyle\int_0^T Y_s \mathrm{d}s & \dfrac{1}{T^3}\displaystyle\int_0^T Y_s^2 \mathrm{d}s
\end{bmatrix}^{-1}	\begin{bmatrix}
\dfrac{\sigma_1}{T}\displaystyle\int_0^T \sqrt{Y_s}\mathrm{~d}B_s^1\\[8pt]
-\dfrac{\sigma_1}{T^2} \displaystyle\int_0^T Y_s\sqrt{Y_s}\mathrm{~d}B_s^1,
\end{bmatrix}.
\end{equation}
Similarly, we get 
\begin{equation}
		\label{hat m - m T hat k and T hat theta SSC}		\begin{bmatrix}
	\hat{m}_T^{\mathsf{T}}-m^{\mathsf{T}}\\
	T \hat{\kappa}_T^{\mathsf{T}}\\
	T \hat{\theta}_T^{\mathsf{T}}
\end{bmatrix}=\begin{bmatrix}
1 & -\dfrac{1}{T^2}\displaystyle\int_0^T Y_s \mathrm{d}s  &-\dfrac{1}{T^2}\displaystyle\int_0^T X_s^{\mathsf{T}} \mathrm{d}s\\[8pt]
-\dfrac{1}{T^2}\displaystyle\int_0^T Y_s \mathrm{d}s &\dfrac{1}{T^3} \displaystyle\int_0^T Y_s^2 \mathrm{d}s &\dfrac{1}{T^3} \displaystyle\int_0^T Y_s X_s^{\mathsf{T}}  \mathrm{d}s\\[8pt]
-\dfrac{1}{T^2}\displaystyle\int_0^T X_s \mathrm{d}s  &   \dfrac{1}{T^3}\displaystyle\int_0^T Y_s X_s \mathrm{d}s  &  \dfrac{1}{T^3}\displaystyle\int_0^T X_s X_s^{\mathsf{T}} \mathrm{d}s 
\end{bmatrix}^{-1}\begin{bmatrix}
\dfrac{1}{T}\displaystyle\int_0^T\sqrt{Y_s}\mathrm{d}B_s^{\mathsf{T}}\tilde{\rho}^{\mathsf{T}}\\[8pt]
-\dfrac{1}{T^2}\displaystyle\int_0^T Y_s \sqrt{Y_s}\mathrm{d}B_s^{\mathsf{T}}\tilde{\rho}^{\mathsf{T}}\\[8pt]
-\dfrac{1}{T^2}\displaystyle\int_0^T \sqrt{Y_s} X_s \mathrm{d}B_s^{\mathsf{T}}\tilde{\rho}^{\mathsf{T}}
\end{bmatrix}.
\end{equation}
By the SDEs given in the model $\eqref{model01}$ in the critical case and for $\kappa=\mathbf{0}_n$, we get
\begin{equation*}
	\begin{bmatrix}
		\dfrac{\sigma_1}{T}\displaystyle\int_0^T \sqrt{Y_s}\mathrm{~d}B_s^1\\[8pt]
		-\dfrac{\sigma_1}{T^2} \displaystyle\int_0^T Y_s\sqrt{Y_s}\mathrm{~d}B_s^1
	\end{bmatrix}=\begin{bmatrix}
		\dfrac{1}{T}(Y_T- Y_0-a)\\
		-\dfrac{1}{T^2}\left(\displaystyle\int_0^T Y_s\mathrm{~d}Y_s-a\displaystyle\int_0^T Y_s\mathrm{d}s\right)
	\end{bmatrix}
\end{equation*}
and 
\begin{equation*}
	\begin{bmatrix}
		\dfrac{1}{T}\displaystyle\int_0^T\sqrt{Y_s}\mathrm{d}B_s^{\mathsf{T}}\tilde{\rho}^{\mathsf{T}}\\[8pt]
		-\dfrac{1}{T^2}\displaystyle\int_0^T Y_s \sqrt{Y_s}\mathrm{d}B_s^{\mathsf{T}}\tilde{\rho}^{\mathsf{T}}\\[8pt]
		-\dfrac{1}{T^2}\displaystyle\int_0^T \sqrt{Y_s} X_s \mathrm{d}B_s^{\mathsf{T}}\tilde{\rho}^{\mathsf{T}}
	\end{bmatrix}
	=
	\begin{bmatrix}
		\dfrac{1}{T}(X_T^{\mathsf{T}}-X_0^{\mathsf{T}}-m^{\mathsf{T}})\\
		-\dfrac{1}{T^2}\left(\displaystyle\int_0^T Y_s\mathrm{~d}X_s^{\mathsf{T}}-\displaystyle\int_0^T Y_s\mathrm{~d}s\,m^{\mathsf{T}}\right)\\
		-\dfrac{1}{T^2}\left(\displaystyle\int_0^T X_s\mathrm{~d}X_s^{\mathsf{T}}-\displaystyle\int_0^T X_s\mathrm{~d}s\,m^{\mathsf{T}}\right)
	\end{bmatrix}.
\end{equation*}
Let us define the functional $V$ defined for $(Y_t)_{t\in[0,T]}$, $(X_t)_{t\in[0,T]}$ and $T$ by
\begin{multline}\label{function V SCC}
	V((Y_t)_{t\in[0,T]},(X_t)_{t\in[0,T]},T)= \left(\dfrac{1}{T}{Y}_T,\,\dfrac{1}{T}{X}_T,\,\dfrac{1}{T^2}\displaystyle\int_0^T {Y}_s\mathrm{~d}s,\,\dfrac{1}{T^2}\displaystyle\int_0^T{X}_s\mathrm{~d}s,\,\dfrac{1}{T^3}\displaystyle\int_0^T{Y}_s^2\mathrm{~d}s,\right.\\
  \left.\dfrac{1}{T^3}\displaystyle\int_0^T {X}_s {X}_s^{\mathsf{T}}\mathrm{~d}s,\,\dfrac{1}{T^3}\displaystyle\int_0^T{Y}_s{X}_s\mathrm{~d}s,\,\dfrac{1}{T^2}\displaystyle\int_0^T{Y}_s\mathrm{~d}{Y}_s,\,\dfrac{1}{T^2}\displaystyle\int_0^T{Y}_s\mathrm{~d}{X}_s,\,\dfrac{1}{T^2}\displaystyle\int_0^T{X}_s\mathrm{~d}{X}_s^{\mathsf{T}}\right).
\end{multline}
Hence, in order to prove that $V((Y_t)_{t\in[0,T]},(X_t)_{t\in[0,T]},T)\stackrel{\mathcal{D}}{\longrightarrow}V((\mathcal{Y}_t)_{t\in[0,1]},(\mathcal{X}_t)_{t\in[0,1]},1)$ it is enough, by Slutsky's lemma, to prove that $V((\mathcal{Y}_t)_{t\in[0,T]},(\mathcal{X}_t)_{t\in[0,T]},T)\stackrel{\mathcal{D}}{=}V((\mathcal{Y}_t)_{t\in[0,1]},(\mathcal{X}_t)_{t\in[0,1]},1)$ and that $V((Y_t)_{t\in[0,T]},(X_t)_{t\in[0,T]},T)-V((\mathcal{Y}_t)_{t\in[0,T]},(\mathcal{X}_t)_{t\in[0,T]},T)\stackrel{\mathbb{P}}{\longrightarrow}\mathbf{0}$.
Firstly, concerning the equality in distribution, using the scaling property of the Brownian motion $B$, we can show, for all $T\in\R_{++}$, that $(\mathcal{Y}_t,\mathcal{X}_t)_{t\in\R_+}\stackrel{\mathcal{D}}{=}\left(\dfrac{1}{T}\mathcal{Y}_{Tt},\dfrac{1}{T}\mathcal{X}_{Tt}\right)_{t\in\R_+}$. In fact, for $t\in\R_+$, through the SDE \eqref{mathcal Y and X}, we get
\begin{equation}
\label{scaling property}
	\begin{cases}
		\dfrac{1}{T}\mathcal{Y}_{Tt}=at+\rho_{11}\displaystyle\int_0^{Tt}\sqrt{\dfrac{\mathcal{Y}_s}{T}}\dfrac{\mathrm{d}B_s^1}{\sqrt{T}}
		\\\dfrac{1}{T}\mathcal{X}_{Tt}=tm+\displaystyle\int_0^{Tt}\sqrt{\dfrac{\mathcal{Y}_s}{T}}\tilde{\rho}\dfrac{\mathrm{d}B_s}{\sqrt{T}}
	\end{cases}\quad \Longleftrightarrow\quad \begin{cases}
	\bar{\mathcal{Y}}_{t}=at+\rho_{11}\displaystyle\int_0^{t}\sqrt{\bar{\mathcal{Y}}_s}\mathrm{~d}\bar{B}_s^1
	\\ \bar{\mathcal{X}}_{t}=tm+\displaystyle\int_0^{t}\sqrt{\bar{\mathcal{Y}}_s}\tilde{\rho}\mathrm{~d}\bar{B}_s,
\end{cases}
\end{equation}
where $\bar{\mathcal{Y}}_s=\dfrac{1}{T}\mathcal{Y}_{Ts}$, $\bar{\mathcal{X}}_s=\dfrac{1}{T}\mathcal{X}_{Ts}$ and $\bar{B}_s=\dfrac{1}{\sqrt{T}}B_{Ts}$. Hence we deduce that $(\mathcal{Y}_t,\mathcal{X}_t)_{t\in\R_+}\stackrel{\mathcal{D}}{=}(\bar{\mathcal{Y}}_t,\bar{\mathcal{X}}_t)_{t\in\R_+}$. Secondly, by Jacod and Shiryaev \cite[Proposition I.4.44]{Jacod}, the vector $V_N((\mathcal{Y})_{t\in[0,1]},(\mathcal{X}_t)_{t\in[0,1]},1)$ given by
\begin{multline*}
	\left(\mathcal{Y}_1,\,\mathcal{X}_1,\,\dfrac{1}{N}\displaystyle\sum_{i=1}^N \mathcal{Y}_{\frac{i}{N}},\,\dfrac{1}{N}\displaystyle\sum_{i=1}^N\mathcal{X}_{\frac{i}{N}},\,\dfrac{1}{N}\displaystyle\sum_{i=1}^N\mathcal{Y}_{\frac{i}{N}}^2,\,\dfrac{1}{N}\displaystyle\sum_{i=1}^N \mathcal{X}_{\frac{i}{N}} \mathcal{X}_{\frac{i}{N}}^{\mathsf{T}},\,\dfrac{1}{N}\displaystyle\sum_{i=1}^N\mathcal{Y}_{\frac{i}{N}}\mathcal{X}_{\frac{i}{N}},\,\displaystyle\sum_{i=1}^N\mathcal{Y}_{\frac{i-1}{N}}\left(\mathcal{Y}_{\frac{i}{N}}-\mathcal{Y}_{\frac{i-1}{N}}\right),\right.\\
  \left.\displaystyle\sum_{i=1}^N\mathcal{Y}_{\frac{i-1}{N}}\left(\mathcal{X}_{\frac{i}{N}}-\mathcal{X}_{\frac{i-1}{N}}\right),\,\displaystyle\sum_{i=1}^N\mathcal{Y}_{\frac{i-1}{N}}\left(\mathcal{X}_{\frac{i}{N}}-\mathcal{X}_{\frac{i-1}{N}}\right)^{\mathsf{T}}\right)
\end{multline*}
converges in probability to $V((\mathcal{Y}_t)_{t\in[0,1]},(\mathcal{X}_t)_{t\in[0,1]},1)$, as $N\to\infty$. Thirdly, by the same argument and using a change of variables, the vector $V_N((\mathcal{Y})_{t\in[0,T]},(\mathcal{X}_t)_{t\in[0,T]},T)$ given by
\begin{multline*}
	\left(\dfrac{1}{T}\mathcal{Y}_T,\,\dfrac{1}{T}\mathcal{X}_T,\,\dfrac{1}{N}\displaystyle\sum_{i=1}^N \dfrac{\mathcal{Y}_{\frac{Ti}{N}}}{T},\,\dfrac{1}{N}\displaystyle\sum_{i=1}^N\dfrac{\mathcal{X}_{\frac{Ti}{N}}}{T},\,\dfrac{1}{N}\displaystyle\sum_{i=1}^N\dfrac{\mathcal{Y}^2_{\frac{Ti}{N}}}{T},\,\dfrac{1}{N}\displaystyle\sum_{i=1}^N \dfrac{\mathcal{X}_{\frac{Ti}{N}}}{T} \dfrac{\mathcal{X}_{\frac{Ti}{N}}^{\mathsf{T}}}{T},\,\dfrac{1}{N}\displaystyle\sum_{i=1}^N\dfrac{\mathcal{Y}_{\frac{Ti}{N}}}{T}\dfrac{\mathcal{X}_{\frac{Ti}{N}}}{T},\right.\\
\left.\displaystyle\sum_{i=1}^N\dfrac{\mathcal{Y}_{\frac{T(i-1)}{N}}}{T}\left(\dfrac{\mathcal{Y}_{\frac{Ti}{N}}}{T}-\dfrac{\mathcal{Y}_{\frac{T(i-1)}{N}}}{T}\right),\,\displaystyle\sum_{i=1}^N\dfrac{\mathcal{Y}_{\frac{T(i-1)}{N}}}{T}\left(\dfrac{\mathcal{X}_{\frac{Ti}{N}}}{T}-\dfrac{\mathcal{X}_{\frac{T(i-1)}{N}}}{T}\right),\,\displaystyle\sum_{i=1}^N\dfrac{\mathcal{Y}_{\frac{T(i-1)}{N}}}{T}\left(\dfrac{\mathcal{X}_{\frac{Ti}{N}}}{T}-\dfrac{\mathcal{X}_{\frac{T(i-1)}{N}}}{T}\right)^{\mathsf{T}}\right)
\end{multline*}
converges in probability to $V((\mathcal{Y}_t)_{t\in[0,T]},(\mathcal{X}_t)_{t\in[0,T]},T)$, as $N\to\infty$. Consequently, since we have\\ $(\mathcal{Y}_t,\mathcal{X}_t)_{t\in\R_+}\stackrel{\mathcal{D}}{=}\left(\dfrac{1}{T}\mathcal{Y}_{Tt},\dfrac{1}{T}\mathcal{X}_{Tt}\right)_{t\in\R_+}$, then we get $V_N((\mathcal{Y}_t)_{t\in[0,1]},(\mathcal{X}_t)_{t\in[0,1]},1)\stackrel{\mathcal{D}}{=}V_N((\mathcal{Y}_t)_{t\in[0,T]},(\mathcal{X}_t)_{t\in[0,T]},T)$ which implies that 
$V((\mathcal{Y}_t)_{t\in[0,1]},(\mathcal{X}_t)_{t\in[0,1]},1)\stackrel{\mathcal{D}}{=}V((\mathcal{Y}_t)_{t\in[0,T]},(\mathcal{X}_t)_{t\in[0,T]},T)$. The aim of the following discussion is to prove that $V((Y_t)_{t\in[0,T]},(X_t)_{t\in[0,T]},T)-V((\mathcal{Y}_t)_{t\in[0,T]},(\mathcal{X}_t)_{t\in[0,T]},T)\stackrel{\mathbb{P}}{\longrightarrow}\mathbf{0}$ using the first absolute moments criteria. At first, we have 
\begin{equation}Y_T-\mathcal{Y}_T=Y_0+\sigma_1\displaystyle\int_0^T (\sqrt{Y_s}-\sqrt{\mathcal{Y}_s})\mathrm{~d}B_s^1
	\label{Yt-Ycalt}
\end{equation}
 and thanks to the comparison theorem given in Ikeda and Watanabe \cite{Ikeda}, we get $Y_t\geq\mathcal{Y}_t$ almost surely, for all $t\in\R_{+}$. Consequently, since $\dfrac{1}{T^2}\mathbb{E}((\sqrt{Y_s}-\sqrt{\mathcal{Y}_s})^2)\leq\dfrac{1}{T^2}\mathbb{E}(Y_s)+\dfrac{1}{T^2}\mathbb{E}(\mathcal{Y}_s)<\infty$ by Remark \ref{Remark moment resulty on time t}, we deduce that $\mathbb{E}\left( \dfrac{1}{T}\displaystyle\int_0^T (\sqrt{Y_s}-\sqrt{\mathcal{Y}_s})\mathrm{~d}B_s^1\right)=0$.
Hence, we obtain	$\mathbb{E}\left(\dfrac{1}{T}(Y_T-\mathcal{Y}_T)\right)\leq \dfrac{1}{T}\mathbb{E}(Y_0)\rightarrow 0$
and by Fubini's theorem, we get
\begin{equation*}
		\mathbb{E}\left(\dfrac{1}{T^2} \displaystyle\int_0^T (Y_s-\mathcal{Y}_s)\mathrm{~d}s\right)\leq\dfrac{1}{T^2} \displaystyle\int_0^T 	\mathbb{E}\left( Y_s-\mathcal{Y}_s\right)\mathrm{~d}s\leq \dfrac{1}{T}\mathbb{E}(Y_0)\rightarrow 0,\quad \text{as } T\to\infty.
\end{equation*}
Hence, we obtain $\dfrac{1}{T}(Y_T-\mathcal{Y}_T)\stackrel{\mathbb{P}}{\longrightarrow}0$ and $\dfrac{1}{T^2}\displaystyle\int_0^T(Y_s-\mathcal{Y}_s)\mathrm{~d}s\stackrel{\mathbb{P}}{\longrightarrow}0$, as $T\to\infty$. In a similar way, we prove that $\dfrac{1}{T}(X_T-\mathcal{X}_T)\stackrel{\mathbb{P}}{\longrightarrow}\mathbf{0}_n$ and that $\dfrac{1}{T^2}\displaystyle\int_0^T(X_s-\mathcal{X}_s)\mathrm{~d}s\stackrel{\mathbb{P}}{\longrightarrow}\mathbf{0}_n$, as $T\to\infty$.
Next, using Minkowski inequality and Fubini's theorem, we get 
\begin{align*}
	\mathbb{E}((Y_s-\mathcal{Y}_s)^2)&\leq 2\mathbb{E}\left(Y_0^2\right)+2\sigma_1^2\mathbb{E}\left(\displaystyle\int_0^s (\sqrt{Y_u}-\sqrt{\mathcal{Y}_u})^2\mathrm{~d}u\right)\\
	&\leq 2\mathbb{E}\left(Y_0^2\right)+2\sigma_1^2\mathbb{E}\left(\displaystyle\int_0^s (Y_u-\mathcal{Y}_u)\mathrm{~d}u\right)=2\mathbb{E}(Y_0^2)+2\sigma_1^2 \mathbb{E}(Y_0)s,
\end{align*}
thus we get $\underset{s\in[0,T]}{\sup}\mathbb{E}\left( (Y_s-\mathcal{Y}_s)^2\right)=O(T)$. Furthermore, we have $\underset{s\in[0,T]}{\sup}\mathbb{E}(Y_s^2+\mathcal{Y}_s^2)=O(T^2)$. Hence, by combining Cauchy-Schwarz and Minkowski inequalities, we deduce that $$
		\underset{s\in[0,T]}{\sup}\mathbb{E}\left( Y_s^2-\mathcal{Y}_s^2\right)\leq\underset{s\in[0,T]}{\sup}\sqrt{2\,	\mathbb{E}\left( (Y_s-\mathcal{Y}_s)^2\right)	\mathbb{E}\left(Y_s^2+\mathcal{Y}_s^2\right)}=O(T^{\frac{3}{2}}).$$ In the following, we denote by $C$ a positive constant independent of time that may change its value from line to line. Consequently, we get
\begin{equation*}
	\mathbb{E}\left(\left\vert	\dfrac{1}{T^3}\displaystyle\int_0^T(Y_s^2-\mathcal{Y}_s^2)\mathrm{~d}s\right\vert\right)\leq C \dfrac{T^{\frac{5}{2}}}{T^3}\to 0,\quad\text{as }T\to\infty. 
\end{equation*}
Similarly, it is easy to check that $\underset{s\in[0,T]}{\sup}\mathbb{E}\left(\left\| X_s-\mathcal{X}_s \right\|_1^2\right)= O(T)$, as $T\to\infty$ and that \\ \noindent$\underset{s\in[0,T]}{\sup}\mathbb{E}\left(\left\| X_s\right\|_1^2+\left\| \mathcal{X}_s\right\|_1^2\right)=O(T^2)$, as $T\to\infty$. Hence, we get
\begin{align*}
	\mathbb{E}\left(\left\|\dfrac{1}{T^3}\displaystyle\int_0^T (X_s X_s^{\mathsf{T}}-\mathcal{X}_s \mathcal{X}_s^{\mathsf{T}})\mathrm{~d}s\right\|_1\right)
&\leq\dfrac{1}{T^3}\displaystyle\int_0^T\left(\mathbb{E}\left(\left\| (X_s -\mathcal{X}_s )X_s^{\mathsf{T}}\right\|_1\right)+\mathbb{E}\left(\left\| \mathcal{X}_s( X_s- \mathcal{X}_s)^{\mathsf{T}}\right\|_1\right)\right)\mathrm{d}s\\
&\leq\dfrac{1}{T^3}\displaystyle\int_0^T\left(\mathbb{E}\left(\left\| X_s -\mathcal{X}_s\right\|_1\left\|X_s\right\|_1\right)+\mathbb{E}\left(\left\| \mathcal{X}_s\right\|_1\left\| X_s- \mathcal{X}_s\right\|_1\right)\right)\mathrm{d}s\\
&\leq\dfrac{1}{T^3}\displaystyle\int_0^T\sqrt{\mathbb{E}\left(\left\| X_s -\mathcal{X}_s \right\|_1^2\right)}\left(\sqrt{\mathbb{E}\left(\left\|X_s\right\|_1^2\right)}+\sqrt{\mathbb{E}\left(\left\| \mathcal{X}_s\right\|_1^2\right)}\right)\mathrm{d}s\\
&\leq C\dfrac{T^{\frac{5}{2}}}{T^3}\to 0,\quad\text{as }T\to\infty. 
\end{align*}
In addition, we obtain
\begin{align*}	\mathbb{E}\left(\left\|\dfrac{1}{T^3}\displaystyle\int_0^T (Y_s X_s-\mathcal{Y}_s \mathcal{X}_s)\mathrm{~d}s\right\|_1\right)&\leq\dfrac{1}{T^3}\displaystyle\int_0^T \left(\mathbb{E}\left(\left( Y_s -\mathcal{Y}_s\right)\left\| X_s\right\|_1\right)+\mathbb{E}\left(\mathcal{Y}_s\left\| X_s-\mathcal{X}_s\right\|_1\right)\right)\mathrm{d}s\\
	&\leq \dfrac{1}{T^3}\displaystyle\int_0^T \left(\sqrt{\mathbb{E}\left(\left( Y_s -\mathcal{Y}_s\right)^2\right)\mathbb{E}\left(\left\| X_s\right\|_1^2\right)}+\sqrt{\mathbb{E}\left(\mathcal{Y}_s^2\right)\mathbb{E}\left(\left\| X_s-\mathcal{X}_s\right\|_1^2\right)}\right)\mathrm{d}s\\
	&\leq C\dfrac{T^{\frac{5}{2}}}{T^3}\to 0,\quad\text{as }T\to\infty.
\end{align*}
Finally, by following the same method used in \cite[page 21,22]{Boylog}, it is easy to check the convergence of the last three terms of $V((Y_t)_{t\in[0,T]},(X_t)_{t\in[0,T]},T)-V((\mathcal{Y}_t)_{t\in[0,T]},(\mathcal{X}_t)_{t\in[0,T]},T)$. In the following, we treat, for example, the proof related to the last term. By a simple decomposition then by applying  Cauchy-Schwarz inequality, we obtain 
\begin{align*}
	\mathbb{E}\left(\dfrac{1}{T^2}\left\|\displaystyle\int_0^T X_s\mathrm{~d}X_s^{\mathsf{T}}-\displaystyle\int_0^T \mathcal{X}_s\mathrm{~d}\mathcal{X}_s^{\mathsf{T}}\right\|_1\right)&\leq \dfrac{1}{T^2}\sqrt{\mathbb{E}\left(\left\|\displaystyle\int_0^T (X_s-\mathcal{X}_s)\mathrm{~d}X_s^{\mathsf{T}} \right\|_1^2\right)}+\dfrac{1}{T^2}\sqrt{\mathbb{E}\left(\left\|\displaystyle\int_0^T \mathcal{X}_s\mathrm{~d}(X_s-\mathcal{X}_s)^{\mathsf{T}} \right\|_1^2\right)}.
\end{align*}
On the one hand, by equation (6.11) in \cite{Boylog}, we have $\mathbb{E}\left(\left\|X_s-\mathcal{X}_s\right\|_1^4\right)=O(T^3)$, then we obtain 
\begin{align*}
	\mathbb{E}\left(\left\|\displaystyle\int_0^T (X_s-\mathcal{X}_s)\mathrm{~d}X_s^{\mathsf{T}} \right\|_1^2\right)&\leq2\norm{m}_1^2\mathbb{E}\left(\left\|\displaystyle\int_0^T (X_s-\mathcal{X}_s)\mathrm{~d}s\right\|_1^2\right)+2\norm{\tilde{\rho}}_1^2\mathbb{E}\left(\left\|\displaystyle\int_0^T\sqrt{Y_s} (X_s-\mathcal{X}_s)\mathrm{~d}B_s^{\mathsf{T}}\right\|_1^2\right)\\
	&\leq2\norm{m}_1^2\displaystyle\int_0^T\int_0^T \mathbb{E}\left(\norm{X_s-\mathcal{X}_s}_1\norm{X_u-\mathcal{X}_u}_1\right)\mathrm{d}s\mathrm{~d}u\\
	&\quad+2\norm{\tilde{\rho}}_1^2\displaystyle\int_0^T\mathbb{E}\left(Y_s \left\|X_s-\mathcal{X}_s\right\|_1^2\right)\mathrm{d}s\\
	&\leq2\norm{m}_1^2\displaystyle\int_0^T\int_0^T\sqrt{ \mathbb{E}\left(\norm{X_s-\mathcal{X}_s}_1^2\right)\mathbb{E}\left(\norm{X_u-\mathcal{X}_u}_1^2\right)}\mathrm{~d}s\mathrm{~d}u\\
	&\quad+2\norm{\tilde{\rho}}_1^2\displaystyle\int_0^T\sqrt{\mathbb{E}\left(Y_s^2\right) \mathbb{E}\left(\left\|X_s-\mathcal{X}_s\right\|_1^4\right)}\mathrm{d}s\\
	&\leq C\left( \displaystyle\int_0^T\int_0^T\sqrt{T^2}\mathrm{~d}s\mathrm{~d}u+\displaystyle\int_0^T\sqrt{T^5}\mathrm{~d}s\right)\leq C T^{\frac{7}{2}}.
\end{align*}
On the other hand, using equation (6.12) in \cite{Boylog}, we have $\mathbb{E}\left(\norm{\mathcal{X}_s}_1^4\right)=O(T^4)$, then we get
\begin{align*}
	\mathbb{E}\left(\left\|\displaystyle\int_0^T \mathcal{X}_s\mathrm{~d}(X_s-\mathcal{X}_s)^{\mathsf{T}} \right\|_1^2\right)&\leq \norm{\tilde{\rho}}_1^2\mathbb{E}\left(\displaystyle\int_0^T (\sqrt{Y_s}-\sqrt{\mathcal{Y}_s})^2\norm{\mathcal{X}_s}_1^2\mathrm{~d}s\right)\leq \norm{\tilde{\rho}}_1^2\displaystyle\int_0^T \mathbb{E}\left((Y_s-\mathcal{Y}_s)\norm{\mathcal{X}_s}_1^2\right)\mathrm{~d}s\\
	&\leq \norm{\tilde{\rho}}_1^2\displaystyle\int_0^T \sqrt{\mathbb{E}\left((Y_s-\mathcal{Y}_s)^2\right)\mathbb{E}\left(\norm{\mathcal{X}_s}_1^4\right)}\mathrm{~d}s\leq C\displaystyle\int_0^T \sqrt{T^5}\mathrm{~d}s=C T^{\frac{7}{2}}.
\end{align*}
Hence, by combining the above two results, we deduce that
\begin{align*}
	\mathbb{E}\left(\dfrac{1}{T^2}\left\|\displaystyle\int_0^T X_s\mathrm{~d}X_s^{\mathsf{T}}-\displaystyle\int_0^T \mathcal{X}_s\mathrm{~d}\mathcal{X}_s^{\mathsf{T}}\right\|_1\right)\leq C \dfrac{T^{\frac{7}{4}}}{T^2}\to 0,\quad\text{as }T\to\infty,
\end{align*}
which completes the proof.
\end{proof}
\subsection{A special supercritical case} \label{A special supercritical case continuous observations}
We recall that, based on the classification given in Proposition \ref{classification}, the $\mathit{AD}(1,n)$ model is set to be in the supercritical case if $b\wedge\lambda_{\max}(\theta)\in\R_{--}$. In the following, we consider the special critical case when $b\in(\lambda_{\max}(\theta),0)$ under the additional assumption 
$\text{diag}(P^{-1}m)P^{-1}\kappa\in\R_-^n$, where $P$ is a modal matrix transforming $\theta$ to the diagonal matrix $D$, $\theta=PDP^{-1}$. Note that the last assumption is used in the proof of relation \eqref{convergence exp int Y X X} which will be needed in the proof of the following Theorem and in the particular case $n=1$, it becomes $m\kappa\in\R_{-}$ which corresponds to the conditon given in \cite[7.3 Theorem]{Boylog}.

\begin{theoreme}
	Let $a\in\R_{+}$, $b\in\R_{--}$ $m\in\R^n$, $\kappa\in\R^n$ and $\theta\in\mathcal{M}_n$ a diagonalizable negative definite matrix such that $b\in(\lambda_{\max}(\theta),0)$ and $\text{diag}(P^{-1}m)P^{-1}\kappa\in\R_-^n$. Let $(Z_t)_{t\in\R_+}$ be the unique strong solution of the SDE $\eqref{model0 Z}$ with initial random values $Z_0=(Y_0,X_0)^\mathsf{T}$ independent of $(B_t)_{t\in\R_{+}}$ satisfying $\mathbb{P}(Y_0\in\R_{++})=1$. Suppose that $\text{diag}(P^{-1}m)P^{-1}\kappa\in\R_-^n$. Then
\begin{equation}
	Q_T(\hat{\tau}_T-\tau)\stackrel{\mathcal{D}}{\longrightarrow}V^{-1}\eta\xi,\quad \text{as } T\to \infty,
	\label{distribution convergence special supercritical case}
\end{equation}
where $Q_T=\text{diag}\left(Te^{\frac{bT}{2}},e^{-\frac{bT}{2}},\mathbf{I}_n\otimes \tilde{Q}_T
\right)\in\mathcal{M}_{d^2+1}$, with $\tilde{Q}_T=\text{diag}\left(Te^{\frac{bT}{2}}, e^{-\frac{bT}{2}},e^{\frac{(b-2\lambda_{\min}(\theta))}{2}T}\mathbf{I}_n\right)$,\\ $V=\text{diag}(V_1,\mathbf{I}_n\otimes V_2)$, with
\begin{equation*}
V_1=\begin{bmatrix}
	1&\dfrac{\mathrm{C}_1}{b}\\[7pt]
	0&-\dfrac{\mathrm{C}_1^2}{2b}
\end{bmatrix}\quad \text{and}\quad V_2=\begin{bmatrix}
1&\dfrac{\mathrm{C}_1}{b}&\dfrac{1}{\lambda_{\min}(\theta)}\mathrm{C}_J^{\mathsf{T}}(\theta)\\[7pt]
0&-\dfrac{\mathrm{C}_1^2}{2b}&-\dfrac{\mathrm{C}_1}{b+\lambda_{\min}(\theta)}\mathrm{C}_J^{\mathsf{T}}(\theta)\\[7pt]
\mathbf{0}_n&-\dfrac{\mathrm{C}_1}{b+\lambda_{\min}(\theta)}\mathrm{C}_J(\theta)&-\dfrac{1}{2\lambda_{\min}(\theta)}\mathrm{C}_J(\theta)\mathrm{C}_J^{\mathsf{T}}(\theta)
\end{bmatrix}
\end{equation*}
where $\mathrm{C}_1$ and $\mathrm{C}_J(\theta)$ are the almost sure limits of $e^{bt}Y_t$ and $e^{\lambda_{\min}(\theta)t}X_t$, respectively, as $T\to\infty$ and $\xi$ is a $d^2+1$-dimensional standard normally distributed random vector independent of $\eta\in\mathcal{M}_{d^2+1}$ which is defined as follows
\begin{equation}
	 \eta\eta^{\mathsf{T}}=\begin{bmatrix}
	 	\sigma_1^2\mathcal{C}_1&\sigma_1(\rho_{J1}\otimes\mathcal{C}_{2})\\\sigma_1(\rho_{J1}\otimes\mathcal{C}_{2})^{\mathsf{T}}&\tilde{\rho}\tilde{\rho}^{\mathsf{T}}\otimes\mathcal{C}_{3}
	 \end{bmatrix},
 \label{eta eta^T}
 \end{equation}
with $\mathcal{C}_{1}=\begin{bmatrix}
-\dfrac{\mathrm{C}_1}{b}&\dfrac{\mathrm{C}_1^2}{2b}\\[5pt]
\dfrac{\mathrm{C}_1^2}{2b}&-\dfrac{\mathrm{C}_1^3}{3b}
\end{bmatrix}$, $\mathcal{C}_{2}=\begin{bmatrix}-\dfrac{\mathrm{C}_1}{b}&\dfrac{\mathrm{C}_1^2}{2b}&\dfrac{\mathrm{C}_1}{b+\lambda_{\min}(\theta)}\mathrm{C}_J^{\mathsf{T}}(\theta)\\[5pt]
\dfrac{\mathrm{C}_1^2}{2b}&-\dfrac{\mathrm{C}_1^3}{3b}&-\dfrac{\mathrm{C}_1^2}{2b+\lambda_{\min}(\theta)}\mathrm{C}_J^{\mathsf{T}}(\theta)
\end{bmatrix}$ and
$$\mathcal{C}_{3}=\begin{bmatrix}
-\dfrac{\mathrm{C}_1}{b}&\dfrac{\mathrm{C}_1^2}{2b}&\dfrac{\mathrm{C}_1}{b+\lambda_{\min}(\theta)}\mathrm{C}_J^{\mathsf{T}}(\theta)\\[5pt]
\dfrac{\mathrm{C}_1^2}{2b}&-\dfrac{\mathrm{C}_1^3}{3b}&-\dfrac{\mathrm{C}_1^2}{2b+\lambda_{\min}(\theta)}\mathrm{C}_J^{\mathsf{T}}(\theta)\\[5pt]
\dfrac{\mathrm{C}_1}{b+\lambda_{\min}(\theta)}\mathrm{C}_J(\theta)&-\dfrac{\mathrm{C}_1^2}{2b+\lambda_{\min}(\theta)}\mathrm{C}_J(\theta)&-\dfrac{\mathrm{C}_1}{b+2\lambda_{\min}(\theta)}\mathrm{C}_J(\theta) \mathrm{C}_J^{\mathsf{T}}(\theta)
\end{bmatrix}.$$
\label{normality theorem supercritical case}
\end{theoreme}
\begin{proof}
	We have by equation \eqref{err}
	\begin{equation*}
			Q_T(\hat{\tau}_T-\tau)=Q_T\begin{bmatrix}
				(G_T^{(1)})^{-1}h_T^{(1)}\\
				(I_{n}\otimes G_T^{(2)})^{-1}\text{vec}(h_T^{(2)})
			\end{bmatrix}=(Q_T G_T^{-1}R_T^{-1}) (R_Th_T),
	\end{equation*}
where $R_T=\text{diag}(e^{\frac{bT}{2}},e^{\frac{3bT}{2}},\mathbf{I}_n \otimes \text{diag}(e^{\frac{bT}{2}},e^{\frac{3bT}{2}},e^{\frac{(b+2\lambda_{\min}(\theta))T}{2}}\mathbf{I}_n))\in\mathcal{M}_{d^2+1}$. Note that by \cite[page 19]{Dahbi}, there exist random variables $\mathrm{C}_1,\mathrm{C}_2(\theta),\ldots,\mathrm{C}_d(\theta)$ such that $e^{bT}Z_T^1 \stackrel{a.s}{\longrightarrow} \mathrm{C}_1$ and $e^{\lambda_{\min}(\theta)T}Z_T^i\stackrel{a.s}{\longrightarrow} \mathrm{C}_i(\theta)$, as $T\to\infty$, for all $i\in J=\lbrace 2,\ldots,d\rbrace$. In a similar way as in \cite[page 19]{Dahbi},  using the Kronecker lemma (see Lemma 5.1 in \cite{Dahbi}),  it easy to check, for all $i,j\in J$ and $k,\ell_i,\ell_j\in\N$, that  
	\begin{equation}
	\begin{aligned}
		e^{(kb+(\ell_i+\ell_j)\lambda_{\min}(\theta)) T}\displaystyle\int_0^T (Z_s^1)^k (Z_s^i)^{\ell_i}(Z_s^j)^{\ell_j} \mathrm{~d}s\stackrel{a.s.}{\longrightarrow} -\dfrac{\mathrm{C}_1^k \mathrm{C}_i^{\ell_i}(\theta)\mathrm{C}_j^{\ell_j}(\theta)}{kb+(\ell_i+\ell_j)\lambda_{\min}(\theta)},\quad\text{ as}\;\; T\to\infty.
	\end{aligned}
\label{convergence exp int Y X X}
\end{equation}
Hence, using relation \eqref{convergence exp int Y X X} and the explicit form of the quadratic variation of the martingale $h_T$ given by 
$\langle\, h\,\rangle_T=\mathrm{H}_T=\begin{bmatrix}
	\mathrm{H}_T^{(1)}&\mathrm{H}_T^{(2)}\\
	{\mathrm{H}_T^{(2)}}^{\mathsf{T}}&\mathrm{H}_T^{(3)}
\end{bmatrix}$ introduced in the proof of Theorem \ref{consistency CLSE b>0}, it is easy to check by simple calculation that $R_T \langle \,h\,\rangle_T R_T^{\mathsf{T}}\stackrel{a.s.}{\longrightarrow}\eta\eta^{\mathsf{T}}$. Therefore, thanks to Theorem \ref{CLT Van Zanten} (see Appendix B), we deduce that for each random matrix $A$, we have $(R_T h_T,A) \stackrel{\mathcal{D}}{\longrightarrow}(\eta\xi,A)$, where $\xi$ is a $d^2+1$-dimensional standard normally distributed random vector independent of $(\eta,A)$. Furthermore, it is easy to check that 
$$
	Q_T G_T^{-1}R_T^{-1}=\left(\text{diag}
		(\mathcal{J}_T^{(1)},\mathbf{I}_n\otimes \mathcal{J}_T^{(2)})\right)^{-1}:=\mathcal{J}_T^{-1},
\text{ where } \mathcal{J}_T^{(1)}=\begin{bmatrix}
1&-e^{bT}\displaystyle\int_0^TY_s\mathrm{~d}s\\[7pt]
-\dfrac{e^{bT}}{T}\displaystyle\int_0^TY_s\mathrm{~d}s&e^{2bT}\displaystyle\int_0^TY_s^2\mathrm{~d}s
\end{bmatrix}$$ $$\text{ and }\mathcal{J}_T^{(2)}=\begin{bmatrix}
1&-e^{bT}\displaystyle\int_0^TY_s\mathrm{~d}s&-e^{\lambda_{\min}(\theta)T}\displaystyle\int_0^TX_s^{\mathsf{T}}\mathrm{~d}s\\[7pt]
-\dfrac{e^{bT}}{T}\displaystyle\int_0^TY_s\mathrm{~d}s&e^{2bT}\displaystyle\int_0^TY_s^2\mathrm{~d}s&e^{(b+\lambda_{\min}(\theta))T}\displaystyle\int_0^TY_sX_s^{\mathsf{T}}\mathrm{~d}s\\[7pt]
-\dfrac{e^{\lambda_{\min}(\theta)T}}{T}\displaystyle\int_0^TX_s\mathrm{~d}s&e^{(b+\lambda_{\min}(\theta))T}\displaystyle\int_0^TY_sX_s\mathrm{~d}s&e^{2\lambda_{\min}(\theta)T}\displaystyle\int_0^TX_sX_s^{\mathsf{T}}\mathrm{~d}s
\end{bmatrix}.$$
Consequently, by relation \eqref{convergence exp int Y X X}, we deduce that $Q_T G_T^{-1}R_T^{-1}\stackrel{\mathbb{P}}{\longrightarrow}\left(\text{diag}(V_1,\mathbf{I}_n\otimes V_2)\right)^{-1}=V^{-1}$, as $T\to\infty$, where $V_1$ and $V_2$ are the a.s. limits of $\mathcal{J}_T^{(1)}$ and $\mathcal{J}_T^{(2)}$, respectively. In addition, the matrix $V$ is almost surely invertible since, on the one hand, we have $\det(V_1)=-\dfrac{\mathrm{C}_1^2}{2b}>0,$ almost surely and on the other hand, we have $\det(V_2)=\det(\tilde{V}_2)$, where $$\tilde{V}_2:=\begin{bmatrix}
-\dfrac{\mathrm{C}_1^2}{2b}&-\dfrac{\mathrm{C}_1}{b+\lambda_{\min}(\theta)}\mathrm{C}_J^{\mathsf{T}}(\theta)\\
-\dfrac{\mathrm{C}_1}{b+\lambda_{\min}(\theta)}\mathrm{C}_J(\theta)&-\dfrac{1}{2\lambda_{\min}(\theta)}\mathrm{C}_J(\theta) \mathrm{C}_J^{\mathsf{T}}(\theta)
\end{bmatrix}$$ and for all $z\in\R^d\setminus\lbrace \mathbf{0}_d\rbrace$, we get
$z^{\mathsf{T}}
 \tilde{V}_2 z=-z_1^2\dfrac{\mathrm{C}_1^2}{2b}-2z_1\dfrac{\mathrm{C}_1}{b+\lambda_{\min}(\theta)}z_J^{\mathsf{T}}\mathrm{C}_J(\theta)-\dfrac{1}{2\lambda_{\min}(\theta)}\left(z_J^{\mathsf{T}}\mathrm{C}_J(\theta)\right)^2$.
Hence, using the fact that $\lambda_{\min}(\theta)<b$, we deduce that $z^{\mathsf{T}}\tilde{V_2}z>-\dfrac{1}{b+\lambda_{\min}(\theta)}\left(z_1\mathrm{C}_1+z_J^{\mathsf{T}}\mathrm{C}_J\right)^2>0$, almost surely, thus we obtain $\det(V)=\det(V_1)\times(\det(V_2))^n>0$ almost surely. Finally, using $(R_T h_T,A) \stackrel{\mathcal{D}}{\longrightarrow}(\eta\xi,A)$ for all random matrix $A$, it is sufficient to take $A=V$ and the proof is completed. 
\end{proof}
\begin{remarque}
  Thanks to Theorem \ref{normality theorem supercritical case}, we can mention that in the special supercritical case, the deduced CLSE estimators $\hat{a}_T$ and $\hat{m}_T$ are not even weakly consistent and that their error terms have an exponential expansion. However, we have the strong consistency of $\hat{b}_T$ and the weak consistency of $(\hat{\kappa}_T,\hat{\theta}_T)$.    
\end{remarque}
\section{CLSE based on discrete time observations}
We consider $(Y_t,X_t)_{t\in\R_+}$ the $\mathit{AD}(1,n)$ process strong solution of the SDE $\eqref{model01}$ with random initial value $(Y_0,X_0)^{\mathsf{T}}$ independent of $(B_t)_{t\in\R_+}$ satisfying $\mathbb{P}(Y_0\in\R_{++})=1$. The following discussion is about the construction of a CLSE for the drift parameter $\tau$ based on discrete time observations $(Y_{t_k},X_{t_k})_{k\in\lbrace 0,1,\ldots, N\rbrace}$, for $N\in\N\setminus\lbrace 0\rbrace$, with $t_k=k\Delta_N$ such that $\Delta_N\to 0$ and $N\Delta_N=t_N\to\infty$, as $N\to\infty$. The construction of the discrete CLSE follows the same steps as those in the continuous setting (see section 3). Namely, we consider first the following extremum problem
\begin{equation}
	\begin{aligned}
		\underset{\tau\in \R^{d^2+1}}{\arg\min}\displaystyle\sum_{k=1}^{N}\left[\left(Y_{t_k}-\mathbb{E}\left( Y_{t_k}\vert \mathcal{F}_{t_{k-1}}\right)\right)^2+\left(X_{t_k}-\mathbb{E}\left( X_{t_k}\vert \mathcal{F}_{t_{k-1}}\right)\right)^{\mathsf{T}}\left(X_{t_k}-\mathbb{E}\left(X_{t_k}\vert \mathcal{F}_{t_{k-1}}\right)\right)\right]
	\end{aligned}
	\label{extremum problem discrete}
\end{equation}
and using the explicit expressions of conditional expectations with some Taylor expansions, we introduce the  CLSE estimator $\check{\tau}_{t_N}^{\Delta_N}=(\check{a}_{t_N}^{\Delta_N},\check{b}_{t_N}^{\Delta_N},\text{vec}([\check{m}_{t_N}^{\Delta_N},\check{\kappa}_{t_N}^{\Delta_N},\check{\theta}_{t_N}^{\Delta_N}]^{\mathsf{T}}))^{\mathsf{T}}$ equal to
\begin{multline*}
		\dfrac{1}{\Delta_N}\;\underset{\check{\tau}\in \R^{d^2+1}}{\arg\min}\displaystyle\sum_{k=1}^{N}\left[\left(Y_{t_k}-Y_{t_{k-1}}-\left( \check{a}- \check{b}Y_{t_{k-1}} \right)\right)^2+\left(X_{t_k}-X_{t_{k-1}}-\left( \check{m} -Y_{t_{k-1}} \check{\kappa}-\check{\theta} X_{t_{k-1}}  \right)\right)^{\mathsf{T}}\right.\\
		 \left.\times\left(X_{t_k}-X_{t_{k-1}}-\left( \check{m} -Y_{t_{k-1}} \check{\kappa}-\check{\theta} X_{t_{k-1}}  \right)\right)\right],
	\end{multline*}
with 
$\check{\tau}:=(\check{a},\check{b},\text{vec}([\check{m},\check{\kappa},\check{\theta}]^{\mathsf{T}}))^{\mathsf{T}}$. Then, it can be observed that, on the one hand, we have
\begin{equation}
(\check{a}_{t_N}^{\Delta_N},\check{b}_{t_N}^{\Delta_N})^{\mathsf{T}}=\dfrac{1}{\Delta_N}\underset{(\check{a},\check{b})^{\mathsf{T}}\in \R^2}{\arg\min}\displaystyle\sum_{k=1}^{N}\left(Y_{t_k}-Y_{t_{k-1}}-\left( \check{a}- \check{b}Y_{t_{k-1}} \right)\right)^2,
	\label{c d discrete}
\end{equation}
and on the other hand, we have
\begin{equation}
	\begin{aligned}
		\text{vec}([\check{m}_{t_N}^{\Delta_N},\check{\kappa}_{t_N}^{\Delta_N},\check{\theta}_{t_N}^{\Delta_N}]^{\mathsf{T}})=\dfrac{1}{\Delta_N}\;\underset{\text{vec}([{\check{m}},{\check{\kappa}},{\check{\theta}}]^{\mathsf{T}})\in \R^{d^2-1}}{\arg\min}\displaystyle\sum_{k=1}^{N}&\left(X_{t_k}-X_{t_{k-1}}-\left(\check{m} -Y_{t_{k-1}} \check{\kappa}-\check{\theta} X_{t_{k-1}}  \right)\right)^{\mathsf{T}}\\
		&\times\left(X_{t_k}-X_{t_{k-1}}-\left(\check{m} -Y_{t_{k-1}} \check{\kappa}-\check{\theta} X_{t_{k-1}}  \right)\right).
	\end{aligned}
	\label{delta varepsilon zeta discrete}
\end{equation}
Hence, in order solve the extremum problem, we proceed with similar calculation as in the previous section (see relations \eqref{derivative 1 equals zero} and \eqref{derivative 2 equals 0}) and we obtain
\begin{equation*}
	\left(\check{a}_{t_N}^{\Delta_N},\check{b}_{t_N}^{\Delta_N}
	\right)^{\mathsf{T}}=\left(\Delta_N\check{\Gamma}^{(1)}_{t_N,\Delta_N}\right)^{-1}\check{\phi}^{(1)}_{t_N,\Delta_N}.
\end{equation*}
and
\begin{equation*}
	\left(\check{m}_{t_N}^{\Delta_N},\check{\kappa}_{t_N}^{\Delta_N},\check{\theta}_{t_N}^{\Delta_N}
\right)^{\mathsf{T}}=\left(\Delta_N\check{\Gamma}^{(2)}_{t_N,\Delta_N}\right)^{-1}\check{\phi}^{(2)}_{t_N,\Delta_N},
\end{equation*}
where 
			\begin{equation*}
				\check{\Gamma}^{(1)}_{t_N,\Delta_N}=\begin{bmatrix}
					N& -\displaystyle\sum_{k=1}^{N}Y_{t_{k-1}}\\
					-\displaystyle\sum_{k=1}^{N}Y_{t_{k-1}}& \displaystyle\sum_{k=1}^{N}Y_{t_{k-1}}^2
				\end{bmatrix},\quad \check{\phi}^{(1)}_{t_N,\Delta_N}=\begin{bmatrix}
					Y_{t_N}-Y_{0}\\
					-\displaystyle\sum_{k=1}^{N} \left(Y_{t_k}-Y_{t_{k-1}}\right)Y_{t_{k-1}}
				\end{bmatrix}
			\end{equation*}
   and
   	\begin{equation*}
				\check{\Gamma}^{(2)}_{t_N,\Delta_N}=\begin{bmatrix}
					N & -\displaystyle\sum_{k=1}^{N}Y_{t_{k-1}}  &-\displaystyle\sum_{k=1}^{N}X_{t_{k-1}}^{\mathsf{T}}\\
					-\displaystyle\sum_{k=1}^{N}Y_{t_{k-1}} & \displaystyle\sum_{k=1}^{N}Y_{t_{k-1}}^2 & \displaystyle\sum_{k=1}^{N} Y_{t_{k-1}}X_{t_{k-1}}^{\mathsf{T}}\\
					-\displaystyle\sum_{k=1}^{N}X_{t_{k-1}} &  \displaystyle\sum_{k=1}^{N} Y_{t_{k-1}}X_{t_{k-1}} & \displaystyle\sum_{k=1}^{N} X_{t_{k-1}}X_{t_{k-1}}^{\mathsf{T}}
				\end{bmatrix},\quad
				\check{\phi}^{(2)}_{t_N,\Delta_N}=\begin{bmatrix}
					X_{t_N}^{\mathsf{T}}-X_{0}^{\mathsf{T}}\\
					-\displaystyle\sum_{k=1}^{N} Y_{t_{k-1}}\left(X_{t_k}^{\mathsf{T}}-X_{t_{k-1}}^{\mathsf{T}}\right)\\
					-\displaystyle\sum_{k=1}^{N}X_{t_{k-1}}\left(X_{t_k}^{\mathsf{T}}-X_{t_{k-1}}^{\mathsf{T}}\right)
				\end{bmatrix}.
			\end{equation*}
Consequently, the CLSE is given by
\begin{equation}
\check{\tau}_{t_N}^{\Delta_N}=\begin{bmatrix}
	\check{a}_{t_N}^{\Delta_N}\\[4pt]
	\check{b}_{t_N}^{\Delta_N}\\[4pt]
	\text{vec} \begin{bmatrix}
		\left(\check{m}_{t_N}^{\Delta_N}\right)^\mathsf{T}\\ \left(\check{\kappa}_{t_N}^{\Delta_N}\right)^\mathsf{T}\\ \left(\check{\theta}_{t_N}^{\Delta_N}\right)^\mathsf{T}
\end{bmatrix} \end{bmatrix}=\begin{bmatrix}\left(\Delta_N\check{\Gamma}_{t_N,\Delta_N}^{(1)}\right)^{-1}\check{\phi}_{t_N,\Delta_N}^{(1)}\\[4pt]
		(I_{n}\otimes \Delta_N\check{\Gamma}_{t_N,\Delta_N}^{(2)})^{-1}\text{vec}(\check{\phi}_{t_N,\Delta_N}^{(2)})
	\end{bmatrix}=\left(\Delta_N\check{\Gamma}_{t_N,\Delta_N}\right)^{-1}\check{\phi}_{t_N,\Delta_N},
 \label{estimator discrete}
\end{equation}
where $\check{\Gamma}_{t_N,\Delta_N}:=\begin{bmatrix}
	\check{\Gamma}_{t_N,\Delta_N}^{(1)}& \textbf{0}\\
	\textbf{0}& I_{n}\otimes \check{\Gamma}_{t_N,\Delta_N}^{(2)}
\end{bmatrix}$ and $\check{\phi}_{t_N,\Delta_N}:=\begin{bmatrix}
\check{\phi}_{t_N,\Delta_N}^{(1)}\\
\text{vec}(\check{\phi}_{t_N,\Delta_N}^{(2)})
\end{bmatrix}$.

Hence, the error term relative to the discrete CLSE can be written as follows 
\begin{equation}\label{discrete-continuous+continuous-exact}
    \check{\tau}_{t_N}^{\Delta_N}-\tau=(\check{\tau}_{t_N}^{\Delta_N}-\hat{\tau}_{t_N})+(\hat{\tau}_{t_N}-\tau).
\end{equation}
	Since we have already treated in section \ref{Continuous section} the asymptotic behavior of $\hat{\tau}_{t_N}-\tau$, as $N\to \infty$, in the subcritical, special critical and special supercritical cases, then in order to study the asymptotic properties of discrete CLSE $\check{\tau}_{t_N}^{\Delta_N}$, we consider only the first term  $\check{\tau}_{t_N}^{\Delta_N}-\hat{\tau}_{t_N}$, as $N\to\infty$. More precisely, we will prove that this term normalized with the correspondent rate of convergence tends to zero in probability. In order to that, we need to establish some probabilistic results on the model.   
 
\subsection{Stochastic analysis for the $\mathit{AD}(1,n)$ process}
Next, we give a technical proposition about the moment of $Z_t-Z_s$ when $t-s\in(0,1)$ in the subcritical, special critical and special supercritical cases which will be used later for the proofs of propositions \ref{Thm lien entre intégrale et somme pour b>0}, \ref{Thm relation sum with integral SCC} and \ref{Thm relation sum with integral SSC}. In the following $C$ denotes a positive constant independent on $t$ and $s$ that may change its value from line to line.
\begin{lemme}
Let $s\in[0,t]$ such that $t-s\in(0,1)$. Then, 
\begin{itemize}
\item 	In the subcritical case $b\in\R_{++}$ and $\theta$ is a positive definite matrix, for all $q\geq 1$, we have

	\begin{equation}
		\mathbb{E}\left(\norm{Z_t-Z_s}_1^q\right)\leq C(t-s)^{\frac{q}{2}}.
		\label{upperbound expectation b>0}
	\end{equation}

\item In the special critical case $b=0$, $\kappa=\mathbf{0}_n$ $\theta=\mathbf{0}_{n,n}$, we have
\begin{equation}
	\mathbb{E}\left(\norm{Z_t-Z_s}_1^q\right)\leq C(t-s)^{\frac{q}{2}}(1+t^{\frac{q}{2}}).
	\label{int Zs - sum Zk SCC}
\end{equation}
\item In the special supercritical case $b\in\R_{--}$ and $\lambda_{\max}(\theta)<b$, for all $q\geq 1$, we have
\begin{equation}\label{upperbound expectation SSC}
\mathbb{E}\left((Y_t-Y_s)^q\right)\leq C(t-s)^{\frac{q}{2}}e^{-qbt}\quad\text{and}\quad    \mathbb{E}\left(\norm{X_t-X_s}_1^q\right)\leq C(t-s)^{\frac{q}{2}}e^{-q\lambda_{\min}(\theta)t}
\end{equation}
\end{itemize}
\end{lemme}
\begin{proof}
 First, by the power mean inequality and \cite[Proposition 4]{Alaya}, we have
   \begin{equation*}
   \begin{aligned}
   	 	\label{E(Z) decomposition}\mathbb{E}\left(\norm{Z_t-Z_s}_1^q\right)&\leq 2^{q-1} 	\mathbb{E}\left(\vert Y_t-Y_s\vert ^q\right)+2^{q-1} 	\mathbb{E}\left(\norm{X_t-X_s}_1^q\right)\\
   	 	&\leq C (t-s)^{\frac{q}{2}}+2^{q-1} 	\mathbb{E}\left(\norm{X_t-X_s}_1^q\right).
      \end{aligned}
   \end{equation*}
Using the relation \eqref{Xt Xs expression}, we obtain
\begin{equation*}
	X_t-X_s=(I_n-e^{-\theta(t-s)})(\theta^{-1}m-X_s)-\int_s^tY_ue^{-\theta(t-u)}\kappa\mathrm{~d}u+\int_s^t\sqrt{Y_u}e^{-\theta(t-u)}\tilde{\rho}\mathrm{~d}B_u.
\end{equation*}
Hence, by the power mean inequality and the norm inequalities $\norm{Ax}_1\leq\norm{A}_1\norm{x}_1$ and $\norm{A}_1\leq\sqrt{n}\norm{A}_2$, for all $A\in\mathcal{M}_n$ and $x\in\R^n$, there exists a positive constant $C$ such that
\begin{align*}
   \mathbb{E}\left(\norm{X_t-X_s}_1^q\right)&\leq C\left((1-e^{-\lambda_{\max}(\theta)(t-s)})^q\,\mathbb{E}(\norm{\theta^{-1}m-X_s}_1^q)+\mathbb{E}\left(\norm{\int_{s}^{t}Y_ue^{-\theta(t-u)}\kappa\mathrm{~d}u}_1^q\right)\right.\\
   &\quad \left.+\mathbb{E}\left(\norm{\int_s^t \sqrt{Y_u} e^{-\theta(t-u)}\tilde{\rho}\mathrm{~d}B_u}_1^q\right)\right).
\end{align*}
Using the following relation $1-e^{-x}\leq x$, for all $x\in\R$ and the fact that $\sup\limits_{t\in\R_+}\mathbb{E}(\vert X^i_t\vert^q)$ is finite by Proposition \ref{moment result on time t}, for all $i\in\lbrace1,\ldots,n\rbrace$, we conclude that the first term in the right-hand side of the above equation is bounded by $C(t-s)^q$. For the second term, using the same arguments then the Hölder inequality on the integral and Fubini's theorem, since we have $\sup\limits_{t\in\R_+}\mathbb{E}(Y_t^q)<\infty$, we get
\begin{align*}
	\mathbb{E}\left(\norm{\int_{s}^{t}Y_ue^{-\theta(t-u)}\kappa\mathrm{~d}u}_1^q\right)&\leq C (t-s)^{q-1} \left(\int_{s}^{t}\mathbb{E}(Y_u^q)\,e^{-q\lambda_{\min}(\theta)(t-u)}\mathrm{~d}u\right)\leq C (t-s)^q.
\end{align*}
For the last term, we start with the Burkholder-Davis-Gundy inequality to deduce similarly that
\begin{align*}	\mathbb{E}\left(\norm{\int_s^t\sqrt{Y_u} e^{-\theta(t-u)}\tilde{\rho}\mathrm{~d}B_u}_1^q\right)&\leq C\,\mathbb{E}\left(\norm{\int_s^tY_u e^{-\theta(t-u)}\tilde{\rho}\tilde{\rho}^{\mathsf{T}}e^{-\theta^{\mathsf{T}}(t-u)}\mathrm{~d}u}_1^{\frac{q}{2}}\right)\\
	&\leq C\, \mathbb{E}\left(\left(\int_s^tY_u e^{-2\lambda_{\min}(\theta)(t-u)} \mathrm{~d}u\right)^{\frac{q}{2}}\right).
\end{align*}
In addition, similarly as in the proof of \cite[Proposition 4]{Alaya}, in order to give an upper-bound for the above expectation, we need to distinguish the two cases $1\leq q< 2$ and $q\geq 2$ where we apply the Hölder inequality on the expectation and on the integral, respectively. In the both cases, we get 
\begin{align*}	\mathbb{E}\left(\norm{\int_s^t\sqrt{Y_u} e^{-\theta(t-u)}\tilde{\rho}\mathrm{~d}B_u}_1^q\right)&\leq C(t-s)^{\frac{q}{2}},
\end{align*}
which completes the first assertion. For the second assertion, thanks to \cite[Proposition 5]{Alaya} and to the first inequality in \eqref{E(Z) decomposition}, it is enough to prove the property \eqref{int Zs - sum Zk SCC} for the process $X$. First, we write
	\begin{equation*}
		X_t-X_s=(t-s)m+\int_s^t\sqrt{Y_u}\tilde{\rho}\mathrm{~d}B_u.
	\end{equation*}
and we get
$\mathbb{E}(\norm{X_t-X_s}_1^q)\leq C\left((t-s)^q+\mathbb{E}\left(\norm{\int_s^t\sqrt{Y_u}\tilde{\rho}\mathrm{~d}B_u}_1^q\right)\right)$ then we apply the Burkholder-Davis-Gundy inequality to obtain
$\mathbb{E}\left(\norm{\int_s^t\sqrt{Y_u}\tilde{\rho}\mathrm{~d}B_u}_1^q\right)\leq C\,\mathbb{E}\left(\left(\int_s^t Y_u\mathrm{~d}u\right)^{\frac{q}{2}}\right)$.
Consequently, for $q\geq 2$, by applying the Hölder inequality on the integral and the first assertion of Proposition 3 in \cite{Alaya}, we obtain
\begin{equation*}
	\mathbb{E}\left(\norm{\int_s^t\sqrt{Y_u}\tilde{\rho}\mathrm{~d}B_u}_1^q\right)\leq C(t-s)^{\frac{q}{2}}\sup\limits_{s\leq u\leq t}\mathbb{E}(Y_u^{\frac{q}{2}})\leq C(t-s)^{\frac{q}{2}}(1+t^{\frac{q}{2}}).
\end{equation*}
For $1\leq q< 2$, we apply the Hölder inequality on the expectation with $\mathbb{E}(Y_u)=au+\mathbb{E}(Y_0)$ and we obtain
\begin{equation*}
	\mathbb{E}\left(\norm{\int_s^t\sqrt{Y_u}\tilde{\rho}\mathrm{~d}B_u}_1^q\right)\leq C\left(\int_s^t \mathbb{E}(Y_u)\mathrm{d}u\right)^{\frac{q}{2}}\leq C(t-s)^{\frac{q}{2}}(at+\mathbb{E}(Y_0))^{\frac{q}{2}}.
\end{equation*}
For the last assertion, let us recall first that
$$Y_{t}-Y_s=\dfrac{a}{b}\left(1-e^{-b(t-s)}\right)-Y_s\left(1-e^{-b(t-s)}\right)+\sigma_1e^{-bt}\displaystyle\int_s^{t}e^{bu}\sqrt{Y_u}\mathrm{d}B_u^1.$$
Hence, for all $q\geq 1$, by similar arguments as before and using the fact that $\mathbb{E}(Y_s^q)\leq C e^{-qbs}$ and $\mathbb{E}(\norm{X_s}_1^q)\leq C e^{-q\lambda_{\min}(\theta)s}$ when $\lambda_{\max}(\theta)<b<0$ obtained easily by induction from Proposition \ref{moment result on time t}, we get
\begin{align*}
    \mathbb{E}\left((Y_{t}-Y_s)^q\right)&\leq C\left((t-s)^q+e^{-qbs}(t-s)^q+e^{-qbt}(t-s)^{\frac{q}{2}}\right).
\end{align*}
and
\begin{align*}
    \mathbb{E}\left(\norm{X_{t}-X_s}_1^q\right)&\leq C \left((t-s)^q+(t-s)^qe^{-q\lambda_{\min}(\theta)s}+(t-s)^{\frac{q}{2}}e^{-q\lambda_{\min}(\theta)t}\right)
\end{align*}
which completes the proof.
\end{proof}
The task now is to give sufficient conditions on the frequency $\Delta_N$ in order to get the same asymptotic results obtained in the continuous setting in the subcritical, special critical and special super critical cases.
\begin{proposition}\label{Thm lien entre intégrale et somme pour b>0}
	For $b\in\R_{++}$ and $\theta$ positive definite matrix, we have
	\begin{equation*}
		\dfrac{1}{t_N}\int_0^{t_N}Z_s\mathrm{~d}s-\dfrac{1}{t_N}\sum_{k=1}^{N}\Delta_NZ_{t_{k-1}}\longrightarrow \mathbf{0}_d,
	\end{equation*}
\begin{equation*}
	\dfrac{1}{t_N}\int_0^{t_N}Z_sZ_s^{\mathsf{T}}\mathrm{~d}s-\dfrac{1}{t_N}\sum_{k=1}^{N}\Delta_NZ_{t_{k-1}}Z_{t_{k-1}}^{\mathsf{T}}\longrightarrow \mathbf{0}_{d,d},
\end{equation*}	
and if $N\Delta_N^2\to 0$, as $N\to\infty$, then we have
\begin{equation*}
		\dfrac{1}{\sqrt{t_N}}\int_0^{t_N}Z_s\mathrm{~d}Z_s^{\mathsf{T}}-\dfrac{1}{\sqrt{t_N}}\sum_{k=1}^{N}Z_{t_{k-1}}\left(Z_{t_{k}}-Z_{t_{k-1}}\right)^{\mathsf{T}}\longrightarrow \mathbf{0}_{d,d},
	\end{equation*}
	in probability, as $N\to\infty$.
Consequently, we get $\left(\frac{1}{t_N}G_{t_N}^{(i)},\frac{1}{\sqrt{t_N}}f_{t_N}^{(i)}\right)-\left(\frac{\Delta_N}{t_N}\Gamma_{t_N,\Delta_N}^{(i)},\frac{1}{\sqrt{t_N}}\check{\phi}_{t_N,\Delta_N}^{(i)}\right)$ tends in probability to zero as $N\to\infty$.
\end{proposition}
\begin{proof}
By the upper-bound inequality \eqref{upperbound expectation b>0}, on one hand, we get
	\begin{multline*}
		\dfrac{1}{t_N}\mathbb{E}\left(\norm{\int_0^{t_N}Z_s\mathrm{~d}s-\sum_{k=1}^{N}\Delta_NZ_{t_{k-1}}}_1\right)\leq    \dfrac{1}{t_N}\sum_{k=1}^{N}\int_{t_{k-1}}^{t_{k}}\mathbb{E}\left(\norm{Z_s-Z_{t_{k-1}}}_1\right)\mathrm{d}s\\
		\leq \dfrac{C}{t_N}\sum_{k=1}^{N}\int_{t_{k-1}}^{t_{k}}\sqrt{s-t_{k-1}}\,\mathrm{d}s\leq \dfrac{C}{t_N}\sum_{k=1}^{N}\Delta_N^{\frac{3}{2}}=C\sqrt{\Delta_N}\to0,\quad\text{as }N\to\infty.
	\end{multline*}
On the other hand, for $b\in\R{++}$, we have $\mathbb{E}\left(\norm{Z_t}_1^2\right)<\infty$, for all $t\in\R_+$. Hence, by a simple matrix decomposition and by the Cauchy-Schwarz inequalty, we get $\frac{1}{t_N}\mathbb{E}\left(\norm{\int_0^{t_N}Z_sZ_s^{\mathsf{T}}\mathrm{~d}s-\sum_{k=1}^{N}\Delta_NZ_{t_{k-1}}Z_{t_{k-1}}^{\mathsf{T}}}_1\right)$ is less than or equal to
\begin{multline}\label{int ZZ- sum ZZ b>0}
\begin{split}
\dfrac{1}{2t_N}\sum_{k=1}^{N}\int_{t_{k-1}}^{t_{k}}\left(\mathbb{E}\left(\norm{(Z_s-Z_{t_{k-1}})(Z_s^{\mathsf{T}}+Z_{t_{k-1}}^{\mathsf{T}})}_1\right)+\mathbb{E}\left(\norm{(Z_s+Z_{t_{k-1}})(Z_s^{\mathsf{T}}-Z_{t_{k-1}}^{\mathsf{T}})}_1\right)\right)\mathrm{d}s\\
			\leq \dfrac{C}{t_N}\sum_{k=1}^{N}\int_{t_{k-1}}^{t_{k}}\left(\mathbb{E}\left(\norm{Z_s-Z_{t_{k-1}}}_1^2\right)\right)^{\frac{1}{2}}\left( \mathbb{E}\left(\norm{Z_s}_1^2\right)+\mathbb{E}\left(\norm{Z_{t_{k-1}}}_1^2\right)\right)^{\frac{1}{2}}\mathrm{d}s \\
		\leq \dfrac{C}{t_N}\sum_{k=1}^{N}\int_{t_{k-1}}^{t_{k}}\sqrt{s-t_{k-1}}\,\mathrm{d}s\leq C\sqrt{\Delta_N}\to0, \quad\text{as }N\to\infty.    
  \end{split}		
\end{multline}
For the last assertion, using the dynamic of the process $Z$, classical properties of $\mathrm{L}^1$-norm and Cauchy-Schwarz inequality, we get
\begin{align*}
\mathbb{E}&\left(\norm{\dfrac{1}{\sqrt{t_N}}\int_0^{t_N}Z_s\mathrm{~d}Z_s^{\mathsf{T}}-\dfrac{1}{\sqrt{t_N}}\sum_{k=1}^{N}Z_{t_{k-1}}\left(Z_{t_{k}}-Z_{t_{k-1}}\right)^{\mathsf{T}}}_1\right)= \mathbb{E}\left(\norm{\dfrac{1}{\sqrt{t_N}}\sum_{k=1}^N\int_{t_{k-1}}^{t_k}(Z_s-Z_{t_{k-1}})dZ_s^{\mathsf{T}}}_1\right)\\
 &\leq \dfrac{C}{\sqrt{t_N}}\sum_{k=1}^N\mathbb{E}
  \left(\int_{t_{k-1}}^{t_k}\norm{Z_s-Z_{t_{k-1}}}_1\norm{\Lambda(Z_s)}_1\mathrm{d}s\right)+C\left(\dfrac{1}{t_N}\mathbb{E}\left(\norm{\sum_{k=1}^N\int_{t_{k-1}}^{t_k}\sqrt{Y_s}(Z_s-Z_{t_{k-1}})\mathrm{d}B_s^{\mathsf{T}}}^2_1\right)\right)^{\frac{1}{2}}.
\end{align*}
For the first term, by combining the definition of $\Lambda(Z_s)$ defined after equation \eqref{model0 Z} with Remark \ref{Remark moment resulty on time t} applied in the subcritical case, it is easy to check that $\underset{s\in\R_+}{\sup}\mathbb{E}\left(\norm{\Lambda(Z_s)}_1^2\right)$ is finite, then by Cauchy-Schwarz inequality and the property \ref{upperbound expectation b>0}, we deduce that it is bounded up to a constant by $\sqrt{N}\Delta_N$.
Next, thanks to the independence of stochastic integrals defined on distinct intervals $(t_{k-1},t_k)$, $k=1,\ldots,n$, the equivalence between the $\mathrm{L}^1$ and the Frobenius norms, Itô's isometry, Cauchy-Schwarz inequality and the property \eqref{upperbound expectation b>0}, we deduce that $\mathbb{E}\left(\frac{1}{t_N}\norm{\sum_{k=1}^N\int_{t_{k-1}}^{t_k}\sqrt{Y_s}(Z_s-Z_{t_{k-1}})\mathrm{d}B_s^{\mathsf{T}}}^2_1\right)$ is bounded by 
\begin{equation}\label{proof convergence 2}
 \dfrac{C}{t_N}\sum_{k=1}^N\sum_{i=1}^d\int_{t_{k-1}}^{t_k}\left(\mathbb{E}\left(Y_s^2\right)\right)^{\frac{1}{2}}\left(\mathbb{E}\left((Z^i_s-Z_{t_{k-1}}^i)^4\right)\right)^{\frac{1}{2}}\mathrm{d}s\leq C\Delta_N,
\end{equation}
which completes the proof of the propostion.
\end{proof}
\begin{proposition}\label{Thm relation sum with integral SCC}
		For $b=0$ and $\theta=\mathbf{0}_{n,n}$, we have
	\begin{equation*}
		\dfrac{1}{t_N^2}\int_0^{t_N}Z_s\mathrm{~d}s-\dfrac{1}{t_N^2}\sum_{k=1}^{N}\Delta_N Z_{t_{k-1}}\longrightarrow \mathbf{0}_d,
	\end{equation*}
	\begin{equation*}
		\dfrac{1}{t_N^3}\int_0^{t_N}Z_sZ_s^{\mathsf{T}}\mathrm{~d}s-\dfrac{1}{t_N^3}\sum_{k=1}^{N}\Delta_NZ_{t_{k-1}}Z_{t_{k-1}}^{\mathsf{T}}\longrightarrow \mathbf{0}_{d,d},
	\end{equation*}	
and if $N\Delta_N^{2}\to 0$, as $N\to\infty$, then we have
\begin{equation*}
		\dfrac{1}{t_N^2}\int_0^{t_N}Z_s\mathrm{~d}Z_s^{\mathsf{T}}-\dfrac{1}{t_N^2}\sum_{k=1}^{N}Z_{t_{k-1}}\left(Z_{t_{k}}-Z_{t_{k-1}}\right)^{\mathsf{T}}\longrightarrow \mathbf{0}_{d,d},
	\end{equation*}
	in probability, as $N\to\infty$.
\end{proposition}
\begin{proof}
	By the upper-bound inequality \eqref{int Zs - sum Zk SCC}, we have
		\begin{align*}
		\dfrac{1}{t_N^2}\mathbb{E}\left(\norm{\int_0^{t_N}Z_s\mathrm{~d}s-\sum_{k=1}^{N}\Delta_NZ_{t_{k-1}}}_1\right)&\leq    \dfrac{1}{t_N^2}\sum_{k=1}^{N}\int_{t_{k-1}}^{t_{k}}\mathbb{E}\left(\norm{Z_s-Z_{t_{k-1}}}_1\right)\mathrm{d}s\\
	&\leq \dfrac{C\sqrt{\Delta_N}}{t_N^2}\int_{0}^{t_{N}}\left(1+\sqrt{s}\right)\mathrm{d}s\leq\dfrac{C\sqrt{\Delta_N}}{t_N^2}\left( t_N+t_N^{\frac{3}{2}}\right)\to 0,
\end{align*}
as $N\to\infty$. Next, thanks to Minkowski inequality and the induction given in Remark \ref{Remark moment resulty on time t} we get $\mathbb{E}\left(\norm{Z_{s}}_1^2\right)\leq  C (1+t_N+t_N^2)$ which allows us to control the second term for large values of $N$. In fact, the same decomposition in relation \eqref{int ZZ- sum ZZ b>0} combined with the  upper-bound inequality \eqref{int Zs - sum Zk SCC} leads to
\begin{equation*}
	\frac{1}{t_N^3}\mathbb{E}\left(\norm{\int_0^{t_N}Z_sZ_s^{\mathsf{T}}\mathrm{~d}s-\sum_{k=1}^{N}\Delta_NZ_{t_{k-1}}Z_{t_{k-1}}^{\mathsf{T}}}_1\right)\leq\dfrac{C}{t_N^3}\sum_{k=1}^{N}\int_{t_{k-1}}^{t_{k}}\sqrt{\Delta_N}\,t_N^{\frac{3}{2}}\,\mathrm{d}s
 = \dfrac{C}{\sqrt{N}}\to 0,\quad\text{as }N\to\infty.	    \end{equation*}
For the last assertion, following the proof of the subcritical case and using the above arguments for large values of $N$, we deduce firstly 
  \begin{multline}\label{b=0 proof convergence 1}
\dfrac{1}{t^2_N}\displaystyle\sum_{k=1}^N\mathbb{E}\left(\norm{\int_{t_{k-1}}^{t_k}(Z_s-Z_{t_{k-1}})\tau^{\mathsf{T}}\Lambda^{\mathsf{T}}(Z_s)\mathrm{d}s}_1\right)\leq
    \dfrac{C}{t^2_N}\sum_{k=1}^N\mathbb{E}
  \left(\int_{t_{k-1}}^{t_k}\norm{Z_s-Z_{t_{k-1}}}_1\norm{Z_s}_1\mathrm{d}s\right)\\
  \leq \dfrac{C}{t_N^2}\sum_{k=1}^N\int_{t_{k-1}}^{t_k}\left(\mathbb{E}
  \left(\norm{Z_s-Z_{t_{k-1}}}^2_{1}\right)\right)^{\frac{1}{2}}\left(\mathbb{E}\left(\norm{Z_s}^2_{1}\right)\right)^{\frac{1}{2}}\mathrm{d}s\leq C\sqrt{N}\Delta_N
\end{multline}
  and secondly $\displaystyle\dfrac{C}{t_N^4}\sum_{k=1}^N\sum_{i=1}^d\int_{t_{k-1}}^{t_k}\left(\mathbb{E}\left(Y_s^2\right)\right)^{\frac{1}{2}}\left(\mathbb{E}\left((Z^i_s-Z_{t_{k-1}}^i)^4\right)\right)^{\frac{1}{2}}\mathrm{d}s\leq \dfrac{C}{N}$.
\end{proof}
Next, in the special supercritical case, we will consider a matrix representation of the previous propositions based on the following decomposition 
\begin{align}\label{Decomposition discrete - continuous SSC}
\begin{split}
	Q_{t_N}(\check{\tau}_{t_N}^{\Delta_N}-\hat{\tau}_{t_N})&=(Q_{t_N} (\Delta_N\check{\Gamma}_{t_N,\Delta_N})^{-1}R_{t_N}^{-1}) (R_{t_N}\check{\phi}_{t_N,\Delta_N})-(Q_{t_N} G_{t_N}^{-1}R_{t_N}^{-1}) (R_{t_N}f_{t_N})\\
	&=\mathcal{J}_{t_N,\Delta_N}^{-1}R_{t_N}\check{\phi}_{t_N,\Delta_N}-\mathcal{J}_{t_N}^{-1}R_{t_N}f_{t_N}\\
	&=\left(\mathcal{J}_{t_N,\Delta_N}^{-1}-\mathcal{J}_{t_N}^{-1}\right)R_{t_N}\check{\phi}_{t_N,\Delta_N}+\mathcal{J}_{t_N}^{-1}R_{t_N}\left(\check{\phi}_{t_N,\Delta_N}-f_{t_N}\right)\\
&=e^{-\lambda_{\min}(\theta)t_N}\left(\mathcal{J}_{t_N,\Delta_N}^{-1}-\mathcal{J}_{t_N}^{-1}\right)\left(e^{\lambda_{\min}(\theta)t_N}R_{t_N}\check{\phi}_{t_N,\Delta_N}\right)+\mathcal{J}_{t_N}^{-1}R_{t_N}\left(\check{\phi}_{t_N,\Delta_N}-f_{t_N}\right)
\end{split}
\end{align}
where $Q_{t_N}$, $\mathcal{J}_{t_N}$ and $R_{t_N}$ are defined in Theorem \ref{normality theorem supercritical DISCRETE case} and
$\mathcal{J}_{t_N,\Delta_N}:=\left(\text{diag}
		(\mathcal{J}_{t_N,\Delta_N}^{(1)},\mathbf{I}_n\otimes \mathcal{J}_{t_N,\Delta_N}^{(2)})\right)$, with $\mathcal{J}_{t_N,\Delta_N}^{(1)}=\begin{bmatrix}
1&-e^{bt_N}\displaystyle\sum_{k=1}^N \Delta_N Y_{t_k}\\[7pt]
-\dfrac{e^{bt_N}}{t_N}\displaystyle\sum_{k=1}^N \Delta_N Y_{t_k}&e^{2bt_N}\displaystyle\sum_{k=1}^N \Delta_N Y^2_{t_k}
\end{bmatrix}$ and $$\mathcal{J}_{t_N,\Delta_N}^{(2)}=\begin{bmatrix}
1&-e^{bt_N}\displaystyle\sum_{k=1}^N \Delta_N Y_{t_k}&-e^{\lambda_{\min}(\theta)t_N}\displaystyle\sum_{k=1}^N \Delta_N X_{t_k}^{\mathsf{T}}\\[7pt]
-\dfrac{e^{bt_N}}{t_N}\displaystyle\sum_{k=1}^N \Delta_N Y_{t_k}&e^{2bt_N}\displaystyle\sum_{k=1}^N \Delta_NY_{t_k}^2&e^{(b+\lambda_{\min}(\theta))t_N}\displaystyle\sum_{k=1}^N \Delta_N Y_{t_k}X_{t_k}^{\mathsf{T}}\\[7pt]
-\dfrac{e^{\lambda_{\min}(\theta)t_N}}{t_N}\displaystyle\sum_{k=1}^N \Delta_N X_{t_k}&e^{(b+\lambda_{\min}(\theta))t_N}\displaystyle\sum_{k=1}^N \Delta_N Y_{t_k}X_{t_k}&e^{2\lambda_{\min}(\theta)t_N}\displaystyle\sum_{k=1}^N \Delta_N X_{t_k}X_{t_k}^{\mathsf{T}}
\end{bmatrix}.$$
\begin{proposition}\label{Thm relation sum with integral SSC}
		For $b\in\R_{--}$, $\lambda_{\max}(\theta)<b$ and $\text{diag}(P^{-1}m)P^{-1}\kappa\in\R_-^n$, suppose that $N\Delta_N^{\frac{3}{2}}\to0$ and $\frac{N\Delta_N}{\ln(N\Delta_N^{3\slash2})}\to0$, as $N\to\infty$, then we have
\begin{equation*}	
e^{-\lambda_{\min}(\theta)t_N}\left(\mathcal{J}_{t_N,\Delta_N}^{-1}-\mathcal{J}_{t_N}^{-1}\right)\to \mathbf{0}\quad\text{and}\quad R_{t_N}\left(\check{\phi}_{t_N,\Delta_N}-f_{t_N}\right)\to\mathbf{0},
\end{equation*}
in $L^1$ and in probability, as $N\to\infty$.
\end{proposition}
\begin{proof}
   Thanks to proposition \ref{Thm relation sum with integral SSC}, we get
\begin{align*}
e^{bt_N}\mathbb{E}\left(\int_0^{t_N}Y_s\mathrm{d}s-\displaystyle\sum_{k=1}^N \Delta_N Y_{t_{k-1}}\right)\leq
     e^{bt_N}\sum_{k=1}^N\int_{t_{k-1}}^{t_k}\mathbb{E}(Y_{s}-Y_{t_{k-1}})\mathrm{d}s\leq e^{bt_N}\sum_{k=1}^N\int_{t_{k-1}}^{t_k}Ce^{-bt_k}\Delta_N^{\frac{1}{2}}\mathrm{d}s\leq C N\Delta_N^{\frac{3}{2}},
\end{align*}
similarly,
$e^{\lambda_{\min}(\theta)t_N}\mathbb{E}\left(\norm{\displaystyle\int_0^{t_N}X_s\mathrm{d}s-\sum_{k=1}^N \Delta_N X_{t_{k-1}}}_1\right)\leq C N\Delta_N^{\frac{3}{2}}.$
Furthermore, we get
\begin{multline*}
e^{2bt_N}\mathbb{E}\left(\displaystyle\int_0^{t_N}Y_s^2\mathrm{d}s-\sum_{k=1}^N \Delta_N Y_{t_{k-1}}^2\right)\leq e^{2bt_N}\sum_{k=1}^N\int_{t_{k-1}}^{t_k}\left(\mathbb{E}\left((Y_{s}-Y_{t_{k-1}})^2\right)\mathbb{E}\left(Y_{s}^2+Y_{t_{k-1}}^2\right)\right)^{\frac{1}{2}}\mathrm{d}s\\
 \leq C e^{2bt_N}\sum_{k=1}^N\int_{t_{k-1}}^{t_k}e^{-bs}\Delta_N^{\frac{1}{2}}\left(e^{-2bs}+e^{-2bt_{k-1}}\right)^{\frac{1}{2}}\mathrm{d}s\leq C N\Delta_N^{\frac{3}{2}}.
\end{multline*}
By a similar matrix decomposition, we deduce that $$e^{2\lambda_{\min}(\theta)t_N}\mathbb{E}\left(\norm{\displaystyle\int_0^{t_N}X_sX_{s}^{\mathsf{T}}\mathrm{d}s-\sum_{k=1}^N \Delta_N X_{t_{k-1}}X_{t_{k-1}}^{\mathsf{T}}}_1\right)\leq CN\Delta_N^{\frac{3}{2}}.$$ Furthermore, we obtain
\begin{multline*}
e^{(b+\lambda_{\min}(\theta))t_N}\mathbb{E}\left(\norm{\displaystyle\int_0^{t_N}Y_sX_s\mathrm{d}s-\sum_{k=1}^N \Delta_N Y_{t_{k-1}}X_{t_{k-1}}}_1\right)\leq
e^{(b+\lambda_{\min}(\theta))t_N}\sum_{k=1}^N\int_{t_{k-1}}^{t_k} \left(\left(\mathbb{E}(Y^2_s)\right)^{\frac{1}{2}}\right.\\
\times\left(\mathbb{E}\left(\norm{X_s-X_{t_{k-1}}}^2_1\right)\right)^{\frac{1}{2}}\left.+\left(\mathbb{E}\left((Y_s-Y_{t_{k-1}})^2\right)\right)^{\frac{1}{2}}\left(\mathbb{E}\left(\norm{X_{t_{k-1}}}^2_1\right)\right)^{\frac{1}{2}}\right)\mathrm{d}s
\leq CN\Delta_N^{\frac{3}{2}}. 
\end{multline*} 
Next, since we have 
\begin{align*}
    e^{-\lambda_{\min}(\theta)t_N}N\Delta_N^{\frac{3}{2}}=\exp\left(N\Delta_N\left(-\lambda_{\min}(\theta)+\frac{\ln(N\Delta_N^{\frac{3}{2}})}{N\Delta_N}\right)\right)\to 0,
\end{align*}
as $N\to\infty$, then the first part of the Proposition holds. For the last assertion, similarly to the previous cases, we proceed with same calculations for each component of $R_{t_N}\left(\check{\phi}_{t_N,\Delta_N}-f_{t_N}\right)$. For the first component, 
we get $ \mathbb{E}\left(\left\vert e^{\frac{3bt_N}{2}}\displaystyle\sum_{k=1}^N\int_{t_{k-1}}^{t_k}(Y_s-Y_{t_{k-1}})\mathrm{d}Y_s\right\vert\right)$ is less than or equal to
\begin{equation*}
\sqrt{\Delta_N} e^{\frac{3bt_N}{2}}\sum_{k=1}^N\mathbb{E}\left(\int_{t_{k-1}}^{t_k}\vert Y_s-Y_{t_{k-1}}\vert \vert a-b Y_s\vert \mathrm{d}s\right)+\sigma_1^2\left(\mathbb{E}\left(e^{3bt_N}\left\vert\sum_{k=1}^N\int_{t_{k-1}}^{t_k}\sqrt{Y_s}(Y_s-Y_{t_{k-1}})\mathrm{d}B_s^{1}\right\vert^2\right)\right)^{\frac{1}{2}}.
\end{equation*}
Hence, since in the supercritical case, we have $e^{qbt_N}\mathbb{E}
  \left(\left(Y_s-Y_{t_{k-1}}\right)^q\right)\leq C \Delta_N^{\frac{q}{2}}$ and by Proposition \ref{moment result on time t} we have $e^{qbt_N}\mathbb{E}\left(Y_s^q\right)\leq C$, then with similar arguments and calculations as in the proof of Proposition \ref{Thm lien entre intégrale et somme pour b>0}, we deduce that first term is bounded by
  \begin{equation*}
C \sqrt{\Delta_N}e^{-\frac{bt_N}{2}}\sum_{k=1}^N\int_{t_{k-1}}^{t_k}\left(e^{2bt_N}\mathbb{E}
  \left(\left(Y_s-Y_{t_{k-1}}\right)^2\right)\right)^{\frac{1}{2}}\left(e^{2bt_N}\mathbb{E}\left((a-bY_s)^2\right)\right)^{\frac{1}{2}}\mathrm{d}s\leq CN\Delta_N^2 e^{-\frac{bt_N}{2}},
\end{equation*}
  and that the second term is bounded by
  \begin{equation*}
e^{3bt_N}\sum_{k=1}^N\mathbb{E}\left(\int_{t_{k-1}}^{t_k}Y_s(Y_s-Y_{t_{k-1}})^2\mathrm{d}s\right)
   \leq \sum_{k=1}^N\int_{t_{k-1}}^{t_k}\left(e^{2bt_N}\mathbb{E}\left(Y_s^2\right)\right)^{\frac{1}{2}}\left(e^{4bt_N}\mathbb{E}\left((Y_s-Y_{t_{k-1}})^4\right)\right)^{\frac{1}{2}}\mathrm{d}s\leq CN\Delta_N^2.
\end{equation*}
Using the ussumption $\frac{\ln(N\Delta_N^{\frac{3}{2}})}{N\Delta_N}\to-\infty$, we deduce that the last two expressions tend to zero as $N\to\infty$. Analysis similar to the above one shows that the rest of components of $R_{t_N}\left(\check{\phi}_{t_N,\Delta_N}-f_{t_N}\right)$ tends also to zero in probability, as $N\to\infty$, which completes the proof.
\end{proof}
In the following subsection, we are going to study the asymptotic behavior of the estimator $\check{\tau}_{t_N}^{\Delta_N}$ based on the previous results of propositions \ref{Thm lien entre intégrale et somme pour b>0}, \ref{Thm relation sum with integral SCC} and \ref{Thm relation sum with integral SSC}. 
\subsection{Asymtotic behavior of the CLSE estimators}
First, we are going to establish the asymptotic normality of the discrete CLSE $\check{\tau}_{t_N}^{\Delta_N}$ in the subcritical case, we consider the affine diffusion model $\eqref{model01}$ with $a,b\in\R_{++}$ $m,\kappa\in\R^n$, $\theta\in\mathcal{M}_n$ a positive definite diagonalizable matrix.
\subsubsection{Subcritical case}
\begin{theoreme}
We suppose that $N\Delta_N^2\to0$ as $N\to\infty$, then  the discrete CLSE $\check{\tau}_{t_N}^{\Delta_N}$ is asymptotically normal, namely, 
	\begin{equation*}
		\sqrt{t_N}(\check{\tau}_{t_N}^{\Delta_N}-\tau)\stackrel{\mathcal{D}}{\longrightarrow}\mathcal{N}_{d^2+1}({0},[\mathbb{E}(G_\infty)]^{-1}\mathbb{E}(\mathrm{H}_\infty)[\mathbb{E}(G_\infty)]^{-1})
	\end{equation*}
	where $G_{\infty}$ and $\mathrm{H}_{\infty}$ are defined in Theorem \ref{Asymptotic Normality CLSE b>0 Continuous}.
	\label{Asymptotic Normality CLSE b>0 DISCRETE}
\end{theoreme}
\begin{proof}
	By combining the expression of the continuous CLSE \eqref{tau approx} with the one of the discrete CLSE \eqref{estimator discrete}, we get
	\begin{equation*}
		\sqrt{t_N}(\check{\tau}_{t_N}^{\Delta_N}-\hat{\tau}_{t_N})=\begin{bmatrix}
			\left(\dfrac{\Delta_N}{t_N}\check{\Gamma}_{t_N,\Delta_N}^{(1)}\right)^{-1}\dfrac{1}{\sqrt{t_N}}\check{\phi}_{t_N,\Delta_N}^{(1)}\\
			\text{vec}\left(\left(\dfrac{\Delta_N}{t_N}\check{\Gamma}_{t_N,\Delta_N}^{(2)}\right)^{-1}\dfrac{1}{\sqrt{t_N}}\check{\phi}_{t_N,\Delta_N}^{(2)}\right)
		\end{bmatrix}-\begin{bmatrix}
		\left(\dfrac{1}{t_N}G_{t_N}^{(1)}\right)^{-1}\dfrac{1}{\sqrt{t_N}}f_{t_N}^{(1)}\\\text{vec}\left(
		\left(\dfrac{1}{t_N}G_{t_N}^{(2)}\right)^{-1}\dfrac{1}{\sqrt{t_N}}f_{t_N}^{(2)}\right)
	\end{bmatrix}.
	\end{equation*}
Firstly, for all $i\in\lbrace1,2\rbrace$, by the ergodicity theorem \ref{ergodicity theorem}, the central limit theorem \ref{CLT Van Zanten} and Slutsky's lemma, the couple $\left(\frac{1}{t_N}G_{t_N}^{(i)},\frac{1}{\sqrt{t_N}}f_{t_N}^{(i)}\right)$ converges in distribution, as $N\to\infty$.
Secondly, thanks to Proposition \ref{Thm lien entre intégrale et somme pour b>0}, we have 
$\left(\frac{1}{t_N}G_{t_N}^{(i)},\frac{1}{\sqrt{t_N}}f_{t_N}^{(i)}\right)-\left(\frac{\Delta_N}{t_N}\Gamma_{t_N,\Delta_N}^{(i)},\frac{1}{\sqrt{t_N}}\check{\phi}_{t_N,\Delta_N}^{(i)}\right)\stackrel{\mathbb{P}}{\longrightarrow}\mathbf{0}$, as $N\to\infty$.
Finally, since for all $N\in\N$, we have $\dfrac{\Delta_N}{t_N}\check{\Gamma}_{t_N,\Delta_N}^{(i)}$ and $\dfrac{1}{t_N}G_{t_N}^{(i)}$ are positive definite matrices and $(A,B)\mapsto A^{-1}B$ is a continuous matrix function on the set of non-singular matrices, then by combining Lemma \ref{convergence (Xn,Yn) to (X,X) in law} with the continuous mapping theorem, we obtain $\sqrt{t_N}(\check{\tau}_{t_N}^{\Delta_N}-\hat{\tau}_{t_N})\stackrel{\mathbb{P}}{\longrightarrow}\mathbf{0}$, as $N\to\infty$. We complete the proof using the decomposition $\eqref{discrete-continuous+continuous-exact}$ and Theorem \ref{Asymptotic Normality CLSE b>0 Continuous}. 
\end{proof}
Next, we consider the special critical case introduced in Subsection \ref{A special critical case continuous observations} to study the asymptotic behavior of the discrete CLSE $\check{\tau}_{t_N}^{\Delta_N}$, we consider the model $\eqref{model01}$ with $a\in\R_{+}$, $b=0$, $m\in \R^n$, $\kappa=\mathbf{0}_n$ and $\theta=\mathbf{0}_{n,n}$.   
\subsubsection{A special critical case}
\begin{theoreme}
We suppose that $N\Delta_N^{2}\to 0$ as $N\to\infty$, then
{\begin{equation*}
		\begin{bmatrix}
			\check{a}_{t_N}^{\Delta_N}-a\\
			t_N \check{b}_{t_N}^{\Delta_N}\\[3pt]
   \text{vec}\begin{bmatrix}
				\left.\check{m}_{t_N}^{\Delta_N}\right.^{\mathsf{T}}-m^{\mathsf{T}}\\
				t_N \left.\check{\kappa}_{t_N}^{\Delta_N}\right.^{\mathsf{T}}\\
				t_N \left.\check{\theta}_{t_N}^{\Delta_N}\right.^{\mathsf{T}}
			\end{bmatrix}
		\end{bmatrix}\stackrel{\mathcal{D}}{\longrightarrow}		\text{diag}\left(\mathit{U}_1^{-1},I_n\otimes\mathit{U}_2^{-1}\right)(\mathit{R}_1,\text{vec}(R_2))^{\mathsf{T}},\quad\text{as }N\to\infty,
\end{equation*}}
where $\mathit{U}_1$, $\mathit{U}_2$, $\mathit{R}_1$ and $\mathit{R}_2$ are defined in Theorem \ref{Théorème SCC continuous observations}. 
\label{Asymptotic Normality SCC DISCRETE}
\end{theoreme}
\begin{proof}
Let $i\in\{1,2\}$. Similarly to the previous case and by same matrix calculation as in relations \ref{hat a - a and T hat b SCC} and \ref{hat m - m T hat k and T hat theta SSC}, on the one hand, we get 
\begin{align*}
		\begin{bmatrix}	\check{a}_{t_N}^{\Delta_N}-a\\
			t_N \check{b}_{t_N}^{\Delta_N}
			\end{bmatrix}
		=&
			\begin{bmatrix}
					1& -\dfrac{1}{t_N^2}\displaystyle\sum_{k=1}^{N}\Delta_NY_{t_{k-1}}\\
					-\dfrac{1}{t_N^2}\displaystyle\sum_{k=1}^{N}\Delta_NY_{t_{k-1}}&\dfrac{1}{t_N^3} \displaystyle\sum_{k=1}^{N}\Delta_NY_{t_{k-1}}^2
				\end{bmatrix}^{-1}\begin{bmatrix}
					\dfrac{1}{t_N}(Y_{t_N}-Y_{0})\\
					-\dfrac{1}{t_N^2}\displaystyle\sum_{k=1}^{N} \left(Y_{t_k}-Y_{t_{k-1}}\right)Y_{t_{k-1}}
				\end{bmatrix}\\
    &-\begin{bmatrix}
			1&-\dfrac{1}{t_N^2}\displaystyle\int_{0}^{t_N}Y_s\mathrm{d}s\\[3pt]
   -\dfrac{1}{t_N^2}\displaystyle\int_{0}^{t_N}Y_s\mathrm{d}s&\dfrac{1}{t_N^3}\displaystyle\int_{0}^{t_N}Y_s^2\mathrm{d}s
		\end{bmatrix}^{-1}\begin{bmatrix}
					\dfrac{1}{t_N}(Y_{t_N}-Y_{0})\\
					-\dfrac{1}{t_N^2}\displaystyle\int_{0}^{t_N} Y_s\mathrm{d}Y_s
				\end{bmatrix}.
\end{align*}

and on the other hand, we obtain $\begin{bmatrix}
				\left.\check{m}_{t_N}^{\Delta_N}\right.^{\mathsf{T}}-m^{\mathsf{T}}\\
				t_N \left.\check{\kappa}_{t_N}^{\Delta_N}\right.^{\mathsf{T}}\\
				t_N \left.\check{\theta}_{t_N}^{\Delta_N}\right.^{\mathsf{T}}
			\end{bmatrix}$ is equal to
\begin{align*}
&\begin{bmatrix}
					1 & -\dfrac{1}{t_N^2}\displaystyle\sum_{k=1}^{N}\Delta_N Y_{t_{k-1}}  &-\dfrac{1}{t_N^2}\displaystyle\sum_{k=1}^{N}\Delta_NX_{t_{k-1}}^{\mathsf{T}}\\
					-\dfrac{1}{t_N^2}\displaystyle\sum_{k=1}^{N}\Delta_N Y_{t_{k-1}} & \dfrac{1}{t_N^3}\displaystyle\sum_{k=1}^{N}\Delta_N Y_{t_{k-1}}^2 &\dfrac{1}{t_N^3} \displaystyle\sum_{k=1}^{N}\Delta_N Y_{t_{k-1}}X_{t_{k-1}}^{\mathsf{T}}\\
					-\dfrac{1}{t_N^2}\displaystyle\sum_{k=1}^{N}\Delta_NX_{t_{k-1}} &  \dfrac{1}{t_N^3}\displaystyle\sum_{k=1}^{N}\Delta_N Y_{t_{k-1}}X_{t_{k-1}} & \dfrac{1}{t_N^3}\displaystyle\sum_{k=1}^{N}\Delta_N X_{t_{k-1}}X_{t_{k-1}}^{\mathsf{T}}
				\end{bmatrix}^{-1}\begin{bmatrix}
					\dfrac{1}{t_N}(X_{t_N}-X_{0})^{\mathsf{T}}\\
					-\dfrac{1}{t_N^2}\displaystyle\sum_{k=1}^{N} Y_{t_{k-1}}\left(X_{t_k}^{\mathsf{T}}-X_{t_{k-1}}^{\mathsf{T}}\right)\\
					-\dfrac{1}{t_N^2}\displaystyle\sum_{k=1}^{N}X_{t_{k-1}}\left(X_{t_k}^{\mathsf{T}}-X_{t_{k-1}}^{\mathsf{T}}\right)
				\end{bmatrix}\\
    &-\begin{bmatrix}
		1 & -\dfrac{1}{t_N^2}\displaystyle\int_0^{t_N} Y_s \mathrm{d}s  &-\dfrac{1}{t_N^2}\displaystyle\int_0^{t_N} X_s^{\mathsf{T}} \mathrm{d}s\\[8pt]
		-\dfrac{1}{t_N^2}\displaystyle\int_0^{t_N} Y_s \mathrm{d}s &\dfrac{1}{t_N^3} \displaystyle\int_0^{t_N} Y_s^2 \mathrm{d}s &\dfrac{1}{t_N^3} \displaystyle\int_0^{t_N} Y_s X_s^{\mathsf{T}}  \mathrm{d}s\\[8pt]
		-\dfrac{1}{t_N^2}\displaystyle\int_0^{t_N} X_s \mathrm{d}s  &   \dfrac{1}{t_N^3}\displaystyle\int_0^{t_N} Y_s X_s \mathrm{d}s  &  \dfrac{1}{t_N^3}\displaystyle\int_0^{t_N} X_s X_s^{\mathsf{T}} \mathrm{d}s 
	\end{bmatrix}^{-1}\begin{bmatrix}
	\dfrac{1}{t_N}(X_{t_N}-X_{0})^{\mathsf{T}}\\[2pt]
	-\dfrac{1}{t_N^2}\displaystyle\int_0^{t_N} Y_{s} \mathrm{d}X_{s}^{\mathsf{T}}\\[8pt]
	-\dfrac{1}{t_N^2}\displaystyle\int_0^{t_N} X_{s} \mathrm{d}X_{s}^{\mathsf{T}}
\end{bmatrix}.
\end{align*}
 Firstly, analogously to the proof of Theorem \ref{Théorème SCC continuous observations},
we collect all the present components into the functional vector $V((Y_t)_{t\in[0,t_N]},(X_t)_{t\in[0,t_N]},t_N)$ defined by relation \eqref{function V SCC} which converges in distribution to $V((\mathcal{Y}_t)_{t\in[0,1]},(\mathcal{X}_t)_{t\in[0,1]},1)$, as $N\to\infty$. Secondly, using Proposition \ref{Thm relation sum with integral SCC}, we obtain $$V_{t_N,\Delta_N}((Y_t)_{t\in[0,t_N]},(X_t)_{t\in[0,t_N]})-V((Y_t)_{t\in[0,t_N]},(X_t)_{t\in[0,t_N]},t_N)\stackrel{\mathbb P}{\longrightarrow}\mathbf{0},\quad \text{as }N\to\infty,$$ where $V_{t_N,\Delta_N}((Y_t)_{t\in[0,t_N]},(X_t)_{t\in[0,t_N]})$ is equal to
\begin{multline*}
	\left(\dfrac{1}{t_N}{Y}_{t_N},\,\dfrac{1}{t_N}{X}_{t_N},\,\dfrac{1}{t_N^2}\displaystyle\sum_{k=1}^N \Delta_N{Y}_{t_{k-1}},\,\dfrac{1}{t_N^2}\displaystyle\sum_{k=1}^N \Delta_N X_{t_{k-1}},\,\dfrac{1}{t_N^3}\displaystyle\sum_{k=1}^N \Delta_N Y_{t_{k-1}}^2,\,\dfrac{1}{t_N^3}\displaystyle\sum_{k=1}^N \Delta_N X_{t_{k-1}} X_{t_{k-1}}^{\mathsf{T}},\right.\\
  \left.\dfrac{1}{t_N^3}\displaystyle\sum_{k=1}^N \Delta_N Y_{t_{k-1}}X_{t_{k-1}},\,\dfrac{1}{t_N^2}\displaystyle\sum_{k=1}^N Y_{t_{k-1}}(Y_{t_k}-Y_{t_{k-1}}),\,\dfrac{1}{t_N^2}\displaystyle\sum_{k=1}^NY_{t_{k-1}}(X_{t_k}-X_{t_{k-1}}),\,\dfrac{1}{t_N^2}\displaystyle\sum_{k=1}^N X_{t_{k-1}}(X_{t_k}-X_{t_{k-1}})^{\mathsf{T}}\right).
\end{multline*}
Finally, since for all $N\in\N$, we have the non-singularity of the above inverted matrices, then by combining Lemma \ref{convergence (Xn,Yn) to (X,X) in law} with the continuous mapping theorem, we obtain $t_N(\check{\tau}_{t_N}^{\Delta_N}-\hat{\tau}_{t_N})\stackrel{\mathbb{P}}{\longrightarrow}\mathbf{0}$, as $N\to\infty$, in particular $\check{a}_{t_N}^{\Delta_N}-\hat{a}_{t_N}\stackrel{\mathbb{P}}{\longrightarrow}0$ and $\check{m}_{t_N}^{\Delta_N}-\hat{m}_{t_N}\stackrel{\mathbb{P}}{\longrightarrow}\mathbf{0}_n$ as $N\to\infty$. We complete the proof using the decomposition $\eqref{discrete-continuous+continuous-exact}$ and Theorem \ref{Théorème SCC continuous observations}. 
\end{proof}
The last part of this section is devoted to study the asymptotic normality relative to the discrete CLSE $\check{\tau}_{t_N}^{\Delta_N}$ in the special supercritical case introduced in Subsection \ref{A special supercritical case continuous observations}. 
\subsubsection{A special Supercritical case}
Let us recall that in this case we consider the model $\eqref{model01}$ with $a\in\R_{+}$, $b\in\R_{--}$ $m\in\R^n$, $\kappa\in\R^n$ and $\theta\in\mathcal{M}_n$ a diagonalizable negative definite matrix such that $b\in(\lambda_{\max}(\theta),0)$ and $\text{diag}(P^{-1}m)P^{-1}\kappa\in\R_-^n$.
\begin{theoreme}
	 Suppose that $N\Delta_N^{\frac{3}{2}}\to0$ and $\frac{N\Delta_N}{\ln(N\Delta_N^{3\slash2})}\to0$, as $N\to\infty$. Then
	\begin{equation}
		Q_{t_N}(\check{\tau}_{t_N}^{\Delta_N}-\tau)\stackrel{\mathcal{D}}{\longrightarrow}V^{-1}\eta\xi,\quad \text{as } N\to \infty,
		\label{distribution convergence special supercritical DISCRETE case}
	\end{equation}
	where $Q_{t_N}$, $V$, $\xi$ and $\eta$ are defined as in Theorem \ref{normality theorem supercritical case}. 
\label{normality theorem supercritical DISCRETE case}
\end{theoreme}
\begin{proof}
Using the matrix decomposition \eqref{Decomposition discrete - continuous SSC}, since we have $\mathcal{J}_{t_N}^{-1}\stackrel{\mathbb{P}}{\longrightarrow}V^{-1}$, then thanks to Proposition \ref{Thm relation sum with integral SSC}, we deduce that the second term $\mathcal{J}_{t_N}^{-1}R_{t_N}\left(\check{\phi}_{t_N,\Delta_N}-f_{t_N}\right)\stackrel{\mathbb{P}}{\longrightarrow}\mathbf{0}$, as $N\to\infty$. Furthermore, since we have we have \begin{align*}
e^{\lambda_{\min}(\theta)t_N}R_{t_N}\check{\phi}_{t_N,\Delta_N}=e^{\lambda_{\min}(\theta)t_N}R_{t_N}\left(\check{\phi}_{t_N,\Delta_N}-f_{t_N}\right)+\mathcal{J}_{t_N}e^{\lambda_{\min}(\theta)t_N}Q_{t_N}\tau+e^{\lambda_{\min}(\theta)t_N}R_{t_N}h_{t_N},
\end{align*}
where $e^{\lambda_{\min}(\theta)t_N}Q_{t_N}\to\mathbf{0}$, $\mathcal{J}_{t_N}\stackrel{\mathbb{P}}{\longrightarrow}V$ and $R_{t_N}h_{t_N}$ converges in distribution as $N\to\infty$, then by Proposition \ref{Thm relation sum with integral SSC}, we get $e^{\lambda_{\min}(\theta)t_N}R_{t_N}\check{\phi}_{t_N,\Delta_N}\stackrel{\mathbb P}{\longrightarrow}\mathbf{0}$, as $N\to\infty$.
Finally, it is sufficient to prove that $e^{-\lambda_{\min}(\theta)t_N}\left(\mathcal{J}_{t_N,\Delta_N}^{-1}-\mathcal{J}_{t_N}^{-1}\right)\stackrel{\mathbb P}{\longrightarrow}\mathbf{0}$, as $N\to\infty$. To do that, we use 
 the mean value theorem on the differentiable function $f$ defined on the space of invertible matrices by $f(A)=e_k^{\mathsf{T}}A^{-1}e_{k'}$ where $e_k$ and $e_{k'}$ are two arbitrary canonical basis vectors to get
$
\left\vert f\left(\mathcal{J}_{t_N,\Delta_N}^{(i)}\right)-f\left(\mathcal{J}_{t_N}^{(i)}\right)\right\vert\leq \norm{\mathcal{J}_{t_N,\Delta_N}^{(i)}-\mathcal{J}_{t_N}^{(i)}}_1 \underset{A\in\mathcal{A}_i}{\sup}\norm{\mathrm{D}f(A)}_1$, with $\mathcal{A}_i:= \left\lbrace \lambda\mathcal{J}_{t_N}^{(i)}+(1-\lambda)\mathcal{J}_{t_N,\Delta_N}^{(i)},\lambda\in[0,1]\right\rbrace$. Note that, since $\mathcal{J}_{t_N}^{(i)}$ and $\mathcal{J}_{t_N,\Delta_N}^{(i)}$ converge almost surely to the same invertible matrix $V$, then for sufficient large values of $N$, $\mathcal{A}_i$ is a subset of invertible matrices. Now, thanks to relation (61) in \cite{matrix cookbook}, we have $\mathrm{D}f(A)=\left(A^{-1}e_{k'}e_{k}^{\mathsf{T}}A^{-1}\right)^{\mathsf{T}}$ and as $\underset{A\in\mathcal{A}_i}{\sup}\norm{\mathrm{D}f(A)}_1$ converges almost surely to $\norm{\mathrm{D}f(V_i)}_1$. Then, we obtain by Theorem \ref{Thm relation sum with integral SSC}, $e^{-\lambda_{\min}(\theta)t_N}\left\vert e_k^{\mathsf{T}}\left(\mathcal{J}_{t_N,\Delta_N}^{-1}-\mathcal{J}_{t_N}^{-1}\right)e_{k'}\right\vert\stackrel{a.s.}{\longrightarrow}\mathbf{0}$, as $N\to\infty$ and by considering all the possible combinations of $e_k$ and $e_{k'}$, we get the desired result. Finally, by combining relation $\eqref{discrete-continuous+continuous-exact}$ with Theorem \ref{normality theorem supercritical case}, we complete the proof.   
\end{proof}

\vspace*{0.5cm}
\appendix
\section{Moment results}

In the first appendix, we are going to establish two propositions about the mixed moments of the $\mathit{AD}(1,n)$ processes given by the model \eqref{model01}. Let us introduce first the process given for all $t\in\R_+$ by $\tilde{X}_t:=PX_t$, where its $i^{\text{th}}$ component $\tilde{X}_t^i$ satisfies 
	\begin{equation*}
		\mathrm{d}\tilde{X}_t^i=(\tilde{m}_i-\tilde{\kappa}_iY_t-\lambda_i\tilde{X}_t^i)\mathrm{d}t+\sqrt{Y_t}\check{\rho}_{i,\mathit{H}}\mathrm{d}B_t,
	\end{equation*}
with  $\mathit{H}=\{1,2,\ldots,d\}$, $\tilde{m}=P m$, $\tilde{\kappa}=P\kappa$, $D=\text{diag}(\lambda)=P\theta P^{-1}$, $\check{\rho}=P\tilde{\rho}$ and $\check{\sigma}_i^2=\sum_{j=1}^d\check{\rho}_{i,j}^2$. We pay the attention that $\mathbb{E}\left(Y_t^k\prod_{i=1}^n(\tilde{X}_t^i)^{\ell_i}\right)$ behaves with the same convergence and expansion rates as $\mathbb{E}\left(Y_t^k\prod_{i=1}^n(X_t^i)^{\ell_i}\right)$, as $t\to\infty$, and that, in the ergodic case, the finiteness of $\mathbb{E}\left(Y_\infty^k\prod_{i=1}^n(\tilde{X}_\infty^i)^{\ell_i}\right)$ implies the one of $\mathbb{E}\left(Y_\infty^k\prod_{i=1}^n(X_\infty^i)^{\ell_i}\right)$ for all $k,\ell_1,\ldots\ell_n\in\N$. Since such properties are important and used in this paper, we are going in the following to establish the moment results on the processes $(Y,\tilde{X})^{\mathsf T}$ instead of $(Y,X)^{\mathsf T}$.
\begin{proposition}\label{moment result on time t}
	Let us consider the affine diffusion model $\eqref{model01}$ with $a\in\R_{+}$, $b\in\R$, $m,\kappa\in\R^n$, $\theta\in\mathcal{M}_n$ and the random initial value $(Y_0,X_0)$ independent of $(B_t)_{t\in\R_+}$ satisfying $\mathbb{P}(Y_0\in\R_+)=1$ and $\mathbb{E}\left(Y_0^{r_0}\prod_{i=1}^{n}(X_0^i)^{r _i}\right)<\infty$, for some $r_0,r_1,\ldots,r_n\in\N$. Then, for all $k\in\lbrace0,\ldots,r_0\rbrace$ and $\ell_i\in\lbrace0,\ldots,r_i\rbrace$, we have $\mathbb{E}\left(Y_t^{k}\prod_{q=1}^n(\tilde{X}_t^q)^{\ell_q}\right)$ is equal to 
	\begin{align*}
		&e^{-\left(bk+\sum_{j=1}^{n}\lambda_j\ell_j\right)t}\ \mathbb{E}\left(Y_0^{k}\prod_{q=1}^n(\tilde{X}_0^q)^{\ell_q}\right)+\left(ak+\dfrac{\sigma_1^2}{2}k(k-1)\right)\int_0^t e^{-\left(bk+\sum_{j=1}^{n}\lambda_j\ell_j\right)(t-u)}\\  &\times\mathbb{E}\left(Y_u^{k-1}\prod_{q=1}^n(\tilde{X}_u^q)^{\ell_q}\right)\mathrm{d}u+\sum_{p=1}^n\ell_p\left(\tilde{m}_p +k\sigma_1\check{\rho}_{p1} \right)\int_0^t e^{-\left(bk+\sum_{j=1}^{n}\lambda_j\ell_j\right)(t-u)}\\ &\times\mathbb{E}\left(Y_u^k(\tilde{X}_u^p)^{\ell_p-1}\prod_{\substack{q=1\\q\neq p}}^n(\tilde{X}_u^q)^{\ell_q}\right)\mathrm{d}u-\sum_{p=1}^n \tilde{\kappa}_p \ell_p \int_{0}^t e^{-\left(bk+\sum_{j=1}^{n}\lambda_j\ell_j\right)(t-u)}\\ 
		&\times\mathbb{E}\left(Y_u^{k+1} (\tilde{X}_u^p)^{\ell_p-1}\prod_{\substack{q=1\\q\neq p}}^n(\tilde{X}_u^q)^{\ell_q}\right)\mathrm{d}u+\sum_{p=1}^n\dfrac{\check{\sigma}_{p+1}^2}{2}\ell_p(\ell_p-1)\int_{0}^t e^{-\left(bk+\sum_{j=1}^{n}\lambda_j\ell_j\right)(t-u)}\\ &\times\mathbb{E}\left(Y_u^{k+1}(\tilde{X}_u^p)^{\ell_p-2}\prod_{\substack{q=1\\q\neq p}}^n(\tilde{X}_u^q)^{\ell_q}\right)\mathrm{d}u+\dfrac{1}{2}\sum_{p=1}^n\sum_{\substack{i=1\\ i\neq p}}^n\check{\rho}_{i,\mathit{H}}\left(\check{\rho}_{p,\mathit{H}}\right)^{\mathsf{T}}\ell_i\ell_p\\
		&\times \int_0^t e^{-\left(bk+\sum_{j=1}^{n}\lambda_j\ell_j\right)(t-u)}\ \mathbb{E}\left(Y_u^{k+1}(\tilde{X}_u^i)^{\ell_i-1}(\tilde{X}_u^p)^{\ell_p-1}\prod_{\substack{q=1\\q\neq i,p}}^n(\tilde{X}_u^q)^{\ell_q}\right)\mathrm{d}u,
	\end{align*}
 with the convention $\mathbb{E}\left(Y_t^k\prod_{q=1}^n(\tilde{X}_t^q)^{\ell_q}\right):=0$ when one of the powers $k,\ell_1,\ldots,\ell_n$ is negative. 
\end{proposition}
\begin{remarque}
	Under the same assumptions of Proposition \ref{moment result on time t}, we get the following special moments results
	$$\mathbb{E}(Y_t)=e^{-bt}\mathbb{E}(Y_0)+a\displaystyle\int_0^te^{-b(t-u)}\mathrm{d}u,\quad \mathbb{E}(Y^2_t)=e^{-2bt}\mathbb{E}(Y_0^2)+(2a+\sigma_1^2)\displaystyle\int_0^te^{-2b(t-u)}\mathbb{E}(Y_u)\mathrm{d}u,$$ 
	$$ \mathbb{E}(Y^3_t)=e^{-3bt}\mathbb{E}(Y_0^3)+3(a+\sigma_1^2)\displaystyle\int_0^te^{-3b(t-u)}\mathbb{E}(Y^2_u)\mathrm{d}u,$$ $$\mathbb{E}(\tilde{X}^i_t)=e^{-\lambda_it}\mathbb{E}(\tilde{X}^i_0)+\tilde{m}_i\displaystyle\int_0^te^{-\lambda_i(t-u)}\mathrm{d}u-\tilde{\kappa}_i\int_0^te^{-\lambda_i(t-u)}\mathbb{E}(Y_u)\mathrm{d}u,$$ \begin{align*}\mathbb{E}(Y_t\tilde{X}^i_t)=&\,e^{-(b+\lambda_i)t}\mathbb{E}(Y_0\tilde{X}^i_0)+a\displaystyle\int_0^te^{-(b+\lambda_i)(t-u)}\mathbb{E}(\tilde{X}^i_u)\mathrm{d}u\\&+(\tilde{m}_i+\sigma_1\check{\rho}_{i1})\displaystyle\int_0^te^{-(b+\lambda_i)(t-u)}\mathbb{E}(Y_u)\mathrm{d}u-\tilde{\kappa}_i\int_0^te^{-(b+\lambda_i)(t-u)}\mathbb{E}(Y^2_u)\mathrm{d}u,\end{align*}
	\begin{align*}\mathbb{E}(Y_t^2\tilde{X}^i_t)=&\,e^{-(2b+\lambda_i)t}\mathbb{E}(Y_0^2\tilde{X}^i_0)+(2a+\sigma_1^2)\displaystyle\int_0^te^{-(2b+\lambda_i)(t-u)}\mathbb{E}(Y_u\tilde{X}^i_u)\mathrm{d}u\\&+(\tilde{m}_i+2\sigma_1\check{\rho}_{i1})\displaystyle\int_0^te^{-(2b+\lambda_i)(t-u)}\mathbb{E}(Y_u^2)\mathrm{d}u-\tilde{\kappa}_i\int_0^te^{-(2b+\lambda_i)(t-u)}\mathbb{E}(Y^3_u)\mathrm{d}u,\end{align*}
	
	\begin{align*}\mathbb{E}((\tilde{X}^i_t)^2)=&\,e^{-2\lambda_it}\mathbb{E}((\tilde{X}^i_0)^2)+2\tilde{m}_i\displaystyle\int_0^te^{-2\lambda_i(t-u)}\mathbb{E}(\tilde{X}_u^i)\mathrm{d}u-2\tilde{\kappa}_i\int_0^te^{-2\lambda_i(t-u)}\mathbb{E}(Y_u\tilde{X}_u^i)\mathrm{d}u\\&+\check{\sigma}^2_{i+1}\int_0^te^{-2\lambda_i(t-u)}\mathbb{E}(Y_u)\mathrm{d}u,\end{align*}
	\begin{align*}
		\mathbb{E}(\tilde{X}^i_t\tilde{X}^j_t)=&\,e^{-(\lambda_i+\lambda_j)t}\mathbb{E}(\tilde{X}^i_0\tilde{X}^j_0)+\tilde{m}_i\displaystyle\int_0^te^{-(\lambda_i+\lambda_j)(t-u)}\mathbb{E}(\tilde{X}_u^j)\mathrm{d}u-\tilde{\kappa}_i\int_0^te^{-(\lambda_i+\lambda_j)(t-u)}\mathbb{E}(Y_u\tilde{X}_u^j)\mathrm{d}u\\&+\tilde{m}_j\displaystyle\int_0^te^{-(\lambda_i+\lambda_j)(t-u)}\mathbb{E}(\tilde{X}_u^i)\mathrm{d}u-\tilde{\kappa}_j\int_0^te^{-(\lambda_i+\lambda_j)(t-u)}\mathbb{E}(Y_u\tilde{X}_u^i)\mathrm{d}u\\&+\check{\rho}_{i,\mathit{H}}\left(\check{\rho}_{j,\mathit{H}}\right)^{\mathsf{T}}\int_0^te^{-(\lambda_i+\lambda_j)(t-u)}\mathbb{E}(Y_u)\mathrm{d}u,
	\end{align*}
	\begin{align*}\mathbb{E}(Y_t(\tilde{X}^i_t)^2)=&\,e^{-(b+2\lambda_i)t}\mathbb{E}(Y_0(\tilde{X}^i_0)^2)+a\displaystyle\int_0^te^{-(b+2\lambda_i)(t-u)}\mathbb{E}((\tilde{X}^i_u)^2)\mathrm{d}u\\&+2(\tilde{m}_i+\sigma_1\check{\rho}_{i1})\displaystyle\int_0^te^{-(b+2\lambda_i)(t-u)}\mathbb{E}(Y_u\tilde{X}_u^i)\mathrm{d}u-2\tilde{\kappa}_i\int_0^te^{-(b+2\lambda_i)(t-u)}\mathbb{E}(Y^2_u\tilde{X}_u^i)\mathrm{d}u\\&+\check{\sigma}^2_{i+1}\int_0^te^{-(b+2\lambda_i)(t-u)}\mathbb{E}(Y^2_u)\mathrm{d}u,\end{align*}
	and	
	\begin{align*}
		\mathbb{E}(Y_t\tilde{X}^i_t\tilde{X}^j_t)=&\,e^{-(b+\lambda_i+\lambda_j)t}\mathbb{E}(Y_0\tilde{X}^i_0\tilde{X}^j_0)+a\displaystyle\int_0^te^{-(b+\lambda_i+\lambda_j)(t-u)}\mathbb{E}(\tilde{X}^i_u\tilde{X}^j_u)\mathrm{d}u\\&+(\tilde{m}_i+\sigma_1\check{\rho}_{i1})\displaystyle\int_0^te^{-(b+\lambda_i+\lambda_j)(t-u)}\mathbb{E}(Y_u\tilde{X}_u^j)\mathrm{d}u-\tilde{\kappa}_i\int_0^te^{-(b+\lambda_i+\lambda_j)(t-u)}\mathbb{E}(Y^2_u\tilde{X}_u^j)\mathrm{d}u\\&+(\tilde{m}_j+\sigma_1\check{\rho}_{j1})\displaystyle\int_0^te^{-(b+\lambda_i+\lambda_j)(t-u)}\mathbb{E}(Y_u\tilde{X}_u^i)\mathrm{d}u-\tilde{\kappa}_j\int_0^te^{-(b+\lambda_i+\lambda_j)(t-u)}\mathbb{E}(Y^2_u\tilde{X}_u^i)\mathrm{d}u\\&+\check{\rho}_{i,\mathit{H}}\left(\check{\rho}_{j,\mathit{H}}\right)^{\mathsf{T}}\int_0^te^{-(b+\lambda_i+\lambda_j)(t-u)}\mathbb{E}(Y^2_u)\mathrm{d}u,\end{align*}
	\label{Remark moment resulty on time t}
\end{remarque}
\begin{proof}(Proposition \ref{moment result on time t})
By applying Itô's Formula on the product $Y_t^k\prod_{q=1}^n(\tilde{X}_t^q)^{\ell_q}$, we deduce by simple
integration and by the continuous square integrable martingales property of the Brownian integrals that  $\mathbb{E}\left(Y_t^k\prod_{q=1}^n(\tilde{X}_t^q)^{\ell_q}\right)-\mathbb{E}\left(Y_0^k\prod_{q=1}^n(\tilde{X}_0^q)^{\ell_q}\right)$ is equal to
	\begin{align}
  \label{Ito for mixed moments}
 \begin{split}
		&-\left(bk+\sum_{j=1}^n \lambda_j \ell_j\right)\displaystyle\int_0^t\mathbb{E}\left(Y_u^k\prod_{q=1}^n(\tilde{X}_u^q)^{\ell_q}\right)\mathrm{d}u+\left(ak+\dfrac{\sigma_1^2}{2}k(k-1)\right)\displaystyle\int_0^t \mathbb{E}\left(Y_u^{k-1}\prod_{q=1}^n(\tilde{X}_u^q)^{\ell_q}\right)\mathrm{d}u\\&+\sum_{j=1}^n\displaystyle\int_0^t\left[\ell_j\left(\tilde{m}_j +k\sigma_1\check{\rho}_{j1} \right)\mathbb{E}\left(Y_u^k(\tilde{X}_u^j)^{\ell_j-1}\prod_{\substack{q=1\\q\neq j}}^n(\tilde{X}_u^q)^{\ell_q}\right)- \tilde{\kappa}_j \ell_j \mathbb{E}\left(Y_u^{k+1} (\tilde{X}_u^j)^{\ell_j-1}\prod_{\substack{q=1\\q\neq j}}^n(\tilde{X}_u^q)^{\ell_q}\right)\right.\\&\qquad\qquad+\left.\dfrac{\check{\sigma}_{j+1}^2}{2}\ell_j(\ell_j-1)\mathbb{E}\left(Y_u^{k+1}(\tilde{X}_u^j)^{\ell_j-2}\prod_{\substack{q=1\\q\neq j}}^n(\tilde{X}_u^q)^{\ell_q}\right)+\dfrac{1}{2}\sum_{\substack{i=1\\ i\neq j}}^n\check{\rho}_{i,\mathit{H}}\left(\check{\rho}_{j,\mathit{H}}\right)^{\mathsf{T}}\ell_i\ell_j\right.\\
		&\qquad\qquad\left.\times\mathbb{E}\left(Y_u^{k+1}(\tilde{X}_u^i)^{\ell_i-1}(\tilde{X}_u^j)^{\ell_j-1}\prod_{\substack{q=1\\q\neq i,j}}^n(\tilde{X}_u^q)^{\ell_q}\right)\right]\mathrm{d}u.
   \end{split}
	\end{align}
Finally, by introducing the functions $f(k,\ell_1,\ell_2,\ldots,\ell_n,t):=\mathbb{E}\left(Y_t^k\prod_{q=1}^n(\tilde{X}_t^q)^{\ell_q}\right),$ for all $t\in\R_+$ and $ k,\ell_1,\ell_2,\ldots,\ell_n\in\N$ and by combining the last equation with the time derivative $e^{\left(bk+\sum_{j=1}^{n}\lambda_j\ell_j\right)t}f$  given by 
$$\left(bk+\sum_{j=1}^{n}\lambda_j\ell_j\right)e^{\left(bk+\sum_{j=1}^{n}\lambda_j\ell_j\right)t}f(k,\ell_1,\ell_2,\ldots,\ell_n,t)+e^{\left(bk+\sum_{j=1}^{n}\lambda_j\ell_j\right)t}\dfrac{\partial f}{\partial t}(k,\ell_1,\ell_2,\ldots,\ell_n,t),$$
we deduce, by simple integration, the induction result given in the proposition.
\end{proof}

\begin{proposition}\label{moment result}
	Let us consider the affine diffusion model $\eqref{model01}$ with $a,b\in\R_{++}$ $m,\kappa\in\R^n$, $\theta\in\mathcal{M}_n$ a positive definite matrix. Let $(Y_{\infty},X_{\infty})$ the random vector defined by relation \eqref{loi limite Stationarity theorem}. Then, for all $k\in\N$ and $\ell=(\ell_1,\ldots,\ell_n)\in\N^n$ the mixed expectation $\mathbb{E}\left(Y_\infty^k\prod_{q=1}^n(\tilde{X}_\infty^q)^{\ell_q}\right)$ is finite and equals to
	
	\begin{align*}
		&\dfrac{1}{bk+\sum_{j=1}^n \lambda_j \ell_j}\left[\left(ak+\dfrac{\sigma_1^2}{2}k(k-1)\right)\mathbb{E}\left(Y_\infty^{k-1}\prod_{q=1}^n(\tilde{X}_\infty^q)^{\ell_q}\right)+\sum_{j=1}^n\Biggl(\ell_j\left(\tilde{m}_j +k\sigma_1\check{\rho}_{j1} \right)\right.\\&\times\mathbb{E}\left(Y_\infty^k(\tilde{X}_\infty^j)^{\ell_j-1}\prod_{\substack{q=1\\q\neq j}}^n(\tilde{X}_\infty^q)^{\ell_q}\right)-\tilde{\kappa}_j \ell_j \mathbb{E}\left(Y_\infty^{k+1} (\tilde{X}_\infty^j)^{\ell_j-1}\prod_{\substack{q=1\\q\neq j}}^n(\tilde{X}_\infty^q)^{\ell_q}\right)+\dfrac{\check{\sigma}_{j+1}^2}{2}\ell_j(\ell_j-1)\\&\left.\times\mathbb{E}\left(Y_\infty^{k+1}(\tilde{X}_\infty^j)^{\ell_j-2}\prod_{\substack{q=1\\q\neq j}}^n(\tilde{X}_\infty^q)^{\ell_q}\right)\left.+\dfrac{1}{2}\sum_{\substack{i=1\\ i\neq j}}^n \check{\rho}_{i,\mathit{H}}\left(\check{\rho}_{j,\mathit{H}}\right)^{\mathsf{T}}\ell_i\ell_j \right.
		\left.\mathbb{E}\left(Y_\infty^{k+1}(\tilde{X}_\infty^i)^{\ell_i-1}(\tilde{X}_\infty^j)^{\ell_j-1}\prod_{\substack{q=1\\q\neq i,j}}^n(\tilde{X}_\infty^q)^{\ell_q}\right)\right)\right],
	\end{align*}
	 with the convention $\mathbb{E}\left(Y_\infty^k\prod_{q=1}^n(\tilde{X}_\infty^q)^{\ell_q}\right):=0$ if one of the powers $k,\ell_1,\ldots,\ell_n$ is negative. 
\end{proposition}
\begin{remarque}
	Under the same assumptions of Proposition \ref{moment result}, for all $1\leq i,j\leq n$ such that $i\neq j$, we get the following special moments results
	$$\mathbb{E}(Y_\infty)=\dfrac{a}{b},\quad \mathbb{E}(Y^2_\infty)=\dfrac{a(2a+\sigma_1^2)}{2b^2},\quad \mathbb{E}(Y^3_\infty)=\dfrac{a(a+\sigma_1^2)(2a+\sigma_1^2)}{2b^2},$$ 
	$$\mathbb{E}(\tilde{X}^i_\infty)=\dfrac{b\tilde{m}_i -a\tilde{\kappa}_i}{b\lambda_i},\quad\mathbb{E}(Y_\infty\tilde{X}^i_\infty)=\dfrac{a\mathbb{E}(\tilde{X}_\infty^i)+(\tilde{m}_i+\sigma_1\check{\rho}_{i1})\mathbb{E}(Y_\infty)-\tilde{\kappa}_i\mathbb{E}(Y_\infty^2)}{b+\lambda_i},$$
	$$\mathbb{E}(Y_\infty^2\tilde{X}^i_\infty)=\dfrac{(2a+\sigma_1^2)\mathbb{E}(Y_\infty\tilde{X}_\infty^i)+(\tilde{m}_i+2\sigma_1\check{\rho}_{i1})\mathbb{E}(Y_\infty^2)-\tilde{\kappa}_i\mathbb{E}(Y_\infty^3)}{2b+\lambda_i},$$
	$$\mathbb{E}((\tilde{X}_\infty^i)^2)=\dfrac{2\tilde{m}_i\mathbb{E}(\tilde{X}_\infty^i)-2\tilde{\kappa}_i\mathbb{E}(Y_\infty\tilde{X}_\infty^i)+\check{\sigma}^2_{i+1}\mathbb{E}(Y_\infty)}{2\lambda_i},$$
	
	$$\mathbb{E}(Y_\infty(\tilde{X}^i_\infty)^2)=\dfrac{a\mathbb{E}((\tilde{X}_\infty^i)^2)+2(\tilde{m}_i+\sigma_1\check{\rho}_{i1})\mathbb{E}(Y_\infty\tilde{X}_\infty^i)-2\tilde{\kappa}_i\mathbb{E}(Y_\infty^2\tilde{X}_\infty^i)+\check{\sigma}_{i+1}^2\mathbb{E}(Y_\infty)}{b+2\lambda_i},$$
	$$\mathbb{E}(\tilde{X}_\infty^i\tilde{X}^j_\infty)=\dfrac{\tilde{m}_i\mathbb{E}(\tilde{X}_\infty^j)+\tilde{m}_j\mathbb{E}(\tilde{X}_\infty^i)-\tilde{\kappa}_i\mathbb{E}(Y_\infty\tilde{X}_\infty^j)-\tilde{\kappa}_j\mathbb{E}(Y_\infty\tilde{X}_\infty^i)+\check{\rho}_{i,\mathit{H}}\left(\check{\rho}_{j,\mathit{H}}\right)^{\mathsf{T}}\mathbb{E}(Y_\infty)}{\lambda_i+\lambda_j},$$
	and	
	\begin{align*}
		\mathbb{E}(Y_\infty\tilde{X}^i_\infty\tilde{X}^j_\infty)=&\dfrac{a\mathbb{E}(\tilde{X}_\infty^i\tilde{X}_\infty^j)+(\tilde{m}_i+\sigma_1\check{\rho}_{i1})\mathbb{E}(Y_\infty\tilde{X}_\infty^j)+(\tilde{m}_j+\sigma_1\check{\rho}_{j1})\mathbb{E}(Y_\infty\tilde{X}_\infty^i)}{b+\lambda_i+\lambda_j}\\
		&-\dfrac{\tilde{\kappa}_i\mathbb{E}(Y_\infty^2\tilde{X}_\infty^j)+\tilde{\kappa}_j\mathbb{E}(Y_\infty^2\tilde{X}_\infty^i)-\check{\rho}_{i,\mathit{H}}\left(\check{\rho}_{j,\mathit{H}}\right)^{\mathsf{T}}\mathbb{E}(Y_\infty^2)}{b+\lambda_i+\lambda_j}.
	\end{align*}
	\label{Remark moment à l'infini}
\end{remarque}
\begin{proof}(Proposition \ref{moment result})
	At first we suppose that all the mixed moments of $(Y_0,X_0)^{\mathsf{T}}$ are finite and $\mathbb{P}(Y_0\in\R_{++})=1$ due the fact that the distribution of $(Y_\infty,X_\infty)^{\mathsf{T}}$ does not depend on the initial value $(Y_0,X_0)^{\mathsf{T}}$. It is easy to check by analogous arguements used in the particular case of \cite[proof of Theorem 5.1]{Boylog1}, that
	\begin{equation}
\displaystyle\int_0^t\mathbb{E}\left(Y_s^{k}\displaystyle\prod_{i=1}^{n}(X_s^i)^{\ell_i}\right)\mathrm{d}s<\infty,\qquad\text{for all } t\in\R_{+} \text{ and } k,\ell_1,\ldots,\ell_n\in \N \setminus\lbrace0\rbrace.
		\label{mix expectation property}
	\end{equation}
	Next, thanks to relation \eqref{Ito for mixed moments}, we deduce that for all $M\in\N\setminus\lbrace0\rbrace$, the functions $t\mapsto f(k,\ell_1,\ell_2,\ldots,\ell_n,t):=\mathbb{E}\left(Y_{t}^k\prod_{q=1}^n(\tilde{X}_{t}^q)^{\ell_q}\right)$, where the indices $k,\ell_1,\ell_2,\ldots,\ell_n\in\N$ satisfy $k+\sum_{i=1}^n\ell_i\leq M$, are solutions of a homogeneous linear system of differential equations and its coefficient matrix is constant with eigenvalues $-(bk+\sum_{i=1}^n \lambda_i\ell_i)$ and $0$. Hence, we deduce that for all $k,\ell_1,\ell_2,\ldots,\ell_n\in\N$, these functions are linear combinations of $\exp(-(bk+\sum_{i=1}^n \lambda_i\ell_i)t)$ and the constant function and, in particular, they are bounded and convergent, as $t\to\infty$. Next, let us recall that $Y_{t}^k\prod_{q=1}^n(\tilde{X}_{t}^q)$ is uniformly integrable and that it converges in distribution, thanks to Theorem \ref{Stationarity theorem} and the continuous mapping theorem, to $Y_{\infty}^k\prod_{q=1}^n(\tilde{X}_{\infty}^q)^{\ell_q}$, as $t\to\infty$. Consequently, by the moment convergence theorem, see, e.g., \cite[Lemma 2.2.1]{Stroock}, we get $$\lim\limits_{t\to\infty}f(k,\ell_1,\ell_2,\ldots,\ell_n,t)=\mathbb{E}\left(Y_{\infty}^k\displaystyle\prod_{q=1}^n(\tilde{X}_{\infty}^q)^{\ell_q}\right)$$
	and such limit is finite. Finally, since, due to Theorem \ref{Stationarity theorem}, the distribution of $(Y_{\infty},X_{\infty})$ does not depend on $(Y_0,X_0)$, we may and do suppose that they have the same distribution, thus, by Theorem \ref{Stationarity theorem}, the process $(Y_t,X_t)_{t\in\R_{+}}$ is strictly stationary. Hence, for all $t\in\R_{+}$ and $k,\ell_1,\ell_2,\ldots,\ell_n\in\N$, we get $$f(k,\ell_1,\ell_2,\ldots,\ell_n,t):=\mathbb{E}\left(Y_{\infty}^k\prod_{q=1}^n(\tilde{X}_{\infty}^q)^{\ell_q}\right)$$ and in particular, we obtain  $f'(k,\ell_1,\ell_2,\ldots,\ell_n,t)=0$. The equation \eqref{Ito for mixed moments} gives the induction result for $\mathbb{E}\left(Y_{\infty}^k\prod_{q=1}^n(\tilde{X}_{\infty}^q)^{\ell_q}\right)$, $k,\ell_1,\ell_2,\ldots,\ell_n\in\N$ which completes the proof. 
\end{proof}

\vspace*{0.5cm}
\section{About the $\mathit{AD}(1,n)$ model and its limit theorems}

In a first part of this appendix, we recall the clasification result of the $\mathit{AD}(1,n)$ model introduced by Ben Alaya, Dahbi and Fathallah in \cite{Dahbi} and its relative stationaity and ergodicity theorems.
\begin{proposition}
	Let $a\in\R_{++}$, $b\in\R$, $m,\kappa\in\R^n$ and let $\theta$ be a real diagonalizable matrix in $\mathcal{M}_n$. Suppose that $\theta$ is either a positive definite, a negative definite or a zero matrix with eigenvalues different to $b$ or all equal to $b$. Let $Z=(Y,X)^\mathsf{T}$ be the unique strong solution of the SDE $\eqref{model01}$ and suppose that $Z_0$ is integrable. Then, for all $t\in\R_{+},$ we have\\
	\begin{enumerate}
		\item \label{b_positif}	for $b\in\R_{++}$,
		$
		\mathbb{E}\left(Y_{t}\right)=
		\frac{a}{b}+\mathbf{O}(e^{-b t})
		$
		and 
		\begin{equation*}
			\mathbb{E}(X_t)=
			\begin{cases}
				\theta^{-1}m-\dfrac{a}{b}\theta^{-1}\kappa+\mathbf{O}(e^{-(\lambda_{\min}(\theta)\wedge b)t})\mathbf{1}_n,&\lambda_{\min}(\theta)\in\R_{++},\\
				t\left( m-\dfrac{a}{b}\kappa\right) +\mathbf{O}(1)\mathbf{1}_n,&\lambda_{\min}(\theta)=\lambda_{\max}(\theta)=0,\\
				e^{-t \theta}\left(\mathbb{E}(X_0)(\theta-b \mathbf{I}_n)^{-1}\kappa+\Xi_0\right)+\mathbf{O}(1)\mathbf{1}_n,&\lambda_{\max}(\theta)\in\R_{--},		
			\end{cases}	
			\label{expect b>0}~~~~~~~~~~~~~~~~~~~~~~~~~~~~~~~~~~~~~~~~~~~~~~~
		\end{equation*}
		\item \label{b_nul} for $b=0$, 
		$
		\mathbb{E}\left(Y_{t}\right)=
		at+\mathbf{O}(1)
		$
		and 
		\begin{equation*}
			\mathbb{E}(X_t)=
			\begin{cases}
				-ta\theta^{-1}\kappa +\mathbf{O}(1)\mathbf{1}_n,&\lambda_{\min}(\theta)\in\R_{++},\\
				-\dfrac{t^2}{2}a\kappa+\mathbf{O}(t)\mathbf{1}_n,&\lambda_{\min}(\theta)=\lambda_{\max}(\theta)=0,\\
				e^{-t \theta}\left(\mathbb{E}(Y_0)\theta^{-1}\kappa+\mathbb{E}(X_0)-\theta^{-1}m-a\theta^{-2}\kappa\right)+\mathbf{O}(t)\mathbf{1}_n,&\lambda_{\max}(\theta)\in\R_{--},		
			\end{cases}	
			\label{expec b=0}~~~~~~~~~~~~~~~~~~~~~~~~~~~~~~~~
		\end{equation*}
		\item \label{b_negatif} for $b\in\R_{--}$,
		$
		\mathbb{E}\left(Y_{t}\right)=
		\left(\mathbb{E}(Y_0)-\dfrac{a}{b}\right)e^{-bt}+\mathbf{O}(1)
		$
		and
		\begin{equation*}
			\mathbb{E}(X_t)=
			\begin{cases}
				e^{-bt}\left(\left(\frac{a}{b}-\mathbb{E}\left(Y_{0}\right)\right)\left(\theta-b \mathbf{I}_{n}\right)^{-1} \kappa\right)+\mathbf{O}(1)\mathbf{1}_n,&\lambda_{\min}(\theta)\in\R_{++},\\
				e^{-bt}\left( \dfrac{1}{b}\mathbb{E}(Y_0)\kappa+\mathbb{E}(X_0)-\dfrac{a}{b^2}\kappa\right) +\mathbf{O}(t)\mathbf{1}_n,&\lambda_{\min}(\theta)=\lambda_{\max}(\theta)=0,\\
				e^{-bt}\left(-\mathbb{E}(Y_0)(\theta-b\mathbf{I}_n)^{-1}\kappa+\dfrac{a}{b}(\theta-b\mathbf{I}_n)^{-1}\kappa \right)+\mathbf{O}(e^{-\lambda_{\min}(\theta) t})\mathbf{1}_n ,&b<\lambda_{\min}(\theta)\leq\lambda_{\max}<0,\\
				te^{-bt}\left( -\mathbb{E}(Y_0)\kappa+\dfrac{a}{b}\kappa\right) +\mathbf{O}(e^{-bt})\mathbf{1}_n ,&\lambda_{\min}(\theta)=\lambda_{\max}(\theta)=b,\\
				e^{-t \theta}\left( \mathbb{E}(Y_0)(\theta-b\mathbf{I}_n)^{-1}\kappa+\Xi_0\right)+\mathbf{O}(e^{-bt})\mathbf{1}_n ,&\lambda_{\max}(\theta)<b,
			\end{cases}	
		\end{equation*}
		where $\Xi_0=\mathbb{E}(X_0) -\theta^{-1}m+\dfrac{a}{b}\theta^{-1}\kappa-\dfrac{a}{b}(\theta-b\mathbf{I}_n)^{-1}\kappa.$
		
	\end{enumerate}
	\label{classification}
\end{proposition}
\begin{definition}
	Let $(Z_t)_{t\in\R_{+}}$ be the unique strong solution of $\eqref{model0 Z}$ satisfying $\mathbb{P}(Z_0^1\in\R_{++})=1$. Suppose that $\theta$ is either a positive definite, a negative definite or a zero matrix with eigenvalues different to $b$ or all equal to $b$. We call $(Z_t)_{t\in\R_+}$ subcritical if $b\wedge \lambda_{\min}(\theta)\in\R_{++}$, i.e., when $\mathbb{E}(Z_t)$ converges as $t\to\infty$, critical if either $b\in\R_+$ and $\lambda_{\min}(\theta)=\lambda_{\max}(\theta)=0$ or $b=0$ and $\lambda_{\min}(\theta)\in\R_{++}$, i.e., when $\mathbb{E}(Z_t)$ has a polynomial expansion and supercritical if $b\wedge\lambda_{\max}(\theta)\in\R_{--}$, i.e., when $\mathbb{E}(Z_t)$ has an exponential expansion.
\end{definition}
\begin{theoreme}
	Let us consider the $\mathit{AD}(1,n)$ model $\eqref{model01}$ with $a\in\R_+,\ b\in\R_{++},\ m\in\R^n,\ \kappa\in\R^n$, $\theta\in\mathcal{M}_n$ a diagonalizable positive definite matrix and $Z_0=(Y_0,X_0)^{\mathsf{T}}$ is a random initial value independent of $(B_t)_{t\in\R_+}$ satisfying $\mathbb{P}(Y_0\in\R_+)=1$.
	\begin{enumerate}[label=\arabic*)]
		\item Then, $Z_t\stackrel{\mathcal{D}}{\longrightarrow}Z_{\infty}$ as $t\to\infty$, where $Z_{\infty}=(Y_\infty,X_\infty)^{\mathsf{T}}$ and the distribution of $Z_\infty$ is given by
		\begin{equation}
			\mathbb{E}\left(e^{\nu^{\mathsf{T}} Z_\infty}\right)=\mathbb{E}\left(e^{-\lambda Y_\infty +i \mu^{\mathsf{T}} X_\infty}\right)=\exp\left(a\displaystyle\int_0^{\infty} \mathcal{K}_s\left(-\lambda,\mu \right)\mathrm{~d}s+i\mu^{\mathsf{T}}\theta^{-1} m\right),
			\label{loi limite Stationarity theorem}
		\end{equation}
		for $\nu=(-\lambda,i\mu)\in\mathcal{U}_1\times \mathcal{U}_2$, where $\mathcal{U}_1:=\lbrace u\in \mathbb{C} : \texttt{Re}(u)\in\R_{-}\rbrace$ and $\mathcal{U}_2:=\lbrace u\in \mathbb{C}^n : \texttt{Re}(u)=0\rbrace$ and for all $(u_1,u_2)\in \mathcal{U}_1\times\R^n$, the function $t\mapsto \mathcal{K}_t(u_1,u_2)$ is the unique solution of the following (deterministic) non-linear Riccati partial differential equation (PDE)
		\begin{equation}
			\begin{cases}
				\dfrac{\partial \mathcal{K}_t}{\partial t}(u_1,u_2)&=
				\dfrac{\rho_{11}^2}{2}\mathcal{K}_t^2(u_1,u_2)-\left(b-i\rho_{11}\rho_{J1}^{\mathsf{T}}e^{-t\theta^{\mathsf{T}}}u_2\right) \mathcal{K}_t(u_1,u_2) -i\kappa^{\mathsf{T}}\,e^{-t\theta^{\mathsf{T}}}u_2 \\&\quad-\frac{1}{2}\left(\text{vec}(\rho_{JJ}\rho_{JJ}^{\mathsf{T}})\right)^\mathsf{T}\,e^{-t (\theta^\mathsf{T}\oplus\theta^\mathsf{T})}(u_2\otimes u_2)-\dfrac{1}{2}\left(\rho_{J1}^{\mathsf{T}}e^{-t\theta^{\mathsf{T}}}u_2\right)^2\\ 
				\mathcal{K}_0(u_1,u_2)&=u_1.
			\end{cases}
			\label{EDP_loi limite Stationarity theorem}
		\end{equation} 
		\label{i Stationarity}
		\item In addition, if $Z_0$ has the same distribution as $Z_{\infty}$ given by $\eqref{loi limite Stationarity theorem}$, then $(Z_t)_{t\in\R_+}$ is strictly stationary.
	\end{enumerate}
	\label{Stationarity theorem}
\end{theoreme}
\begin{theoreme} Let us consider the $\mathit{AD}(1,n)$ model $\eqref{model01}$ with $a\in\R_{++}$, $b\in\R_{++}$, $m\in\R^n$, $\kappa\in\R^n$ and $\theta\in\mathcal{M}_n$ a diagonalizable positive definite matrix with initial random values $Z_0=(Y_0,X_0)^{\mathsf{T}}$ independent of $(B_t)_{t\in\R_{+}}$ satisfying $\mathbb{P}(Y_0\in\R_{++})=1$. Then the process $Z$ is exponentially ergodic, namely, there exists $\delta\in\R_{++},\ B\in\R_{++}$ and $r\in\R_{++}$ such that
	\begin{equation}
		\underset{\vert g\vert\leq V+1}{\sup}\left\vert \mathbb{E}\left(g(Z_t)\vert Z_0=z_0\right)-\mathbb{E}(g(Z_{\infty}))\right\vert\leq B(V(z_0)+1))e^{-\delta t},
		\label{ergodicity1}
	\end{equation}
	for all $t\in\R_{+}$ and $z_0=(y_0,x_0)^{\mathsf{T}}\in\mathcal{D}$, where the supremum is running for Borel measurable functions $g:\mathcal{D}\to\R$, $V(y,x):=y^2+r\norm{x}_2^2$, for all $(y,x)\in\mathcal{D}$,
	and $Z_\infty=(Y_\infty,X_\infty)^{\mathsf{T}}$ is defined by $\eqref{loi limite Stationarity theorem}$. Moreover, for all Borel measurable functions $f:\R\times\R^n \to \R$ such that $\mathbb{E}\left(\vert f(Z_\infty) \vert\right)<\infty$, we have
	\begin{equation}
		\mathbb{P}\left( \underset{T\to\infty}{\lim}\dfrac{1}{T}\displaystyle\int_0^Tf(Z_s)\mathrm{~d}s=\mathbb{E}\left(f(Z_\infty)\right)\right)=1.
		\label{ergodic}
	\end{equation}
	\label{ergodicity theorem}
\end{theoreme}
In what follows we recall some limit theorems for continuous local martingales. We use these limit theorems for studying the asymptotic behavior of the MLE for $\t.$ First we recall a strong law of large numbers for continuous local martingales see Liptser and Shiryaev 

\begin{theoreme}(Liptser and Shiryaev (2001)) Let $(\Omega,\mathcal{F},(\mathcal{F})_{t\in \real_{+}},\mathbb{P})$ be a filtered probability space satisfying the usual conditions. Let  $(M_{t})_{t \in \real_{+}}$ be a square-integrable continuous local martingale with respect to the filtration $(\mathcal{F})_{t\in \real_{+}}$ such that $\mathbb{P} \left( M_0 = 0\right)=1.$ Let $(\xi_t)_{t\in \real_{+}}$ be a progressively measurable process such that
	$$\mathbb{P}\left( \displaystyle\int_{0}^{t}  \xi_{u}^{2}\mathrm{~d} \left\langle M\right\rangle_u < \infty\right)=1, \quad t \in \real_{+},$$
	and
	\begin{align*}
		\displaystyle\int_{0}^{t}  \xi_{u}^{2} \mathrm{~d} \left\langle M\right\rangle_u \stackrel{a.s.}{\longrightarrow} \infty,\quad as\;\; t\to \infty,
	\end{align*}
	where $\left( \left\langle M\right\rangle_t  \right)_{t \in \real_{+}} $ denotes the quadratic variation process of $M.$ Then
	\begin{align*}
		\dfrac{\displaystyle\int_{0}^{t}  \xi_{u} \mathrm{~d} M_u}{\displaystyle\int_{0}^{t}  \xi_{u}^{2} \mathrm{~d} \left\langle M\right\rangle_u}
		\stackrel{a.s.}{\longrightarrow} 0,\quad as\;\; t\to\infty.
	\end{align*}
	If $(M_{t})_{t \in \real_{+}}$ is a standard Wiener process, the progressive measurability of  $(\xi_t)_{t\in \real_{+}}$ can be relaxed to measurability and adaptedness to the filtration $(\mathcal{F})_{t\in \real_{+}}.$
	\label{LFGN}
\end{theoreme}

The next theorem is about the asymptotic behaviour of continuous multivariate local martingales, see \cite[Theorem 4.1]{Zanten}.
\begin{theoreme}
	For $p\in\N \setminus \lbrace 0\rbrace$, let $M=(M_t)_{t\in\R_+}$ be a $p$-dimensional square-integrable continuous
	local martingale with respect to the filtration $(\mathcal{F}_t)_{t\in\R_+}$ such that $\mathbb{P}(M_0=0)=1$. Suppose that
	there exists a function $Q:\left[t_0,\infty\right) \to \mathcal{M}_p$ with some $t_0\in\R_+$ such that $Q(t)$ is an invertible
	(non-random) matrix for all $t\in\R_+$, $\underset{t\to\infty}{\lim} \norm{Q(t)}=0$ and
	$$	Q(t) \langle M\rangle_t Q(t)^{\mathsf{T}}\stackrel{\mathbb{P}}{\longrightarrow}\eta\eta^{\mathsf{T}},\quad \text{as }t\to \infty,$$
	where $\eta$ is a random matrix in $\mathcal{M}_p$. Then, for each random matrix $A\in\mathcal{M}_{k,l}$,  $k,l\in\N\setminus\lbrace 0\rbrace$, defined on $(\Omega,\mathcal{F},\mathcal{P})$, we have
	$$
	(Q(t)M_t,A) \stackrel{\mathcal{D}}{\longrightarrow}
	(\eta Z,A),\quad \text{as }t\to \infty,$$
	where $Z$ is a $p$-dimensional standard normally distributed random vector independent of $(\eta,A)$.
	\label{CLT Van Zanten}
\end{theoreme}
\begin{lemme}\label{convergence (Xn,Yn) to (X,X) in law}
Let $X_n$ and $Y_n$ two random variables such that $X_n\stackrel{\mathcal{D}}{\longrightarrow}X$ and $X_n-Y_n\stackrel{\mathcal{D}}{\longrightarrow}0$, as $N\to\infty$, then $(X_n,Y_n)\stackrel{\mathcal{D}}{\longrightarrow}(X,X)$, as $N\to\infty$. 
\end{lemme}
\bibliographystyle{plain}

\end{document}